\documentclass[a4paper,english,11pt]{amsart}

\usepackage[utf8]{inputenc}
\usepackage[T1]{fontenc}
\usepackage{babel}
\usepackage[babel]{csquotes}

\usepackage{amsmath}
\usepackage{amsfonts}
\usepackage{amsthm}
\usepackage{bbm}
\usepackage{amssymb}
\usepackage[top=3.4cm, bottom=3.6cm, left=3.6 cm, right=3.6 cm]{geometry}
\usepackage{mathrsfs}
\usepackage{hyperref}
\hypersetup{colorlinks=true, linkcolor=blue, urlcolor=blue, citecolor=blue}
\usepackage{stmaryrd}
\usepackage{cleveref}
\usepackage{mathtools}
\usepackage{soul}
\usepackage{bm}
\usepackage{verbatim}
\usepackage{xcolor}
\usepackage{enumitem}

\newtheorem{th.}{Theorem}[section]
\newtheorem{thm}[th.]{Theorem}

\newtheorem{lemma}[th.]{Lemma} 
 
\newtheorem{proposition}[th.]{Proposition}
\newtheorem{corollary}[th.]{Corollary}

\newtheorem{thmx}{Theorem}

\theoremstyle{definition}
\newtheorem{definition}[th.]{Definition}

\numberwithin{equation}{section}

\newcommand{\Q}{\mathbb{Q}}
\newcommand{\N}{\mathbb{N}}
\newcommand{\Z}{\mathbb{Z}}
\newcommand{\C}{\mathbb{C}}
\newcommand{\R}{\mathbb{R}}

\newcommand{\G}{\Gamma}

\newcommand{\kt}{\mathfrak{t}}
\newcommand{\kg}{\mathfrak{g}}

\newcommand{\kp}{\mathfrak{p}}

\newcommand{\ku}{\mathfrak{u}}
\newcommand{\ks}{\mathfrak{s}}
\newcommand{\ka}{\mathfrak{a}}

\newcommand {\cB} {{\mathcal B}}

\newcommand {\cD} {{\mathcal D}}
\newcommand {\cE} {{\mathcal E}}
\newcommand {\cF} {{\mathcal F}}

\newcommand {\cI} {{\mathcal I}}

\newcommand {\cN} {{\mathcal N}}

\newcommand {\cP} {{\mathcal P}}

\newcommand {\cR} {{\mathcal R}}
\newcommand {\cS} {{\mathcal S}}

\newcommand {\cZ} {{\mathcal Z}}

\DeclareMathOperator{\supp}{supp}

\DeclareMathOperator{\SL}{SL}
\DeclareMathOperator{\PSL}{PSL}

\DeclareMathOperator{\GL}{GL}

\DeclareMathOperator{\SO}{SO}

\DeclareMathOperator{\Lie}{Lie}

\DeclareMathOperator{\Ad}{Ad}

\DeclareMathOperator{\diag}{diag}

\newcommand{\bA}{{\bf A}}

\newcommand{\bD}{{\bf D}}
\newcommand{\bG}{{\bf G}}
\newcommand{\bH}{{\bf H}}
\newcommand{\bL}{{\bf L}}
\newcommand{\bM}{{\bf M}}
\newcommand{\bS}{{\bf S}}
\newcommand{\bT}{{\bf T}}
\newcommand{\bU}{{\bf U}}

\newcommand{\bQ}{{\bf Q}}
\newcommand{\bP}{{\bf P}}
\newcommand{\bV}{{\bf V}}
\newcommand{\bX}{{\bf X}}

\newcommand{\dd}{\,\mathrm{d}}

\DeclareMathOperator{\vol}{vol}

\title[Integrability of Siegel transforms and an application]{Integrability of Siegel transforms and an application}

\author{Ren\'e Pfitscher}

\address{Ren\'e Pfitscher. School of Mathematical Sciences, University of Science and Technology of China}

\email{pfitscher@ustc.edu.cn}

\thanks{School of Mathematical Sciences, University of Science and Technology of China}



\begin{document}

\begin{abstract}
We establish sharp algebraic criteria for the $L^p$-integrability, for $p = 1, 2, \infty$, of a natural generalization of the Siegel transform to the setting of rational representations of semisimple algebraic $\Q$-groups, extending Siegel’s analytic work in the geometry of numbers. 

As an application, we derive an effective asymptotic formula for the number of rational approximations of bounded height to almost every real point on a rank-one flag variety at the Diophantine exponent. The argument combines the integrability criterion with effective equidistribution estimates for translated orbits of maximal compact subgroups, a result of independent interest.  
\end{abstract}

\maketitle
\setcounter{tocdepth}{1}
\tableofcontents

\section{Introduction}
The Siegel transform, introduced in 1945 by Siegel \cite{Siegel45}, maps a function of sufficient decay on the Euclidean space $\mathbb{R}^n$ to a function on the moduli space of unimodular lattices $\Omega = \mathrm{SL}_n(\mathbb{R}) / \mathrm{SL}_n(\mathbb{Z})$. Let $\cP(\Z^n)$ denote the set of primitive elements of $\Z^n$ and let $B_c^{\infty}(\R^n)$ be the space of Borel measurable bounded compactly supported functions $f : \R^n \rightarrow \C$. Then, for every $f \in B_c^{\infty}(\R^n)$, the primitive Siegel transform $S f : \Omega \rightarrow \C$ of $f$ is defined by
\begin{equation} \label{eq:primitive_Siegel_Transform}
\forall \, g \in \mathrm{SL}_n(\mathbb{R}), \qquad S f(g\Z^n) = \sum_{\bm{v} \in \cP(\Z^n)} f (g \bm{v}).
\end{equation}
Let $\mu_{\Omega}$ be the unique $\SL_{n}(\R)$-invariant probability measure on $\Omega$, let $\zeta$ be the Riemann zeta function and let $\lambda_{\R^n}$ be the usual Lebesgue measure on $\R^n$. Siegel's mean value formula \cite{Siegel45} expresses the average of $Sf$ in terms of the average of $f$: 
\begin{equation} \label{eq:primitive_Siegel_Transform_formula}
\int_{\Omega} S f \, \dd \mu_{\Omega} = \frac{1}{\zeta(n)} \int_{\R^n} f \, \dd \lambda_{\R^n}.
\end{equation}
Later, extending Siegel’s result, Rogers \cite{Rogers55} proved a $k$-th moment formula for the Siegel transform for $k$ up to $n-1$. A remarkable application of the second moment formula to the geometry of numbers was given by Schmidt \cite{Schmidt60b}, who derived an asymptotic formula for counting lattice points in an expanding family of sets in $\R^n$ from the variance bound
\begin{equation} \label{eq:primitive_Siegel_Transform_variance}
\int_{\Omega} \left| S f - \frac{1}{\zeta(n)} \int_{\R^n} f \, \dd \lambda_{\R^n} \right|^2 \dd \mu_{\Omega} \, \ll \, \int_{\R^n} |f|^2 \, \dd \lambda_{\R^n}.
\end{equation}
In other words, the \emph{centered} Siegel transform $\overline{S} f = Sf - 1/\zeta(n) \int_{\R^n} f \, \dd \lambda_{\R^n}$ extends to a bounded linear operator 
\[
\overline{S} : L^2(\R^n) \rightarrow L^2(\Omega). 
\]
The variance bound \eqref{eq:primitive_Siegel_Transform_variance} also yields an alternative proof of Schmidt’s strengthening \cite{Schmidt60a} of Khintchine’s theorem \cite{Khintchine26} in metric Diophantine approximation on $\R^n$ and its projective counterpart $\mathbb{P}(\R^n)$. There has been an active line of research extending classical results in Diophantine approximation from Euclidean space to other varieties, such as spheres \cite{AG22, KM15, KY23, Ouaggag23}, projective quadrics \cite{SK18, deSaxce22b, FKMS22}, Grassmannians \cite{deSaxce22a}, and more general flag varieties \cite{deSaxce20}. 

The purpose of this paper is to study fundamental integrability properties of a natural extension of the Siegel transform \eqref{eq:primitive_Siegel_Transform} from the Euclidean space to the setting of generalized flag varieties. Our results have an application to metric Diophantine approximation on rank-one flag varieties. The proof of this application relies in addition on the effective single and double equidistribution property for expanding orbits of maximal compact subgroups, a result of independent interest. 

\subsection{Main results}
Let $\bG$ be a connected simply-connected almost $\Q$-simple $\Q$-group and let $\bP$ be a proper parabolic $\Q$-subgroup of $\bG$. We denote algebraic varieties defined over $\mathbb{Q}$ by bold letters and their sets of real points by ordinary letters. For instance, we write $G = \bG(\R)$ to denote the group of real points of $\bG$. Let $\G \subset \bG(\Q)$ be an arithmetic subgroup of $G$. Let $\pi_{\chi} : \bG \rightarrow \GL(\bV_{\chi})$ be an irreducible representation defined over $\Q$ which is generated by a line $\bD_{\chi}$ defined over $\Q$ of highest weight $\chi$ such that $\bP = \mathrm{Stab}_{\bG} (\bD_{\chi})$ (Section~\ref{sec:Reps}). In particular, the space of real points $X = \bX(\R)$ of the generalized flag variety $\bX = \bG / \bP$ embeds into the projective space $\mathbb{P}(V_{\chi})$. We fix a highest weight vector $\bm{e}_{\chi} \in \bD_{\chi}(\Q)$ and define $\widetilde{X}$ to be the orbit $\widetilde{X} = G \, \bm{e}_{\chi} \subset V_{\chi}$. We refer to $\widetilde{X}$ as the \emph{cone over $X$ relative to $\chi$}. Fix a $\G$-stable lattice $\bV_{\chi}(\Z) \subset \bV_{\chi}(\Q)$ of $V_{\chi}$ and denote by $\cP_{\chi}$ the set of primitive elements of $\bV_{\chi}(\Z) \cap \widetilde{X}$. Let $B_c^{\infty}(\widetilde{X})$ be the space of Borel measurable bounded compactly supported complex-valued functions $f : \widetilde{X} \rightarrow \C$.
\begin{definition} [Siegel transform]
For every $f \in B_c^{\infty}(\widetilde{X})$, we define the \emph{Siegel transform} $S_{\chi} f : \Omega \rightarrow \C$ of $f$ by
\begin{equation*} \label{def:Siegel-Transform}
\forall \, g \in G, \quad S_{\chi} f(g \G) = \sum_{\bm{v} \in \cP_{\chi}} f(g \bm{v}). 
\end{equation*} 
\end{definition}

Let $\mu_{\Omega}$ be the unique $G$-invariant Borel probability measure on the homogeneous space $\Omega = G/\G$. We will answer the question: \emph{For any $p = 1,2, \infty$, what are necessary and sufficient conditions for $S_{\chi}$ to map $B_c^{\infty}(\widetilde{X})$ into $L^p(\Omega)$?} 

In our first result, the equivalences $(1)$ - $(4)$ are likely known to experts; the formula \eqref{eq:WeilIntegrationFormula*} below is a consequence of a general integration formula due to Weil \cite[Theorem~2.51]{Folland15}. Let $\bL = \mathrm{Stab}_{\bG} (\bm{e}_{\chi}) \subset \bP$. 

\begin{thmx} [$L^1$-integrability] \label{thm:L1} 
The following assertions are equivalent.
\begin{enumerate}[label=(\arabic*)]
\item The Siegel transform $S_\chi$ maps $B_c^{\infty}(\widetilde{X})$ into $L^1(\Omega)$.
\item There exists a unique (up to scaling) $G$-invariant Radon measure $\lambda_{\widetilde{X}}$ on $\widetilde{X}$, the Siegel transform $S_\chi$ extends to a bounded operator $S_{\chi} : L^1(\widetilde{X}) \rightarrow L^1(\Omega)$, and $\lambda_{\widetilde{X}}$ can be normalized so that we have a convergent mean value formula:
\begin{equation} \label{eq:WeilIntegrationFormula*}
\forall \, f \in L^1(\widetilde{X}), \quad \int_{\Omega} S_\chi f \, \dd \mu_{\Omega} = \int_{\widetilde{X}} f \, \dd \lambda_{\widetilde{X}}.
\end{equation}
\item The Lie group $L = \bL(\R)$ is unimodular and $\G_L = \G \cap L$ is a lattice in $L$.
\item The parabolic $\Q$-subgroup $\bP$ of $\bG$ is maximal.
\item There exists $\varepsilon > 0$ such that $S_\chi$ maps $B_{c}^{\infty}(\widetilde{X})$ into $L^{1+\varepsilon}(\Omega)$.
\end{enumerate}
\end{thmx}

\begin{thmx} [$L^{\infty}$-integrability] \label{thm:Linfty} 
The following assertions are equivalent.
\begin{enumerate} [label=(\arabic*)]
\item The Siegel transform $S_\chi$ maps $B_c^{\infty}(\widetilde{X})$ into $L^{\infty}(\Omega)$.
\item The $\mathbb{Q}$-rank of $\mathbf{G}$ is $1$.
\item The discrete group $\G_L$ is a cocompact lattice in $L$.
\end{enumerate}
\end{thmx}

Let $\bP_0$ be a minimal parabolic $\Q$-subgroup of $\bG$ contained in $\bP$ and let $\bT$ be a maximal $\Q$-split torus of $\bG$ contained in $\bP_0$. Let $\Phi$ be the root system of $\bG$ relative to $\bT$ and let $\Delta \subset \Phi$ be the corresponding set of simple roots. For each subset $\theta$ of $\Delta$, write $\bP_{\theta}$ for the associated standard parabolic $\Q$-subgroup of $\bG$. Let $W$ be the Weyl group of $\bG$ relative to $\bT$. For every $w \in W$, we define $\bL_w = \bL \cap x_w \bL x_w^{-1}$, where $x_w \in \cN_{\bG}(\bT)(\Q)$ is a representative of $w$, and denote by $X^*(\bL_w^{\circ})_{\Q}$ the group of $\Q$-characters of its identity component. 

\begin{thmx} [$L^2$-integrability] \label{thm:L2} 
Suppose that the Siegel transform $S_\chi$ maps $B_c^{\infty}(\widetilde{X})$ into $L^{2}(\Omega)$. Then 
\[
\forall \, w \in W, \quad X^*(\bL_w^\circ)_{\Q} = \{1\}.
\]
In particular, the parabolic $\Q$-subgroup $\bP$ is maximal and the simple root $\alpha \in \Delta$, satisfying that $\bP = \bP_{\Delta \smallsetminus \{\alpha \}}$, has at most one neighbor in the associated Dynkin diagram.
\end{thmx}

We were unable to determine whether the converse statement is true: \emph{Assuming that for every $w \in W$, we have $X^*(\bL_w^\circ)_{\Q} = \{1\}$, does the Siegel transform $S_\chi$ map $B_c^{\infty}(\widetilde{X})$ into $L^2(\Omega)$?} 

For every $f \in B_c^{\infty}(\widetilde{X})$, define the \emph{centered Siegel transform} $\overline{S}_{\chi} f : \Omega \rightarrow \C$ of $f$ by
\[
\forall \, g \in G, \qquad \overline{S}_{\chi} f(g\G) = S_{\chi}f(g\G) - \int_{\widetilde{X}} f \, \dd \lambda_{\widetilde{X}}.
\] 
Beyond the case $p=q=1$ in Theorem~\ref{thm:L1}, it is natural to ask for which pairs $p,q\in[1,+\infty]$ the Siegel transform, or its centered counterpart, extends to a bounded linear operator $L^p(\widetilde{X}) \to L^q(\Omega)$. More specifically, we would like to include the following question, due to Saxc\'e, suggesting a fractional version of the variance bound \eqref{eq:primitive_Siegel_Transform_variance} that also takes into account point~(5) of Theorem~\ref{thm:L1} as well as Theorem~\ref{thm:L2}:
\emph{Assuming that the parabolic $\Q$-subgroup $\bP$ is maximal, does there exist $\varepsilon>0$ such that the centered Siegel transform $\overline{S}_{\chi}$ extends to a bounded linear operator}
\[
\overline{S}_{\chi} : L^{1+\varepsilon}(\widetilde{X}) \to L^{1+\varepsilon}(\Omega)\,?
\]

\subsection{Effective equidistribution of maximal compact subgroup orbits}
The fact that the Siegel transform $S_{\chi}$ maps $B_c^{\infty}(\widetilde{X})$ into $L^{1+\varepsilon}(\Omega)$ for some small $\varepsilon > 0$ when the parabolic subgroup $\bP$ is maximal (Theorem~\ref{thm:L1}), together with the effective single and double equidistribution property for translated orbits of maximal compact subgroups (Theorem~\ref{thm:Effective}), are the key analytic inputs for our Schmidt-type counting theorem at the Diophantine exponent on rank-one flag varieties. Before describing this application, let us state here our equidistribution result, which will be derived from an effective multiple equidistribution result for expanding translates of horospherical orbits due to Shi (see \cite[Theorem~1.5]{Shi21}). Let $K \subset G$ be a maximal compact subgroup, equipped with the Haar probability measure $\mu_K$. We write $\cS_r$ ($r \in \N^*)$ for the degree $r$ Sobolev norms on $C_c^{\infty}(\Omega)$ and $C^{\infty}(K)$ as defined in \eqref{eq:Norm}.

\begin{thmx} [Effective equidistribution of maximal compact subgroup orbits] \label{thm:Effective}
Suppose the action of $G$ on $\Omega = G/\G$ has a spectral gap. Let $A = \{a(y) : y \in \R_+^{\times} \}$ be an $\Ad$-diagonalizable one-parameter subgroup of $G$ such that $a(e)$ projects non-trivially to each simple factor of $G$. Then there exist a constant $c > 0$ and an integer $r \geq 1$ such that for every compact subset $Q \subset \Omega$, for all $f \in C^{\infty}(K)$, $\phi \in C_c^\infty(\Omega)$, $x \in Q$, and $y \geq 1$, we have
\begin{equation} \label{eq:Single-Eq}
\int_K f(k) \phi \bigl(a(y) k x \bigr) \dd \mu_K(k) = \int_K f \dd \mu_K \int_{\Omega} \phi \dd \mu_{\Omega} + O \big ( y^{-c} \cS_r(f) \cS_r(\phi) \big ),
\end{equation}
and, for all $f \in C^{\infty}(K)$, $\phi_1, \phi_2 \in C_c^\infty(\Omega)$, $x_1, x_2 \in Q$, 
and $y_2 \geq y_1 \geq 1$, we have
\begin{multline} \label{eq:Double-Eq}
\int_K f(k) \phi_{1}\bigl(a(y_1) k x_1 \bigr) \phi_{2} \bigl(a(y_2) k x_2 \bigr) d \mu_K(k) =  \int_K f \dd \mu_K \int_{\Omega} \phi_1 \dd \mu_{\Omega}  \int_{\Omega} \phi_2 \dd \mu_{\Omega} \\ + O \big ( \min \{y_1, y_2/y_1 \}^{-c} \cS_r(f) \cS_r(\phi_1) \cS_r(\phi_2) \big ),
\end{multline} 
where the implicit constant in $O(\cdot)$ depends only on $Q$.
\end{thmx}

\subsection{Application to Diophantine approximation on flag varieties}
Let us now state our Schmidt-type counting theorem at the Diophantine exponent on rank-one flag varieties, which uses Theorems~\ref{thm:L1} and \ref{thm:Effective} as inputs. Many classical results in Diophantine approximation on the real line $\R$ or in Euclidean space $\R^n$ admit a dynamical reinterpretation in terms of properties of certain diagonal orbits in the space of lattices $\Omega = \SL_n(\R) / \SL_n(\Z)$; this is known as Dani's correspondence \cite{Dani85}. Via this dynamical reinterpretation and building on influential work of Margulis, Kleinbock and others \cite{KM96, KM98, KM99, FKMS22}, Saxc\'e \cite{deSaxce20} extended analogues of classical results to generalized flag varieties $\bX = \bG / \bP$ defined over $\Q$. First examples of such varieties include projective $n$-space $\mathbb{P}^n(\R)$, the Grassmann variety $\mathrm{Gr}_{\ell,n}(\R)$ of $\ell$-dimensional subspaces of $\R^n$, projective quadric hypersurfaces (that is, the solution set in $\mathbb{P}^n(\R)$ of a non-degenerate rational quadratic form in $n+1$ variables), and more general flag varieties, parametrizing flags of subspaces of a Euclidean space. 

Let $\psi : \N \rightarrow (0,+\infty)$ be a non-increasing function. By Khintchine's theorem~\cite{Khintchine26}, the inequality
\[
0 \leq q x - p < \psi(q)
\]
has infinitely (resp. at most finitely) many solutions $(p,q) \in \Z \times \N$ for almost every $x \in \R$, if the series $\sum_{q = 1}^{\infty} \psi(q)$ diverges (resp. converges). In the divergence case, Schmidt~\cite{Schmidt60a} strengthened Khintchine's theorem. More precisely, for every $x \in \R$ and $T \geq 1$, he considered the counting function
\begin{equation} \label{eq:SchmidtCountingFunction}
\cN_\psi(x,T) = \# \left \{ (p,q) \in \Z \times \N : 0 \leq q x - p < \psi(q), \, 1 \leq q < T \right \}
\end{equation}
and showed that for almost every $x \in \R$, $\cN_\psi(x,T)$ is asymptotically equal to $\sum_{1\leq q < T} \psi(q)$ as $T$ goes to infinity, with an explicit error term. In fact, Schmidt's result holds not only for the real line, but also for the Euclidean space $\R^n$ of any dimension $n \geq 1$. 

Our goal is to prove a version of this theorem, where the Euclidean space $\R^n$ is replaced by the space of real points $X = \bX(\R)$ of the generalized flag variety $\bX = \bG / \bP$ defined over $\Q$. We assume that $\bP$ is a maximal parabolic $\Q$-subgroup of $\bG$ with abelian unipotent radical. In particular, $\bX$ has $\Q$-rank $1$ and there exists a unique simple root $\alpha \in \Delta$ such that $\bP = \bP_{\Delta \smallsetminus \{\alpha\}}$. Let $Y$ be the unique element in the Lie algebra of $\bT(\R)$ such that 
\[
\alpha(Y) = -1 \quad \text{and} \quad \beta(Y) = 0 \quad \text{for all } \beta \in \Delta \smallsetminus \{\alpha\}.
\]
We suppose that the element $\exp(Y)$ projects non-trivially to each simple factor of $G$. Let $K$ be a maximal compact subgroup of $G$. Let $\sigma_X$ be the unique $K$-invariant probability measure on $X$. We equip $X$ with a $K$-invariant Riemannian distance $d(\cdot, \cdot)$ and the set of rational points $\bX(\Q)$ with a height function $H_\chi$ associated to an irreducible rational representation $\pi_{\chi} : \bG \rightarrow \GL(\bV_{\chi})$ which is generated by a unique rational line $\bD_{\chi}$ of highest weight $\chi$ such that $\mathrm{Stab}_{\bG} (\bD_{\chi}) = \bP$ (see Section \ref{sec:Reps}). By \cite[Th\'eor\`emes 2.4.5 et 3.2.1]{deSaxce20}, there exists a rational number $\beta_\chi \in \Q_{>0}$ such that, for every $c > 0$ and for $\sigma_X$-almost every $x \in X$, the inequality
\begin{equation} \label{eq:DiophantineExpo}
d(x,v) < c \, H_\chi (v)^{-\tau}
\end{equation}
admits infinitely (resp. at most finitely) many solutions $v \in \bX(\Q)$, if $\tau \leq \beta_\chi$ (resp. $\tau > \beta_\chi$). We refer to $\beta_\chi$ as the \emph{Diophantine exponent} of $X$ relative to $\chi$ and to $\tau \in [0,\beta_{\chi}]$ as an \emph{approximation exponent}.

In analogy to \eqref{eq:SchmidtCountingFunction}, for every constant $c > 0$, approximation exponent $\tau \in [0, \beta_{\chi}]$, element $x \in X$, and parameter $T \geq 1$, we define 
\[
\cN_{c,\tau}(x,T) = \# \left \{ v \in \bX(\Q) : d(x,v) < c \, H_{\chi}(v)^{-\tau}, \, 1 \leq H_{\chi}(v) < T \right \}.
\]
In \cite{Pfitscher24}, we provided an almost-sure asymptotic formula for $\cN_{c,\tau}(x,T)$ as $T \rightarrow +\infty$, with an explicit error term in the case where $\tau \in [0, \beta_{\chi})$. Our method did not yield an effective estimate when counting \emph{at the Diophantine exponent}, that is, when $\tau = \beta_{\chi}$. In our application, we upgrade our previous result to an effective asymptotic estimate. Our approach is inspired by a recent effective counting result due to Ouaggag \cite[Theorem~1.2]{Ouaggag23} for spheres, and our result may be viewed as a substantial generalization thereof. 

\begin{thmx}[Effective counting at the Diophantine exponent] \label{thm:Critical}
Let $d = \dim X$ be the dimension of $X$ and let $c > 0$. Then there exists an explicit constant $\varkappa > 0$ and $\varepsilon > 0$ such that for $\sigma_X$-almost every $x \in X$, as $T \rightarrow + \infty$,
\begin{equation} \label{eq:Thm_Critical}
\mathcal{N}_{c, \beta_{\chi}}(x, T) = \varkappa \, c^d \, \ln (T) \left ( 1 + O_x(\ln(T)^{- \varepsilon}) \right ).
\end{equation}
\end{thmx}

\subsection{Proof sketch of Theorem \ref{thm:Critical}}
Let us illustrate Theorem~\ref{thm:Critical} and its proof in the special case of the Grassmann variety $X = \mathrm{Gr}_{\ell, n}(\R)$ ($\ell, n \in \N, \, 1 \leq \ell < n$), parametrizing $\ell$-dimensional subspaces of the Euclidean space $\R^n$; this theorem also applies to projective quadric hypersurfaces and we refer the reader to \cite[Sections $1.2$ and $8$]{Pfitscher24} and the references therein. The argument involves introducing the Siegel transform in this specific setting, studying its analytic properties, and establishing equidistribution of expanding translates of orbits of maximal compact subgroups. 

Let $\bG = \SL_n$, let $\bT \leq \bG$ be the maximal $\Q$-split $\Q$-torus given by the subgroup of $\bG$ consisting of all diagonal matrices, and let $\bP_0$ be the Borel subgroup of $\bG$ consisting of all upper-triangular matrices. Let $\Phi = \Phi(\bG, \bT)$ be the associated root system with ordering induced by $\bP_0$, $\Delta = \{\alpha_1, \dots, \alpha_{n-1}\}$ the set of simple roots, and $\{\lambda_{1}, \dots, \lambda_{n-1} \}$ the set of fundamental $\Q$-weights. Fix $\alpha_{\ell} \in \Delta$ and let $\chi = \lambda_{\ell}$ be the associated fundamental $\Q$-weight. Recall that for all $a = \diag(a_1, \dots, a_n) \in \bT$, we have $\chi(a) = a_1 \cdots a_{\ell}$. Let $\bP = \bP_{\Delta \smallsetminus \{\alpha_{\ell}\}}$ be the corresponding standard parabolic $\Q$-subgroup. Then $\bP$ is the stabilizer in $\bG$ of the rational line spanned by the pure tensor $\bm{e}_{\chi} = \bm{e}_1 \wedge \dots \wedge \bm{e}_{\ell}$ in the $\ell$-th exterior power of the standard representation of $\bG$. The Siegel transform in this case is defined as follows. Let $\widetilde{X} = G \, \bm{e}_{\chi} \subset \bigwedge^\ell \R^n$: this is the set of all non-zero pure tensors of $\bigwedge^\ell \R^n$. Let $\bL = \mathrm{Stab}_{\bG} ( \bm{e}_{\chi})$, $\G = \SL_n(\Z)$ and let $\cP_{\chi}$ be the set of all primitive elements of $\bigwedge^\ell \Z^n$ that are contained in $\widetilde{X}$. The group $\G$ acts transitively on $\cP_{\chi}$: $\cP_{\chi} = \G \, \bm{e}_{\chi} \cong \G/\G_L$. Therefore, for every $f \in B_c^{\infty}(\widetilde{X})$, the Siegel transform $S_{\chi} f : G / \G \rightarrow \C$ is given by
\[
\forall \, g \in G, \qquad S_{\chi} f(g \G) = \sum_{\bm{v} \in \cP_{\chi}} f(g \bm{v}) = \sum_{\gamma \in \G / \G_L} f(g \gamma \bm{e}_{\chi}).
\]
Then $X = G/P$, viewed as a subvariety of $\mathbb{P}(\bigwedge^\ell \R^n)$ via the embedding $g P \mapsto g [\bm{e}_{\chi}]$ (here $[\bm{e}_{\chi}]$ denotes the projectivization of $\bm{e}_{\chi}$), is the Grassmann variety $\mathrm{Gr}_{\ell,n}(\R)$ of $\ell$-dimensional subspaces of $\R^n$. This is in accordance with Schmidt's paper \cite{Schmidt67}, where he used the Pl\"ucker embedding to define the height $H(v)$ of a rational subspace $v$ of $\R^n$: for $v \in \mathrm{Gr}_{\ell,n}(\Q)$ pick $\bm{v} \in \cP_{\chi}$ with $v = [\bm{v}]$ and set $H(v) = \|\bm{v}\|$, where $\|\cdot \|$ denotes the $\SO_n(\R)$-invariant norm on $\bigwedge^{\ell} \R^n$ induced from the standard Euclidean norm on $\R^n$. The distance used on $X$ is the usual Riemannian distance and we equip $X$ with the unique probability measure $\sigma_X$ invariant under the action of the maximal compact subgroup $K = \SO_n(\R) \leq G$. We study the approximation of a real subspace chosen randomly according to $\sigma_X$ by rational subspaces. Write $d$ for the dimension $\dim_\R X = \ell(n-\ell)$. The Diophantine exponent of $X = \mathrm{Gr}_{\ell,n}(\R)$ with respect to $\chi = \lambda_{\ell}$ is given by $\beta_{\chi} = \frac{n}{\ell(n-\ell)}$ (see \cite[Th\'eor\`eme~1]{deSaxce22a}). We wish to determine the asymptotic behavior of the counting function
\begin{equation} \label{eq:CountingFunction}
\cN_{c,\beta_{\chi}}(x,T) = \# \left \{ v \in \mathrm{Gr}_{\ell,n}(\Q) : d(x,v) < c \, H(v)^{-\beta_{\chi}}, \, H(v) < T \right \}
\end{equation}
as $T \rightarrow + \infty$, for $\sigma_X$-almost every $x \in X$. In fact, Theorem \ref{thm:Critical} takes the following form in this special case.

\begin{corollary} \label{cor:Grasssmannian}
Fix integers $1 \leq \ell < n$ and let $X = \mathrm{Gr}_{\ell,n}(\R)$ be the Grassmann variety of $\ell$-dimensional subspaces in $\R^n$. Then there exists an explicit constant $\varkappa > 0$ and $\varepsilon > 0$ such that for $\sigma_X$-almost every $x \in X$, as $T \rightarrow + \infty$,
\begin{equation} \label{cor:eq_Grassmannian}
\mathcal{N}_{c, \beta}(x, T) = \varkappa \, c^d \, \ln (T) \left ( 1 + O_x(\ln(T)^{- \varepsilon}) \right ).
\end{equation}
\end{corollary} 

Let us now go through the main steps of the argument. For simplicity, we assume that $c = 1$ and we write $\cN_{\beta_{\chi}}(x,T) = \cN_{1,\beta_{\chi}}(x,T)$. The first observation is that the quantity $\cN_{\beta_{\chi}}(x,T)$ can be understood as the Siegel transform of the indicator function of a certain subset $\cE_{\beta_{\chi}}(T) \subset \widetilde{X}$ evaluated at a certain point in $\Omega = G/\G$: we can associate to each $x \in X$ an element $k_{x} \in K$ such that 
\[
\cN_{\beta_{\chi}}(x,T) = \# (\cP_{\chi} \cap k_{x} \, \cE_{\beta_{\chi}}(T)) = S_{\chi} \mathbbm{1}_{\cE_{\beta_{\chi}}(T)} (k_{x}^{-1} \G).
\]
By Theorem \ref{thm:L1}, since $\bP = \bP_{\Delta \smallsetminus \{\alpha_{\ell}\}}$ is maximal, the group $L = \bL(\R)$ is unimodular, $\widetilde{X} = G / L$ admits a unique up to scaling Radon measure $\lambda_{\widetilde{X}}$ and the expected value of $S_{\chi} \mathbbm{1}_{\cE_{\beta}(T)}$, viewed as a random variable on $\Omega$, is given by
\[
\int_{\Omega} S_{\chi} \mathbbm{1}_{\cE_{\beta_{\chi}}(T)} \, \dd \mu_{\Omega} = \int_{\widetilde{X}} \mathbbm{1}_{\cE_{\beta_{\chi}}(T)} \, \dd \lambda_{\widetilde{X}} = \lambda_{\widetilde{X}} (\cE_{\beta_{\chi}}(T)). 
\]
The hope is that, for $\sigma_X$-almost every $x \in X$, the quantity $\cN_{\beta_{\chi}}(x,T)$ is asymptotically equal to the volume $\lambda_{\widetilde{X}} (\cE_{\beta_{\chi}}(T))$, as $T \rightarrow + \infty$, and this is what we will show. In fact, the main term on the right-hand side in \eqref{cor:eq_Grassmannian} is just the explicit value of the (main term of the) volume $\lambda_{\widetilde{X}} (\cE_{\beta_{\chi}}(T))$. In order to prove the desired asymptotic estimate, we will exploit the special geometry of the set $\cE_{\beta_{\chi}}(T)$. In fact, this set can be approximated by a set $\cE_{\beta_{\chi}}(T)^{+}$ that admits a simple decomposition under the action of the diagonal subgroup 
\[
\forall \, y \in \R_+^\times, \quad a(y) = \diag \big (\underbrace{y^{-(n-\ell)/n}, \dots, y^{-(n-\ell)/n}}_{\text{$\ell$ times}}, \underbrace{y^{\ell/n}, \dots, y^{\ell/n}}_{\text{$n-\ell$ times}} \big ).
\]
Indeed, there exists a subset $\cF \subset \widetilde{X}$ such that for all integers $N \geq 1$ 
\begin{equation} \label{eq:Proof_Sketch_Decomposition}
\cE_{\beta_{\chi}}(e^N)^+ = \bigsqcup_{i=0}^{N-1} a(e^{\beta_{\chi}})^{-i} \cF. 
\end{equation}

\begin{figure}[htbp]
\includegraphics[scale=0.6]{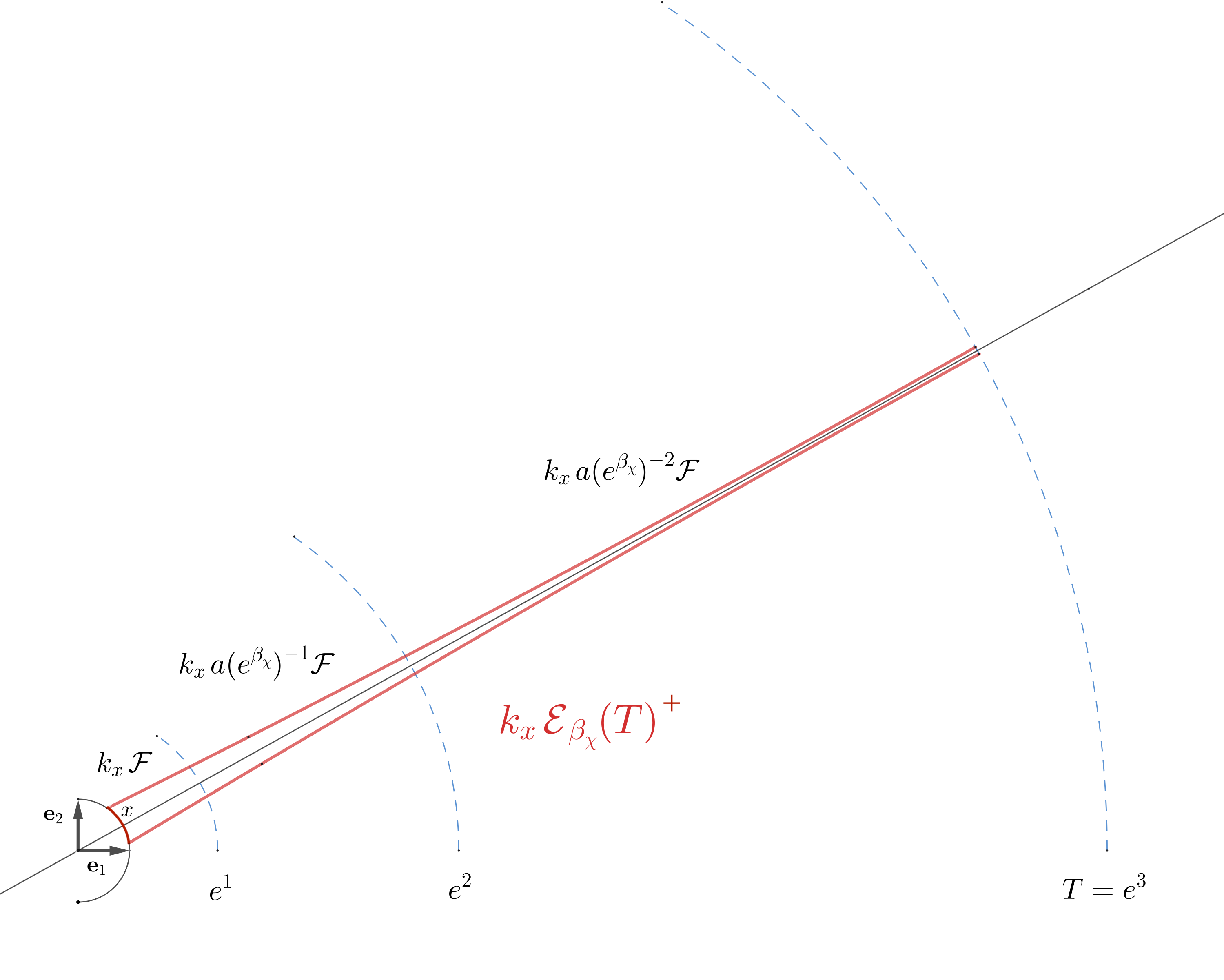}
\caption{The set $\cE_{\beta_{\chi}}(T)^+$ for the group $G = \SL_2(\R)$, the flag variety the real projective line $X = \mathbb{P}^1(\R) = \mathrm{Gr}_{1,2}(\R)$, the punctured affine cone $\widetilde{X} = \R^2 \setminus \{0\}$ above $X$, and the set $\cP_{\chi} = \mathcal{P}(\Z^2)$ of primitive elements of $\Z^2$. Rational approximations to a point $x \in X$ of height bounded by $T$ correspond to primitive lattice points of $\Z^2$ in the red region $k_x \cE_{\beta_{\chi}}(T)$, where $k_x \in \SO_2(\R)$ is a rotation such that $x = k_x [\bm{e}_1]$. The action of $a(y) = \diag ( y^{-1/2}, y^{1/2})$ with $y > 1$ on $\widetilde{X} = \R^2 \setminus \{0\}$ contracts the line through $\bm{e}_\chi = \bm{e}_1$ and expands the line through $\bm{e}_2$. The domain $\cE_{\beta_{\chi}}(T)^+$ can be decomposed into translates of the elementary domain $\cF$ under the action of $a(y)$. The hope is that for $x$ chosen randomly according to the Lebesgue measure on $X$ the number of primitive lattice points in the red region $k_x \cE_{\beta_{\chi}}(T)$, that is, the quantity $\# (\cP_{\chi} \cap k_x \, \cE_{\beta_{\chi}}(T))$, which is the classical primitive Siegel transform of the indicator function $\mathbbm{1}_{\cE_{\beta_{\chi}}(T)}$ evaluated at the rotated lattice $k_x^{-1} \Z^2$, is approximately given (up to a scalar) by the volume of $\cE_{\beta_{\chi}}(T)$.}
\label{fig:4}
\end{figure}

On the level of the Siegel transform this yields the sum decomposition
\[
S_{\chi} \mathbbm{1}_{\cE_{\beta_{\chi}}(T)^+}(k_x^{-1} \G) = \sum_{i=0}^{N-1} \mathbbm{1}_{\cF}(a(e^{\beta_{\chi}})^{i} k_x^{-1} \G).
\]
From now on, we simply view $S_{\chi} \mathbbm{1}_{\cE_{\beta_{\chi}}(T)^+}(k \G)$ as a random variable on the probability space $(K, \mu_K)$, where $\mu_K$ is the Haar probability measure of $K$. 
Up to dividing the right-hand side by $N$, it is a Birkhoff sum, but we will not take this viewpoint. Instead, we shall try to bound a quantity related to the variance of $S_{\chi} \mathbbm{1}_{\cE_{\beta_{\chi}}(T)^+}$ and then conclude by a Borel-Cantelli argument. More specifically, we shall bound a $(1+\varepsilon)$-moment, for some $\varepsilon > 0$, of the centered Siegel transform $S_{\chi} \mathbbm{1}_{\cE_{\beta_{\chi}}(T)^+} - \lambda_{\widetilde{X}}(\cE_{\beta_{\chi}}(T)^+)$, viewed as a random variable on $(K,\mu_K)$: if we can show that for some $\varepsilon > 0$ and all $N \geq 1$, 
\begin{equation} \label{eq:Proof_Sketch}
\int_K \left | S_{\chi} \mathbbm{1}_{\cE_{\beta_{\chi}}(e^N)^+}(k\G) - \lambda_{\widetilde{X}}(\cE_{\beta_{\chi}}(e^N)^+) \right |^{1+\varepsilon} \, \dd \mu_{K}(k) \, \ll \, N,
\end{equation}
then there exists $c > 0$ and $\nu(\varepsilon) \in (0,1)$ such that for $\mu_K$-almost every $k \in K$, 
\[
S_{\chi} \mathbbm{1}_{\cE_{\beta_{\chi}}(e^N)^+}(k\G) = c \, N \, \left ( 1 + O_x(N^{-\nu(\varepsilon)})\right ),
\]
as required. Due to integrability issues of the Siegel transform at this level of generality (see Theorems \ref{thm:L1} and \ref{thm:L2}), we are forced to work with $1 + \varepsilon$ for some small $\varepsilon > 0$ instead of $2$, which would represent the usual variance. 
Using the decomposition \eqref{eq:Proof_Sketch_Decomposition}, we express the argument in the integral of \eqref{eq:Proof_Sketch} as
\begin{equation*} 
S_{\chi} \mathbbm{1}_{\cE_{\beta_{\chi}}(e^N)^+}(k\G) - \lambda_{\widetilde{X}}(\cE_{\beta_{\chi}}(e^N)^+)
= \sum_{i=0}^{N-1} \bigg ( S_{\chi} \mathbbm{1}_{\cF} (a(e^{\beta_{\chi}}) k_x^{-1} \G) - \lambda_{\widetilde{X}}(\cF) \bigg ).
\end{equation*}
and obtain the bound in \eqref{eq:Proof_Sketch} using the effective single and double equidistribution property of expanding translates of $K$-orbits. In particular, we will need to work with smooth compactly supported functions that, on translated $K$-orbits, approximate the Siegel transform $S_{\chi} \mathbbm{1}_{\cF}$, which typically is neither smooth nor compactly supported. 

\subsection{Notation and conventions}
We use the Landau notation $O(\cdot)$ and the Vinogradov symbol $\ll$. Given $A, B > 0$, we use the notation $A \gg B$ for $B \ll A$, and $A \asymp B$ for $A \ll B \ll A$. We use subscripts to indicate the dependence of the constant on parameters. 
For simplicity of exposition, we will work with the set of complex points of an algebraic variety defined over $\Q$, and refer to it simply as the variety itself when no confusion arises. For instance, we write $G = \bG(\R)$ and $\bG = \bG(\C)$ to denote the groups of real and complex points of $\bG$, respectively. Given a discrete subgroup $\G \leq G$ and a closed subgroup $H \leq G$, we write $\G_H$ for $\G \cap H$. Discrete groups are always equipped with the counting measure. 

\vspace{5mm}
\textbf{Acknowledgments}. 
I am very grateful to Nicolas de Saxc\'e for introducing me to this topic, during my doctoral thesis under his supervision, and for sharing with me crucial insights that contributed to the proofs of Theorems \ref{thm:L2} and \ref{thm:Linfty}. I also thank Shucheng Yu for useful discussions, and Fr\'ed\'eric Paulin for numerous corrections and suggestions that led, in particular, to the removal of a restrictive hypothesis in Theorem \ref{thm:Critical} (namely, that $\beta_{\chi} \leq 1$). 

\section{Notation and preliminary results} \label{sec:prel}

Unless stated otherwise, we will always denote by $\bG$ a connected simply-connected almost $\Q$-simple $\Q$-group and by $\bP$ a proper parabolic $\Q$-subgroup of $\bG$. Let $\bP_0$ be a minimal parabolic $\Q$-subgroup of $\bG$ contained in $\bP$ and let $\bT$ be a maximal $\Q$-split $\Q$-torus of $\bG$ contained in $\bP_0$. Let $\Phi$, $\Delta$ and $\{\lambda_{\alpha}\}_{\alpha \in \Delta}$ be the set of roots of $\bG$ relative to $\bT$, with the ordering associated to $\bP_0$, the set of simple roots and the set of relative fundamental $\Q$-weights (see \cite[Section~12]{BT65}), respectively. We let $\G \subset \bG(\Q)$ be an arithmetic subgroup of $G$. We normalize the Haar measure $\mu_{G}$ on $G$ so that the induced $G$-invariant measure $\mu_{\Omega}$ on the quotient $\Omega = G / \G$ is a probability measure. 

\subsection{Structure of parabolic $\Q$-subgroups} \label{sec:Structure}
Let us record some facts concerning the structure of standard parabolic $\Q$-subgroups of $\bG$ (see \cite[Section~11.7]{Borel69}). 
Let $\bU_0$ be the unipotent radical of $\bP_0$. 
For each subset $\theta$ of $\Delta$, we define the $\Q$-subtorus $\bT_{\theta}$ of $\bT$ to be the connected component of the intersection of the kernels of the $\alpha \in \theta$ and the parabolic $\Q$-subgroup $\bP_{\theta}$, containing $\bP_0$, to be the product of the centralizer $\cZ(\bT_{\theta})$ of $\bT_{\theta}$ in $\bG$ and $\bU_0$. In fact, this group is a semi-direct product $\bP_{\theta} = \cZ(\bT_{\theta}) \bU_{\theta}$ of $\cZ(\bT_{\theta})$ and its unipotent radical $\bU_{\theta}$. 
Let $[\theta]$ be the set of $\Q$-roots that are linear combinations of elements of $\theta$. Let $\bQ$ be the largest connected $\Q$-anisotropic $\Q$-subgroup of $\cZ(\bT)$. There exist a connected semisimple $\Q$-subgroup $\bH_{\theta}$ of $\cZ(\bT_{\theta})$ and a connected $\Q$-subgroup $\bQ_{\theta}$ of $\bQ$ such that:
\begin{itemize}
\item the $\Q$-rank of $\bH_{\theta}$ is equal to the number of elements of $\theta$;
\item $\bS_{\theta} = (\bH_{\theta} \cap \bT)^\circ$ is a maximal $\Q$-split torus of $\bH_{\theta}$;
\item $[\theta]$ is the system of $\Q$-roots of $\bH_{\theta}$;
\item $\cZ(\bT_{\theta})$ is the almost direct product of $\bQ_{\theta}$, $\bH_{\theta}$, $\bT_{\theta}$.
\end{itemize}

Let $\bM_{\theta}$ be the identity component of the intersection of the kernels of the $\Q$-characters of $\cZ(\bT_{\theta})$. By \cite[Proposition~10.7, (b)]{Borel69}, we have $X^*(\bM_{\theta})_{\Q} = \{1\}$, the $\Q$-character group of $\bM_{\theta}$. Then $\bM_{\theta} = \bQ_{\theta} \bH_{\theta}$ and $\bP_{\theta}$ is the almost direct product:
\begin{equation} \label{eq:Langlands}
\bP_{\theta} = \bM_{\theta} \, \bT_{\theta} \, \bU_{\theta}.
\end{equation}

\subsection{Representations and height functions} \label{sec:Reps}
Assume that $\pi_{\chi} : \bG \rightarrow \GL(\bV_{\chi})$ is an irreducible rational representation which is generated by a rational line $\bD_{\chi}$ of highest $\Q$-weight $\chi \in X^*(\bT)$ such that $\bP = \mathrm{Stab}_{\bG} (\bD_{\chi})$. Such representations are referred to as strongly rational over $\Q$ and we refer the reader to \cite[Section~12]{BT65} for the details. We fix a highest weight vector $\bm{e}_{\chi} \in \bD_{\chi}(\Q)$ and denote by $x_0 = [\bm{e}_{\chi}] \in \mathbb{P}(V_{\chi})$ the corresponding point in projective space. In particular, the space of real points $X = \bX(\R)$ of the generalized flag variety $\bX = \bG / \bP$ can be identified with the orbital set $G \, [\bm{e}_{\chi}]$ via the map $\iota_{\chi} : g P \mapsto g x_0$. We define $\widetilde{X}$ to be the orbital set $\widetilde{X} = G \, \bm{e}_{\chi}$ in $V_{\chi}$. By abuse of notation, we shall refer to $\widetilde{X}$ as the cone over $X$. Fix a $\G$-stable lattice $\bV_{\chi}(\Z) \subset \bV_{\chi}(\Q)$ of $V_{\chi}$ and denote by $\cP_{\chi}$ the set of primitive elements of $\bV_{\chi}(\Z)$ that are contained in $\widetilde{X}$. Let $K$ be a maximal compact subgroup of $G$ whose Lie algebra is orthogonal that of $T = \bT(\R)$ (with respect to the Killing form on the Lie algebra $\kg$ of $G$). 
We equip $V_{\chi}$ with a Euclidean inner product $\langle \cdot , \cdot \rangle$ for which $\pi_{\chi}(g)$ is unitary (resp. self-adjoint) whenever $g \in K$ (resp. $g \in T$). We denote the induced norm by $\|\cdot \|$. We assume that $\| \bm{e}_{\chi} \| = 1$ and that $\bm{e}_{\chi} \in \cP_{\chi}$. 
First, we define a height function $H$ on $\mathbb{P}(\bV_\chi)(\Q)$ by $H([\bm{v}])=\|\bm{v}\|$, where $\bm{v}$ is a primitive vector in the lattice $\bV_{\chi}(\Z)$ representing $[\bm{v}]$. Then, using the embedding $\iota_\chi$, we obtain a height function $H_\chi$ on $\bX(\Q)$, which is given by
\[
\forall \, v \in \bX(\Q), \quad H_\chi(v) = H(\iota_\chi(v)).
\]

\subsection{Measures and coordinates} \label{sec:Measure}

In this subsection, we assume that the parabolic $\Q$-subgroup $\bP$ of $\bG$ is maximal. In particular, there exists a unique simple root $\alpha \in \Delta$ such that $\bP = \bP_{\Delta \smallsetminus \{\alpha\}}$ is the standard parabolic $\Q$-subgroup associated with the subset $\Delta \smallsetminus \{\alpha\}$ of $\Delta$. In this case, we shall denote the almost direct product decomposition in \eqref{eq:Langlands} simply by $\bP = \bM \, \bA \, \bU$, where $\bM = \bM_{\Delta \smallsetminus \{\alpha\}}$, $\bA = \bT_{\Delta \smallsetminus \{\alpha\}}$, and $\bU = \bU_{\Delta \smallsetminus \{\alpha\}}$. We note that $\bL = \mathrm{Stab}_{\bG} (\bm{e}_{\chi})$ satisfies $\bL^{\circ} = \bM \, \bU$ and hence $\bL^{\circ}$ does not admit any non-trivial $\Q$-characters. Hence, by a theorem of Borel and Harish-Chandra (see \cite[Theorem 9.4]{BHC62}), the group $L = \bL(\R)$ is unimodular, the discrete subgroup $\G_L = \G \cap L$ is a lattice in $L$, and the quotient $G/L$, that we identify with the cone $\widetilde{X}$ over $X$ via the orbital map $gL \mapsto g \bm{e}_{\chi}$, admits a unique (up to scaling) $G$-invariant Radon measure $\mu_{G/L} = \lambda_{\widetilde{X}}$. We let $\mu_L$ be the Haar measure on $L$, normalized so that the induced $L$-invariant measure on the quotient $L/ \G_L$ is a probability measure and we normalize $\mu_{G/L}$ so that (see \cite[Theorem~2.51]{Folland15}),
\begin{align} \label{eq:Folland-Unfolding}
\forall \, f \in C_c(G), \quad \int_G f(g) \, \dd \mu_G(g) &= \int_{G /L} \int_{L} f(g l) \, \dd \mu_{L}(l) \, \dd \mu_{G/L}(g L).
\end{align}

By a slight abuse of notation, we denote by $A$ the connected component with respect to the real topology of $\bA(\R)$. Let us parametrize $A$ as follows. Let $\kt$ be the Lie algebra of $T = \bT(\R)$. There exists a unique element $Y_{\alpha} \in \kt$ such that $\alpha(Y_{\alpha}) = - 1$ and $\beta(Y_{\alpha}) = 0$ for all other simple roots $\beta \in \Delta \smallsetminus \{\alpha\}$. We let 
\begin{equation} \label{eq:Parametrization}
\forall \, y \in \R_+^\times, \quad a(y) = \exp(\log(y) Y_{\alpha}).
\end{equation}
Then $A = \{a(y) : y \in \R_+^\times \}$. Let $d = \dim X$ be the dimension of $X$. Let $\mu_K$ be the Haar probability measure on $K$ and $d \mu_A(a(y)) = y^{-1} \dd y$ the push-forward to $A$ of the Haar measure on $\R_+^{\times}$ via the map $y \mapsto a(y)$. By \cite[Section~2.7]{Pfitscher24}, the group $G$ admits an almost direct product decomposition $G = K \, A \, L$ and there exists a normalizing constant $\omega_0 > 0$ such that the Haar measure $\mu_G$ of $G$ is given by 
\begin{equation} \label{eq:Measure_Decomposition_G}
\dd \mu_G = \omega_0 \, y^{-(d+1)} \, \dd \mu_K \, \dd y \, \dd \mu_L.
\end{equation}
Moreover, let $K_L = K \cap L$ and let $\sigma$ be the pushforward of the measure $\mu_K$ on $K$ to $K_L$ via the map $k \mapsto k \, K_L$. 
The map of $(K/K_L) \times A$ to $G / L = \widetilde{X}$ given by $(k \, K_L, a(y)) \mapsto k a(y) \bm{e}_{\chi}$ is a homeomorphism. In these coordinates, the measure $\mu_{G/L} = \lambda_{\widetilde{X}}$ is given by 
\begin{align}\label{eq:Measure-On_Tilde-X}
\dd \lambda_{\widetilde{X}}(k a(y) \bm{e}_\chi) 
= \omega_0 \, y^{-(d+1)} \, \dd \sigma(k) \, \dd y.
\end{align}

\subsection{Distance on $X$} \label{sec:Distance} 
For our application we also need to specify a probability measure $\sigma_X$ and a distance $d(\cdot, \cdot)$ on $X$. 
By the Iwasawa decomposition $G = KP$, the group $K$ acts transitively by left multiplication on $X$ and we let $\sigma_X$ be the unique $K$-invariant probability measure on $X$.  
Let $\mathbb{S} = \{\bm{x} \in V_\chi : \| \bm{x}\| = 1 \}$ be the unit sphere in $V_\chi$, viewed as a Riemannian submanifold of $V_\chi$. The $K$-equivariant projection map $\mathbb{S} \rightarrow \mathbb{P}(V_\chi)$, $\bm{v} \mapsto [\bm{v}]$, induces a $K$-invariant Riemannian metric on $\mathbb{P}(V_\chi)$, and by restriction also on $X$. The associated Riemannian measure equals $\vol_{\mathrm{R}}(X) \, \sigma_X$, where $\vol_{\mathrm{R}}(X)$ is the total Riemannian volume of $X$. We denote the induced distance on $X$ by $d(\cdot, \cdot)$. Let $\ku^{-}$ be the Lie algebra of the unipotent subgroup $U^-$ opposite to $P$. Let $\phi : U^- \rightarrow X$ be the map given by $\phi(u) = u x_0$. By \cite[Section~2.5]{Pfitscher24}, we have that $D_1 \phi : \ku^- \rightarrow T_{x_0} X$ is a linear isomorphism. We equip $\ku^-$ with a Euclidean structure for which this isomorphism is an isometry and we denote the implied norm on $\ku^-$ by $\|\cdot \|_{\ku^{-}}$. Then, for all $u \in \ku^-$, we have that
\begin{equation} \label{eq:Distance_Estimate}
d(x_0, \exp(u) x_0) = \|u\|_{\ku^{-}} + O\bigl (\|u\|_{\ku^{-}}^2 \bigr ),
\end{equation}
where $\exp : \ku^- \rightarrow U^-$ is the exponential map (see \cite[Lemma~2.1]{Pfitscher24}). 

\section{Integrability of Siegel transforms} \label{sec:GeneralSiegel}
 
Let us recall the definition of the Siegel transform $S_{\chi}$ (see Definition \ref{def:Siegel-Transform}): for every $f \in B_c^{\infty}(\widetilde{X})$, we defined $S_{\chi} f : \Omega \rightarrow \C$ by
\begin{equation} \label{def:Siegel-Transform-2}
\forall \, g \in G, \quad S_{\chi} f(g \G) = \sum_{\bm{v} \in \cP_{\chi}} f(g \bm{v}). 
\end{equation} 
Since $f$ is compactly supported, 
the sum on the right-hand side is finite, and hence converges absolutely. Also, since the subset $\cP_{\chi} \subset \widetilde{X}$ is $\G$-stable, the Siegel transform $S_\chi f$ is a well-defined function on $\Omega = G/\G$. 

\subsection{Some preliminary observations}
The action of $\G$ on the discrete set $\cP_{\chi}$ is not transitive in general. However, as we will now show, $\cP_{\chi}$ can always be expressed as a finite union of $\G$-orbits. 
By a theorem of Borel and Harish-Chandra \cite[Proposition~15.6]{Borel69}, the set of double cosets $\G \backslash \bG(\Q) / \bP(\Q)$ is finite. Moreover, according to \cite[Theorem~11.8]{Borel69}, we have $(\bG / \bP)(\Q) = \bG(\Q) / \bP(\Q)$. Thus, the set of rational points $(\bG / \bP)(\Q)$ of the generalized flag variety $\bG / \bP$ is a finite union of $\G$-orbits.  
The orbit map $\bG \rightarrow \mathbb{P}(\bV_{\chi})$ given by $g \mapsto g x_0$ induces an isomorphism $\bG/\bP \rightarrow \bG \, x_0$ defined over $\Q$. We identify $\bG/\bP$ with $\bG \, x_0$ via this isomorphism. Next, we note that there is a one-to-one correspondence between $(\bG/\bP)(\Q)$ and lines passing through elements of $\cP_{\chi}$. Hence there exist finitely many representatives $\bm{v}_1, \dots, \bm{v}_m \in \cP_{\chi}$, with $\bm{v}_1 = \bm{e}_{\chi}$, such that
\begin{equation} \label{eq:GammaOrbits}
\cP_{\chi} = \bigsqcup_{i=1}^m \G \, \bm{v}_i.
\end{equation}
Let us define, for every $f \in B_c^{\infty}(\widetilde{X})$, arithmetic subgroup $\G' \subset \bG(\Q)$ of $G$ and $\bm{v} \in \widetilde{X} \cap \bV_{\chi}(\Q)$, the \emph{incomplete Eisenstein series} $ E_{\chi, \G',\bm{v}} f : G/\G' \rightarrow \C$ by
\[
\forall \, g \in G, \quad  E_{\chi, \G',\bm{v}} f(g\G') = \sum_{\gamma \in \G' / (\G' \cap L_{\bm{v}}) } f(g \gamma \bm{v}),
\]
where $L_{\bm{v}} = \mathrm{Stab}_G (\bm{v})$. Letting $g_{\bm{v}} \in \bG(\Q)$ be such that $[\bm{v}] = g_{\bm{v}} [\bm{e}_{\chi}]$, we have $L_{\bm{v}} = g_{\bm{v}} L g_{\bm{v}}^{-1}$. For simplicity, we write $E_{\chi} = E_{\chi, \G, \bm{e}_{\chi}}$. Hence, by \eqref{eq:GammaOrbits}, we have
\begin{equation} \label{eq:Sum_Eisenstein}
\forall \, g \in G, \quad S_{\chi} f(g \G) = \sum_{i=1}^m E_{\chi, \G,\bm{v}_i} f(g\G)
\end{equation}
The next lemma shows that the property that $S_{\chi}$ maps $B_{c}^{\infty}(\widetilde{X})$ into $L^p(\Omega)$ is independent of the choices of the arithmetic subgroup $\G$ and the $\G$-stable lattice $\bV_{\chi}(\Z)$.
\begin{lemma} \label{lem:Eisenstein_Single_Orbit}
Let $p \in [1,+\infty]$. The following assertions hold.
\begin{enumerate}
\item The fact that $E_{\chi, \G', \bm{v}}$ maps $B_{c}^{\infty}(\widetilde{X})$ into $L^p(G/\G')$ is independent of the choice of the arithmetic subgroup $\G'$ and the rational element $\bm{v} \in \widetilde{X} \cap \bV_{\chi}(\Q)$. 
\item The Siegel transform $S_{\chi}$ maps $B_{c}^{\infty}(\widetilde{X})$ into $L^p(\Omega)$ if and only if $E_{\chi}$ does so. 
\end{enumerate}
\end{lemma}

Given an arithmetic subgroup $\G'$ of $G$, we identify $G/\G'$ with a fundamental domain $\cF'$ in $G$ and the $G$-invariant Haar measure on $G/\G'$ with the restriction of the Haar measure $\mu_G$ on $G$ to $\cF'$. 

\begin{proof}
The second claim follows from the first by noting that $S_{\chi}$ is a finite sum of incomplete Eisenstein series \eqref{eq:Sum_Eisenstein}. As we will show, the first claim is a consequence of the fact that for a fixed $\Q$-structure on $\bG$, any two arithmetic subgroups ${\G_1, \G_2 \subset \bG(\Q)}$ of $G$ are commensurable, that is, their intersection $\G_1 \cap \G_2$ has finite index in both $\G_1$ and $\G_2$. 

Let us first reduce \emph{(1)} to the case where $\bm{v} = \bm{e}_{\chi}$. For a function $f \in B_c^{\infty}(\widetilde{X})$ and a scalar $\lambda > 0$, let us write $f_{\lambda}(\cdot) = f(\lambda \cdot)$.  Then, for every $\bm{v} \in \widetilde{X} \cap \bV_{\chi}(\Q)$ (with $\lambda_{\bm{v}} > 0$ and $g_{\bm{v}} \in \bG(\Q)$ such that $\bm{v} = \lambda_{\bm{v}} g_{\bm{v}} \bm{e}_{\chi}$) and arithmetic subgroup $\Gamma'$, we have
\begin{align*}
E_{\chi, \G',\bm{v}} f(g\G) = \sum_{\gamma \in \G' / (\G' \cap L_{\bm{v}}) } f(g \gamma \lambda_{\bm{v}} g_{\bm{v}} \bm{e}_{\chi})
= \sum_{\gamma \in g_{\bm{v}}\G'g_{\bm{v}}^{-1} / (g_{\bm{v}}\G'g_{\bm{v}}^{-1} \cap L) } f_{\lambda_{\bm{v}}}(g g_{\bm{v}} \gamma \bm{e}_{\chi}).
\end{align*}
Using the substitution $g \mapsto g g_{\bm{v}}^{-1}$, we have, for $p < + \infty$, 
\begin{align*}
\int_{G/\G'} \left | E_{\chi, \G', \bm{v}} f \right |^p \, \dd \mu_{G} &= \int_{G/g_{\bm{v}} \G' g_{\bm{v}}^{-1}}  \Big | \sum_{\gamma \in g_{\bm{v}} \G' g_{\bm{v}}^{-1} / (g_{\bm{v}} \G' g_{\bm{v}}^{-1} \cap L)} f_{\lambda_{\bm{v}}} (g \gamma \bm{e}_{\chi}) \Big |^p \, \dd \mu_G(g) \\
&= \int_{G/g_{\bm{v}} \G' g_{\bm{v}}^{-1}} \left | E_{\chi, g_{\bm{v}} \G' g_{\bm{v}}^{-1}, \bm{e}_{\chi}} f_{\lambda_{\bm{v}}} \right |^p \, \dd \mu_{G}.
\end{align*}
Suppose now that $p = + \infty$. We may assume that $f \in C_c(\widetilde{X})$. Then $E_{\chi, \G', \bm{v}} f$ is a continuous function and 
\[
\| E_{\chi, \G_1, \bm{v}_1} f \|_{\infty} 
= \sup_{g \in G} \left | \sum_{\gamma \in g_{\bm{v}}\G'g_{\bm{v}}^{-1} / (g_{\bm{v}}\G'g_{\bm{v}}^{-1} \cap L) } f_{\lambda_{\bm{v}}}(g \gamma \bm{e}_{\chi}) \right | =\| E_{\chi, g_{\bm{v}} \G' g_{\bm{v}}^{-1}, \bm{e}_{\chi}} f_{\lambda_{\bm{v}}} \|_{\infty}.
\]
Therefore, by symmetry, it suffices to show that if $E_{\chi, \G_1} := E_{\chi, \G_1, \bm{e}_{\chi}}$ maps $B_c^{\infty}(\widetilde{X})$ into $L^p(G/\G_1)$, then $E_{\chi, \G_2} := E_{\chi, \G_2, \bm{e}_{\chi}}$ maps $B_c^{\infty}(\widetilde{X})$ into $L^p(G/\G_2)$. We may assume without loss of generality that $f \geq 0$, and we first treat the case $p < + \infty$. Observe that, for every $g \in G$, we have
\begin{align*}
\left | E_{\chi, \G_2} f(g \G_2) \right |^p &\leq \left | \sum_{\gamma \in \G_2 / (\G_1 \cap \G_2 \cap L)} f(g \gamma \bm{e}_{\chi}) \right |^p \\
&\leq \max_{\gamma_1 \in \G_2 / \G_1 \cap \G_2} \left | \# \big ( \G_2 / (\G_1 \cap \G_2) \big ) \sum_{\gamma_2 \in (\G_1 \cap \G_2) / (\G_1 \cap \G_2 \cap L)} f(g \gamma_1 \gamma_2 \bm{e}_{\chi}) \right |^p. 
\end{align*}
Let $\cF_2$ be a fundamental domain for $\G_2$ in $G$ and let $\cF_{1,2}$ be the fundamental domain for $\G_1 \cap \G_2$ in $G$ given by $\cF_{1,2} = \bigsqcup_{\gamma_1 \in \G_2 / (\G_1 \cap \G_2)} \cF_2 \gamma_1$. Then, we have
\begin{align*}
\int_{\cF_2} \left | E_{\chi, \G_2} f \right |^p \, \dd \mu_G &\lesssim_{\G_1, \G_2} \sum_{\gamma_1 \in \G_2 / (\G_1 \cap \G_2)} \int_{\cF_2 \gamma_1} \left | \sum_{\gamma \in (\G_1 \cap \G_2) / (\G_1 \cap \G_2 \cap L)} f(g \gamma \bm{e}_{\chi}) \right |^p \, \dd \mu_G(g) \\
&= \int_{\cF_{1,2}} \left | \sum_{\gamma \in (\G_1 \cap \G_2) / (\G_1 \cap \G_2 \cap L)} f(g \gamma \bm{e}_{\chi}) \right |^p \, \dd \mu_G(g).
\end{align*}
Now, since $(\G_1 \cap \G_2) / (\G_1 \cap \G_2 \cap L)$ injects into $\G_1 / (\G_1 \cap L)$,
$$
\left | \sum_{\gamma \in (\G_1 \cap \G_2) / (\G_1 \cap \G_2 \cap L)} f(g \gamma \bm{e}_{\chi}) \right |^p \leq \left | \sum_{\gamma \in \G_1 / (\G_1 \cap L)} f(g \gamma \bm{e}_{\chi}) \right |^p.
$$
Let $\cF_1$ be a fundamental domain for $\G_1$ in $G$ such that $\cF_{1,2} = \bigsqcup_{\gamma_1 \in \G_1 / (\G_1 \cap \G_2)} \cF_1 \gamma_1$. Then, using the fact that the function sending $g$ to $\sum_{\gamma \in \G_1 / \G_1 \cap L} f(g \gamma L)$ is right $\G_1$-invariant, we have
$$
\int_{\cF_{1,2}} \left | \sum_{\gamma \in (\G_1 \cap \G_2) / (\G_1 \cap \G_2 \cap L)} f(g \gamma \bm{e}_{\chi}) \right |^p \, \dd \mu_G(g) \lesssim_{\G_1, \G_2} \int_{\cF_{1}} \left | \sum_{\gamma \in \G_1 / (\G_1 \cap L)} f(g \gamma \bm{e}_{\chi})  \right |^p \, \dd \mu_G(g)
$$
and the latter converges, by assumption. This concludes the proof of the case $p < + \infty$. Now suppose that $p = + \infty$. It follows from the above that
\begin{align*}
\left | E_{\chi, \G_2} f(g\G_2) \right | &\leq \max_{\gamma_1 \in \G_2 / (\G_1 \cap \G_2)}  \# \big ( \G_2 / (\G_1 \cap \G_2) \big ) \left | \sum_{\gamma_2 \in (\G_1 \cap \G_2) / (\G_1 \cap \G_2 \cap L)} f(g \gamma_1 \gamma_2 L) \right | \\
&\leq \max_{\gamma_1 \in \G_2 / (\G_1 \cap \G_2)}  \# \big ( \G_2 / (\G_1 \cap \G_2) \big ) \left | \sum_{\gamma_2 \in \G_1 / (\G_1 \cap L)} f(g \gamma_1 \gamma_2 L) \right |,
\end{align*}
where for the second inequality we used that $(\G_1 \cap \G_2) / (\G_1 \cap \G_2 \cap L)$ injects into $\G_1 / (\G_1 \cap L)$. Taking upper bounds yields the claim. The proof is complete.
\end{proof} 

\subsection{$L^1$-integrability}
In this section, we prove Theorem \ref{thm:L1}. We define 
\[
\forall \, g \in G, \quad \lambda_\chi(g\G) = \min_{\bm{v} \in \bV_{\chi}(\Z) \smallsetminus \{\bm{0}\}} \| g \bm{v} \|
\]
to be the length of the shortest non-zero vector of $g \bV_{\chi}(\Z)$. For every $\Q$-weight $\mu \in X^*(\bT)$ of the representation $\pi_{\chi}$ let 
\[
\bV^{\mu} = \{\bm{v} \in \bV_{\chi} : \forall \, t \in \bT, \, \pi_{\chi}(t) \bm{v} = \mu(t) \bm{v} \}.
\]
This is a $\Q$-subspace of $\bV_{\chi}$. It is known that $\bV^{\chi}$ is one-dimensional, that $\bV_{\chi}$ is the direct sum of the linear subspaces $\bV^{\mu}$,
\begin{equation} \label{eq:Direct-Sum}
\bV_{\chi} = \bigoplus_{\mu} \bV^{\mu},
\end{equation}
and that every $\Q$-weight of $\pi_{\chi}$ has the form
\begin{equation} \label{eq:Q-weights}
\mu = \chi - \sum_{\alpha \in \Delta} c_{\alpha}(\mu) \alpha \quad \text{ with $c_{\alpha}(\mu) \in \N$.}
\end{equation}
We may assume that $\bV_{\chi}(\Z)$ is the $\Z$-span of an orthonormal basis consisting of weight vectors for the action of $T = \bT(\R)$. Moreover, by Lemma \ref{lem:Eisenstein_Single_Orbit}, we may assume that $\G$ is given by the stabilizer in $G$ of the lattice $\bV_{\chi}(\Z)$. In the following lemma, we bound the Siegel transform of a function $f \in B_{c}^{\infty}(\widetilde{X})$ in terms of $\lambda_\chi$.

\begin{lemma} \label{lem:Upper-Bound-Siegel}
Suppose that the parabolic $\Q$-subgroup $\bP$ is maximal. Then, for every $f \in B_{c}^{\infty}(\widetilde{X})$, we have
\begin{equation} \label{eq:Replace_Schmidt}
\forall \, g \in G, \quad | S_{\chi} f(g\G)| \ll_{\supp(f)} \|f\|_{\infty} \, \lambda_\chi(g\G)^{-\beta_{\chi} d }.
\end{equation}
\end{lemma}
Although this bound is not optimal, it is sufficient for our purposes. 

\begin{proof}
We shall need the following consequence of the proof of \cite[Theorem~C]{Pfitscher24}. For every $T \geq 1$, consider the function 
\[
\cN(T) = \# \left \{v \in \bX(\Q) : H_{\chi}(v) < T \right \}
\]
counting rational points in $X$ of height $< T$. Let $\beta_{\chi} \in \Q_{>0}$ be the Diophantine exponent of $X$ with respect to $\chi$ (see \cite[D\'efinition~2.4.1 et Th\'eor\`eme~2.4.5]{deSaxce20}) and let $d = \dim X$ be the dimension of $X$. Then, as $T \rightarrow + \infty$, we have $\cN(T) \sim T^{\beta_{\chi} d}$. Since there is a one-to-one correspondence between points in $\bX(\Q)$ and lines passing through $\cP_{\chi}$, by the definition of the height function $H_{\chi}$, we also have that, as $T \rightarrow + \infty$
\begin{equation} \label{eq:Number_Primitive_Bounded_Height}
\# \left \{\bm{v} \in \cP_{\chi} : \|\bm{v}\| < T \right \} \asymp T^{\beta_{\chi} d}.
\end{equation}
Fix $f \in B_{c}^{\infty}(\widetilde{X})$ and pick $r = r(\supp(f)) \geq 1$ such that $\supp(f)$ is contained in $B_{\widetilde{X}}(r) = \{\bm{v} \in \widetilde{X} : \|\bm{v}\| < r \}$. The proof now proceeds using reduction theory as presented, for instance, in \cite[Section~12, Theorem~13.1]{Borel69}. By a slight abuse of notation, we let $\mathfrak{a}$ be the Lie algebra of $T^{\circ}$ and, for every $\tau \geq 0$, let $\mathfrak{a}_{\tau} = \{Y \in \mathfrak{a} : \forall \, \beta \in \Delta, \, \beta(Y) \leq \tau \}$. We set $A_{\tau} = \exp \, \mathfrak{a}_{\tau}$ and an note that $\mathfrak{a}^- = \mathfrak{a}_0$ is the negative Weyl chamber of $\mathfrak{a}$ with respect to $\Delta$. Let $\bM_0$ be the largest $\Q$-anisotropic $\Q$-subgroup of the centralizer $\cZ_{\bG}(\bT)^{\circ}$ in $\bG$ of $\bT$ and let $\bU_0$ be the unipotent radical of the minimal parabolic $\Q$-subgroup $\bP_0$. There exist $\tau > 0$, a compact subset $\bm{\omega}$ of $M_0 U_0$, and a finite subset $C \subset \bG(\Q)$ such that the Siegel set $\mathfrak{S} = K \, A_{\tau} \, \bm{\omega}$ satisfies
\[
G = \mathfrak{S} \, C \, \G.
\]
In particular, we can express, though not uniquely, each $g \in G$ as $g = k a n c \gamma$ with $k \in K$, $a \in A_{\tau}$, $n \in \bm{\omega}$, $c \in C$, and $\gamma \in \G$. Fix any norm $\|\cdot \|_{\mathfrak{a}}$ on $\mathfrak{a}$ and, for $r_0 > 0$, let $B_{\mathfrak{a}}(r_0)$ denote the corresponding ball centered at the origin with radius $r_0$. Let $r_0 > 0$ be such that $\mathfrak{a}_{\tau}$ is contained in $\mathfrak{a}^- + B_{\mathfrak{a}}(r_0)$. Let $k \in K$, $n \in \bm{\omega}$, $a \in A_{\tau}$, $c \in C$, and $\gamma \in \G$. We express $a = a^- \, \exp(O(1))$ with $a^- \in \exp(\mathfrak{a}^-)$. Using that $\lambda_\chi$ is right $\G$-invariant, that $K$ is compact, that $\bigcup_{a \in A_{\tau}} a \bm{\omega}a^{-1}$ is relatively compact (see \cite[Lemma~12.2]{Borel69}), and that $C \subset \bG(\Q)$ consists of rational elements and is finite, we have
\begin{equation} \label{eq:Proof-lem:Upper-Bound-Siegel}
\lambda_\chi(k a n c \gamma \G) \asymp \lambda_\chi(a^- \G).
\end{equation}
By the description of the $\Q$-weights of the representation $\pi_{\chi}$ in \eqref{eq:Q-weights}, for every $\Q$-weight $\mu$ of $\pi_{\chi}$, we have
\[
\chi(a^-) \leq \mu(a^-). 
\]
Hence, since we assumed $\bV_{\chi}(\Z)$ to be spanned over $\Z$ by an orthonormal basis consisting of weight vectors for the action of $T$, we have $\lambda_\chi(a^-) = \chi(a^-)$. Thus, for every $\bm{v} \in V_{\chi}$, we have $\lambda_\chi(a^-) \| \bm{v} \| \leq \|a^- \bm{v} \|$. Using that the norm $\| \cdot \|$ on $V_{\chi}$ is $K$-invariant, that $\bigcup_{a \in A_{\tau}} a \bm{\omega}a^{-1}$ is relatively compact, and that $C \subset \bG(\Q)$ is finite, there exists a constant $C_0 \geq 1$, independent of $f$, such that, for every $g \in G$ with Siegel decomposition $g = k a n c \gamma$ (and writing $a = a^- \exp(O(1))$ as above), we have
\begin{align*}
|S_{\chi} f(g \G)| &\leq \|f\|_{\infty} \, \# \big \{ \bm{v} \in \cP_{\chi} : \| g \bm{v} \| < r \big \} \\
&\leq \|f\|_{\infty} \, \# \big \{ \bm{v} \in \cP_{\chi} : \| \bm{v} \| < C_0 \, \lambda_\chi(a^-)^{-1} \, r \big \}.
\end{align*}
By the estimate in \eqref{eq:Number_Primitive_Bounded_Height}, we further have 
\[
\# \big \{ \bm{v} \in \cP_{\chi} : \| \bm{v} \| < C_0 \, \lambda_\chi(a^-)^{-1} \, r \big \} \, \ll_{\supp(f)} \, \lambda_\chi(a^-)^{-\beta_{\chi} d}.
\]
This together with \eqref{eq:Proof-lem:Upper-Bound-Siegel} now implies that
\[
|S_{\chi} f(g\G)| \ll_{\supp(f)} \|f\|_{\infty} \, \lambda_\chi(g \G)^{-\beta_\chi d},
\]
finishing the proof of the lemma.
\end{proof}

\begin{proof} [Proof of Theorem~\ref{thm:L1}]
We first show that $(1) \Rightarrow (2)$. In view of the Riesz-Markov-Kakutani representation theorem, since $\Lambda(f) = \int_{\Omega} S_\chi f \, \dd \mu$ defines a positive $G$-invariant linear functional on $B_{c}^{\infty}(\widetilde{X})$ by assumption, there exists a unique $G$-invariant Radon measure $\lambda_{\widetilde{X}}$ on $\widetilde{X}$ such that for all $f \in B_{c}^{\infty}(\widetilde{X})$,
\begin{equation} \label{eq:Formula-Proof-L1}
\int_{\Omega} S_\chi f \, \dd \mu_{\Omega}= \int_{\widetilde{X}} f \, d \lambda_{\widetilde{X}}.
\end{equation}
In particular, we have $\|S_\chi f\|_{L^1(\Omega)} \leq \|f\|_{L^1(\widetilde{X})}$ and $S_{\chi}$ extends to a bounded operator $S_{\chi} : L^1(\widetilde{X}) \rightarrow L^1(\Omega)$; the formula \eqref{eq:Formula-Proof-L1} continues to hold for all $f \in L^1(\widetilde{X})$.

To see the implication $(2) \Rightarrow (3)$, we first note that since $\widetilde{X} = G/L$ carries a positive $G$-invariant Radon measure $\lambda_{\widetilde{X}}$, $L$ must be a unimodular subgroup of $G$ (see, for instance, \cite[Theorem~2.51]{Folland15}). Fix a Haar measure $\mu_L$ on $L$ and the counting measure on the discrete group $\G_L = \G \cap L$.
By assumption, for every non-negative $f \in B_{c}^{\infty}(\widetilde{X})$, we have
\[
\int_{G/\Gamma} \sum_{\gamma \in \Gamma / \G_L} f(g \gamma \bm{e}_{\chi}) \, \dd \mu_{\Omega}(g\G) \leq \int_{\Omega} S_\chi f \, \dd \mu_{\Omega}< + \infty.
\]
Using a standard folding/unfolding argument, there exists (unique up to scaling) $G$ and $L$-invariant measures $\mu_{G/L}$ and $\mu_{L/\G_L}$ on $G/L$ and $L/\G_L$, respectively, that can be normalized such that for every non-negative $f \in B_{c}^{\infty}(\widetilde{X})$,
\[
\int_{G/\Gamma}  \sum_{\gamma \in \Gamma / \G_L} f(g \gamma \bm{e}_{\chi}) \, \dd \mu_{\Omega}(g\G) = \int_{G/L} \int_{L/\G_L} f(gl \bm{e}_{\chi}) \, \dd \mu_{L/\G_L}(l \G_L) \dd \mu_{G/L}(g L).
\]
Using that $L$ stabilizes the vector $\bm{e}_{\chi}$, we further have
\[
\int_{G/L} \int_{L/\G_L} f(gl \bm{e}_{\chi}) \dd \mu_{L/\G_L}(l \G_L) \dd \mu_{G/L}(g L) = \mu_{L/\G_L}(L/\G_L) \int_{G/L} f(g\bm{e}_{\chi}) \dd \mu_{G/L}(g L).
\]
Choosing $f \in B_{c}^{\infty}(\widetilde{X})$ non-negative so that $\int_{G/L} f(g\bm{e}_{\chi}) \, \dd \mu_{G/L}(g) > 0$, we have
\[
\mu_{L/\G_L}(L/\G_L) \leq \frac{\int_{\Omega} S_\chi f \, \dd \mu}{\int_{G/L} f(g\bm{e}_{\chi}) \, \dd \mu_{G/L}(g)} < + \infty,
\]
which shows that $\G_L$ is a lattice in $L$, as required.

Next we show that $(3) \Leftrightarrow (4)$. 
As recorded in Section \ref{sec:Structure}, the group $\bP = \bP_{\theta}$ is an almost direct product $\bP = \bM_{\theta} \bT_{\theta} \bU_{\theta}$.
The stabilizer $\bL$ in $\bG$ of $\bm{e}_{\chi}$ then satisfies $\bL^{\circ} = \bM_{\theta} (\bT_{\theta} \cap \ker(\chi))^{\circ} \bU_{\theta}$. Now, by a theorem of Borel and Harish-Chandra~\cite[Theorem 9.4]{BHC62}, the Lie group $L$ is unimodular and $\G_L$ is a lattice in $L$ if and only if $\bL^{\circ}$ does not admit any non-trivial $\Q$-characters. The group of $\Q$-characters $X^*(\bL^{\circ})_{\Q}$ can be identified, by restriction, with $X^*(\bM_{\theta}(\mathbf{T}_{\theta} \cap \ker(\chi))^{\circ})_{\Q}$. Since $\bM_{\theta}$ is an almost direct product of a $\Q$-anisotropic and a semisimple $\Q$-subgroup (see Section \ref{sec:Structure}), the latter is trivial if and only if the central $\Q$-split torus $(\mathbf{T}_{\theta} \cap \ker(\chi))^{\circ}$ is trivial. 
The torus $\bT_{\theta}$ acts non-trivially on the line through $\bm{e}_{\chi}$ via the character $\chi$. In particular, the quotient $\bT_{\theta} / (\mathbf{T}_{\theta} \cap \ker(\chi))^{\circ}$ is one-dimensional. Thus $(\mathbf{T}_{\theta} \cap \ker(\chi))^{\circ}$ is trivial if and only if $\bT_{\theta}$ is one-dimensional if and only if $\bP$ is maximal, as claimed.

Next we show that $(4) \Rightarrow (5)$. By assumption, $\bP$ is a maximal parabolic $\Q$-subgroup. By Lemma~\ref{lem:Eisenstein_Single_Orbit}, it suffices to show that there exists $\varepsilon > 0$ such that $E_{\chi} (B_{c}^{\infty}) \subset L^{1+\varepsilon} (\Omega)$. Using a folding/unfolding argument similar to that in the proof of the implication $(2) \Rightarrow (3)$, for every $\varepsilon > 0$ and non-negative $f \in B_{c}^{\infty}(\widetilde{X})$, we have
\begin{align*}
\int_\Omega \left | E_{\chi} f \right |^{1+ \varepsilon}  \, \dd \mu_{\Omega} &= \int_\Omega \sum_{\gamma \in \G / \G_L} \left  ( f(g \gamma \bm{e}_{\chi}) \, \big ( E_{\chi} f (g\G) \big)^\varepsilon \right )  \, \dd \mu_{\Omega}(g\G) \\
&= \int_{G/L} \int_{L/\G_L} f(g l \bm{e}_{\chi}) \, \big (E_{\chi} f (g l \G) \big)^\varepsilon \, \dd \mu_{L/\G_L}(l \G_L) \, \dd \mu_{G/L}(g L) \\
&= \int_{G/L} f(g \bm{e}_{\chi}) \, \int_{L/\G_L} \big (E_{\chi} f (g l\G) \big)^\varepsilon \, \dd \mu_{L/\G_L}(l \G_L) \, \dd \mu_{G/L}(g L).
\end{align*}
Moreover, using the identification $\widetilde{X} = G/L = K/K_L \times A$, this further equals  
\[
\int_{\widetilde{X}} f(k a(y) \bm{e}_{\chi}) \, \int_{L/\G_L} \big ( E_{\chi} f (k a(y) l \G) \big )^\varepsilon \, \dd \mu_{L/\G_L}(l \G_L) \, d \lambda_{\widetilde{X}}(k a(y) \bm{e}_{\chi}).
\]
By Lemma \ref{lem:Upper-Bound-Siegel}, we can estimate $|E_{\chi} f(g\G)| \ll_{\supp(f)} \|f\|_{\infty} \lambda_{\chi} (g\G)^{-\beta_\chi d}$. Putting everything together, we obtain, with $\varepsilon' = \varepsilon \beta_{\chi} d$,
\begin{align} \label{eq:WeilIntegrationDeltaK}
&\int_\Omega \left | E_{\chi} f \right |^{1+ \varepsilon} \, \dd \mu_{\Omega} \\ \notag 
&\ll_{\supp f} \|f\|_{\infty}^{\varepsilon} \int_{\widetilde{X}} f(k a(y) \bm{e}_{\chi}) \int_{L/\G_L} \lambda_\chi(k a(y) l \G)^{-\varepsilon'} \, \dd \mu_{L/\G_L}(l \G_L) \dd \lambda_{\widetilde{X}}(k a(y) \bm{e}_{\chi}).
\end{align}
Since $f$ is continuous and compactly supported and $\lambda_\chi(k a(y) l)^{-\varepsilon'} \ll \lambda_\chi(l)^{-\varepsilon'}$ for all $l \in L$ and all $(k K_L, a(y)) \in K/K_L \times A$ such that $k a(y) \bm{e}_{\chi} \in \supp f$, it suffices to show that, for $\varepsilon' > 0$ small enough,
\[
\int_{L/\G_L} \lambda_\chi(l\G)^{-\varepsilon'} \, \dd \mu_{L/\G_L}(l \G_L) \, < \, + \infty. 
\]

By the maximality of $\bP$, we know from Section \ref{sec:Structure} (on the structure of parabolic subgroups) that $\bL^{\circ} = \bM \, \bU$ is the semidirect product of the unipotent radical of $\bP$ and a connected reductive $\Q$-subgroup $\bM$ of $\bG$, that, in turn, is an almost direct product $\bM = \bQ \, \bH$ of a connected $\Q$-anisotropic $\Q$-subgroup $\bQ$ and a semisimple $\Q$-subgroup $\bH$ satisfying the following conditions:
\begin{itemize}
\item the $\Q$-rank of $\bH$ equals $\mathrm{rank}_\Q \, \bG - 1$;
\item $\bT_{\bH} = (\bH \cap \bT)^{\circ}$ is a maximal $\Q$-split torus of $\bH$;
\item $\bP_{0,\bH} = (\bH \cap \bP_0)^{\circ}$ is a minimal parabolic $\Q$-subgroup of $\bH$;
\item $K_H = K \cap H$ is a maximal compact subgroup of $H$;
\item $\Delta_H = \Delta \smallsetminus \{\alpha\}$ is the set of simple roots with ordering associated to $\bP_{0,\bH}$;
\item $[\Delta_H]$ is the system of $\Q$-roots of $\bH$.
\end{itemize}
For any subgroup $N$ of $G$, let us denote by $\G_N$ the intersection $\G \cap N$. The subgroup $\G_Q \, \G_H \, \G_U$ has finite index in $\G_L$ (see \cite[Corollary~6.4]{BHC62}) and $Q \, H \, U$ has finite index in $L$. Without loss, we may assume that $\G_L = \G_Q \, \G_H \, \G_U$ and that $L = Q \, H \, U$. If $\Omega_Q$, $\Omega_H$, and $\Omega_U$ are fundamental sets (in the sense of \cite[Definition~9.6]{Borel69}) for $\G_Q$ in $Q$, $\G_H$ in $H$, and $\G_U$ in $U$, respectively, then $\Omega_Q \, \Omega_H \, \Omega_U$ is a fundamental set for $\G_L$ in $L$. Moreover, we can choose the sets $\Omega_Q$ and $\Omega_U$ to be compact (since $\bQ$ is $\Q$-anisotropic and $\bU$ unipotent). Denoting $\mu_Q$, $\mu_H$, and $\mu_U$ the corresponding Haar measures, it suffices to show the finiteness of 
\[
\int_{\Omega_Q} \int_{\Omega_H} \int_{\Omega_U}  \lambda_\chi(q h u \G)^{-\varepsilon'} \, \dd \mu_Q(q) \dd \mu_H(h) \dd \mu_U(u). 
\]
Since $\Omega_Q$ is compact and $\lambda_\chi(q g \G) \asymp_{\Omega_Q} \lambda_\chi(q g \G)$ for all $q \in \Omega_Q$ and $g \in G$, we further reduce to show that 
\[
\int_{\Omega_H} \int_{\Omega_U}  \lambda_\chi(h u \G)^{-\varepsilon'} \, \dd \mu_H(h) \dd \mu_U(u) \, < \, + \infty. 
\]
The argument now relies again on reduction theory (see \cite[Section~12, Theorem~13.1]{Borel69}), but this time applied to $\G_H$ and $H$. We let $\mathfrak{a}_H$ be the Lie algebra of $T_H^{\circ}$ and, for every $\tau \geq 0$, let $\mathfrak{a}_{\tau,H} = \{Y \in \mathfrak{a}_H : \forall \, \beta \in \Delta_H, \, \beta(Y) \leq \tau \}$. We set $A_{\tau,H} = \exp \, \mathfrak{a}_{\tau,H}$ and an note that $\mathfrak{a}_H^- = \mathfrak{a}_{0,H}$ is the negative Weyl chamber of $\mathfrak{a}_H$ with respect to $\Delta_H$. Let $\bM_{0,\bH}$ be the largest $\Q$-anisotropic $\Q$-subgroup of the centralizer $\cZ_{\bH}(\bT_H)^{\circ}$ in $\bH$ of $\bT_H$ and let $\bU_{0,\bH}$ be the unipotent radical of the minimal parabolic $\Q$-subgroup $\bP_{0,\bH}$. There exist $\tau > 0$, a compact subset $\bm{\omega}_H$ of $M_{0,H} U_{0,H}$, and a finite subset $C_H \subset \bH(\Q)$ such that the Siegel set $\mathfrak{S}_H = K_H \, A_{\tau,H} \, \bm{\omega}_H$ satisfies
\[
H = \mathfrak{S}_H \, C_H \, \G_H.
\]
We let $\Omega_H = \mathfrak{S}_H \, C_H$. We claim that $\lambda_{\chi}(a \G) \ll \lambda_{\chi}(h u \G)$ for all $h = kanc$ with $k \in K_H$, $a \in A_{\tau, H}$, $n \in \bm{\omega}_H$, $c \in C_H$, and $u \in \Omega_U$. Indeed, using the fact that $\bigcup_{a \in A_{\tau,H}} a \bm{\omega}_H a^{-1}$ is relatively compact (see \cite[Lemma~12.2]{Borel69}) and since $\min_{\bm{v} \in \bV_{\chi}(\Z) \smallsetminus \{\bm{0}\}, \, c \in C_H, \, u \in \Omega_U} \| c u \bm{v}\| > 0$ (since $C_H$ is finite, $\Omega_U$ is compact, and $\bV_{\chi}(\Z) \smallsetminus \{\bm{0}\}$ is discrete), we have
\[
\lambda_\chi(h u \G) = \min_{\bm{v} \in \mathbf{V}_\chi(\Z) \smallsetminus \{0\}} \|k a n (a)^{-1} a c u \bm{v}\| \gg \lambda_\chi(a) \min_{\bm{v} \in \mathbf{V}_\chi(\Z)\smallsetminus \{0\}, \, c \in C_H, \, u \in \Omega_U} \| c u \bm{v}\|,
\]
which yields the claim. Let $\rho : \mathfrak{a}_H \rightarrow (0,+\infty)$ be the sum of the positive roots of $H$ relative to $T_H$ with multiplicities counted; it can be written as $\rho = \sum_{\beta \in \Delta_H} n_\beta \beta$ where the $n_\beta$ are strictly positive integers. Let $\log : T_H^{\circ} \rightarrow \mathfrak{a}_{H}$ be the logarithm map. According to \cite[Proposition~8.44]{Knapp23}, since we have a Cartan decomposition $H = K_H P_{0,H}^{\circ}$ with $P_{0,H}^{\circ} = M_{0,H}^{\circ} \, T_H^{\circ} \, U_{0,H}$, the corresponding Haar measures $\mu_{K_H}$, $\mu_{M_{0,H}^{\circ}}$, $\mu_{T_H^{\circ}}$, and $\mu_{U_{0,H}}$ can be normalized so that
\[
\dd \mu_{H}(h) = e^{\rho(\log(a))} \, \dd \mu_{K_H}(k) \, \dd \mu_{M_{0,H}^{\circ}}(m) \, \dd \mu_{T_H^{\circ}}(a) \, \dd \mu_{U_{0,H}}(u).
\]
Putting everything together and using the structure of the Siegel set $\mathfrak{S}_H$,
\[
\int_{\Omega_H} \int_{\Omega_U}  \lambda_\chi(h u \G)^{-\varepsilon'} \, \dd \mu_H(h) \dd \mu_U(u) \, \ll \, \int_{A_{\tau,H}} \lambda_\chi(a)^{-\varepsilon'} e^{\rho(\log(a))} \, \dd \mu_{T_H^{\circ}}(a).
\]
The exponential map $\exp : \mathfrak{a}_H \rightarrow T_H^{\circ}$ is an isomorphism that carries a Lebesgue measure to a Haar measure. Hence for a suitable normalization of the Lebesgue measure $\dd Y$ on $\mathfrak{a}_H$, we have 
\[
\int_{A_{\tau,H}} \lambda_\chi(a)^{-\varepsilon'} e^{\rho(\log(a))} \, \dd \mu_{T_H^{\circ}}(a) = \int_{\mathfrak{a}_{\tau,H}} \lambda_\chi(\exp(Y))^{-\varepsilon'} e^{\rho(Y)} \, \dd Y.
\]
Since the group $\bH$ is semisimple, the set of differentials associated to the elements $\beta \in \Delta_H$, which we continue to denote by $\beta : \mathfrak{a}_H \rightarrow \R$, forms a basis of the dual $\mathfrak{a}_H^*$. Via the identification of $\mathfrak{a}_H^*$ with $\mathfrak{a}_H$ (through the choice of an admissible inner product $\langle \cdot, \cdot \rangle$ on $\mathfrak{a}_H$), this gives a basis $(\beta^*)_{\beta \in \Delta_H}$ of $\mathfrak{a}_H$. Let $\cI$ denote the set of $\Q$-weights on $V_\chi$ for the action of $T$. Since $\bV_{\chi}(\Z)$ is spanned by an orthonormal basis consisting of corresponding weight vectors, we observe that, for every $a \in T^\circ$ (in particular, for every $a \in A_{\tau,H}$), $\lambda_\chi(a) = \min_{\widehat{\chi} \in \cI} \, \widehat{\chi}(a)$ and therefore
\[
\lambda_\chi(a)^{-\varepsilon'} = \left ( \min_{\widehat{\chi} \in I} \, \widehat{\chi}(a) \right )^{-\varepsilon'} = \max_{\widehat{\chi} \in I} \, \left ( \widehat{\chi}(a)^{-\varepsilon'} \right ) \leq \sum_{\widehat{\chi} \in I} \widehat{\chi}(a)^{-\varepsilon'}.
\]
Given a character $\widehat{\chi} \in \cI$, let us also denote by $\widehat{\chi} : \mathfrak{a}_H \rightarrow \R$ the corresponding differential and let us express $\widehat{\chi} = \sum_{\beta \in \Delta_H} c_\beta(\widehat{\chi}) \, \beta$ $(c_\beta(\widehat{\chi}) \in \R)$. The claim now amounts to show that for every $\widehat{\chi} \in I$ the integral
\begin{align*}
\int_{\mathfrak{a}_{\tau,H}} \widehat{\chi}(\exp(Y))^{-\varepsilon'} e^{\rho(Y)} \, \dd Y &= \int_{\mathfrak{a}_{\tau,H}} e^{\langle \rho - \varepsilon' \, \widehat{\chi}, Y \rangle} \, \dd Y \\
&= \prod_{\beta \in \Delta_H} \int_{-\infty}^{\ln \tau} e^{\langle \rho - \varepsilon' \, \widehat{\chi} \, , \, x_{\beta^*} \beta^*  \rangle} \, \dd x_{\beta^*} \\
&= \prod_{\beta \in \Delta_H} \int_{-\infty}^{\ln \tau} e^{(n_\beta - \varepsilon' \, c_\beta(\widehat{\chi})) x_{\beta^*}} \, \dd x_{\beta^*}
\end{align*}
converges. By choosing $\varepsilon' > 0$ sufficiently small, the implication $(4) \Rightarrow (5)$ follows.

Since $(5) \Rightarrow (1)$ is immediate, this concludes the proof of Theorem \ref{thm:L1}.
\end{proof}

\subsection{$L^\infty$-integrability}
In this section, we prove Theorem \ref{thm:Linfty}. 

\begin{proof} [Proof of Theorem \ref{thm:Linfty}]
We first show that $(1) \Leftrightarrow (2)$. Assume that $\mathrm{rank}_{\Q} \, \bG = 1$. In particular, the set of simple roots $\Delta = \{\alpha\}$ consists of a single element and the proper parabolic $\Q$-subgroup $\bP = \bP_0$ is minimal. We need to show that $S_{\chi}(B_{c}^{\infty}(\widetilde{X})) \subset  L^{\infty}(\Omega)$. Let $r > 0$ and let $B_{\widetilde{X}}(r) = \{\bm{v} \in \widetilde{X} \, : \, \|\bm{v}\| < r\}$ be the open ball centered at the origin in $\widetilde{X}$ of radius $r$. It suffices to show that $S_{\chi} \mathbbm{1}_{B_{\widetilde{X}}(r)}$ is bounded on $\Omega = G/\G$. Since the maximal $\Q$-split torus $\bT$ is one-dimensional, we have $A = \bT(\R)^{\circ}$. For each $\tau > 0$, let $A_{\tau} = \{a(y) \in A : \alpha(a(y)) = y^{-1} \leq e^{\tau}\}$. We recall that $\bM_0$ denotes the largest $\Q$-anisotropic $\Q$-subgroup of the centralizer $\cZ_{\bG}(\bT)^{\circ}$ in $\bG$ of $\bT$ and $\bU_0$ denotes the unipotent radical of $\bP_0$. By reduction theory of $G$ relative to $\G$, there exists $\tau > 0$, a relatively compact subset $\bm{\omega}$ of $M_0U_0$, and a finite subset $C \subset \bG(\Q)$ such that the Siegel set $\mathfrak{S} = K A_{\tau} \bm{\omega}$ satisfies
\[
G = \mathfrak{S} \, C \, \G.
\]
This allows us to express $g$ as $g = k a(y) n c \gamma$, where $k \in K$, $a(y) \in A_{\tau}$, $n \in \bm{\omega}$, $c \in C$, and $\gamma \in \G$. Using that $B_{\widetilde{X}}(r)$ is $K$-invariant, that $\bigcup_{a(y) \in A_{\tau} } a(y) \bm{\omega} a(y)^{-1}$ is relatively compact, and that $C$ is a finite set contained in $\bG(\Q)$, there exists $R \geq r$, depending on the choice of the Siegel set $\mathfrak{S}$ and the finite set $C$, such that 
\begin{align*}
S_{\chi} \mathbbm{1}_{B_{\widetilde{X}}(r)} (g) = \# ( B_{\widetilde{X}}(r) \cap g \cP_{\chi}) \leq \# ( B_{\widetilde{X}}(R) \cap  a(y) \cP_{\chi}). 
\end{align*}
Hence it suffices to show that the right-hand side is bounded on $A_{\tau}$. This in turn will follow from the fact that all vectors $\bm{v} \in \cP_{\chi}$, different from $\pm \bm{e}_{\chi}$, get expanded under the action of $a(y)$ as $y \rightarrow +\infty$. By assumption, the Weyl group $W \cong \Z / 2\Z$ consists of $2$ elements and we let $w \in N_G(\bT)(\Q)$ be a representative of the non-trivial element in $W$. By~\cite[Th\'eor\`eme~5.15]{BT65}, the Bruhat decomposition $G = P \sqcup Pw P$ holds. Since $w$ acts on $\chi$ by $w \chi = - \chi$, the element $a(y)$ acts on $w \bm{e}_{\chi}$ by 
\[
a(y) w \bm{e}_{\chi} = w w^{-1} a(y) w \bm{e}_{\chi} = w (-\chi)(a(y)) \bm{e}_{\chi} = y^{- \chi(Y_\alpha)} w \bm{e}_{\chi},
\]
and since $\chi(Y_\alpha) < 0$, $w \bm{e}_{\chi}$ gets expanded as $y \rightarrow +\infty$. We may assume that $w \bm{e}_{\chi}$ is an element of the chosen orthonormal basis of $V_{\chi}$, that consists of weight vectors. We first show that any $\bm{v} \in \cP_{\chi}$ different from $\pm \bm{e}_{\chi}$ satisfies $|\langle \bm{v}, w \bm{e}_{\chi} \rangle| \geq 1$. Since $G = P \sqcup P w P$, there exist $p_1, p_2 \in P$ such that $\bm{v} = p_1 w p_2 \bm{e}_{\chi}$. Now, applying $a(y)$ gives
$$
a(y) \bm{v} = a(y) p_1 a(y)^{-1} w w^{-1} a(y) w \chi(p_2) \bm{e}_{\chi} = a(y) p_1 a(y)^{-1} w y^{-\chi( Y_\alpha)} \chi(p_2) \bm{e}_{\chi}.
$$
Noting that $a(y) p_1 a(y)^{-1} w$ converges to some element $p_{\infty} w$ with $p_{\infty} \in P$ as $y$ tends to infinity, we deduce that $\| a(y) \bm{v} \|$ grows at the highest possible rate: $y^{-\chi(Y_{\alpha})}$. Hence using that $\langle \bm{v}, w \bm{e}_{\chi} \rangle \in \Z$, we must have $| \langle \bm{v}, w \bm{e}_{\chi} \rangle | \geq 1$. Therefore, choosing $y$ so large that $y^{-\chi(Y_{\alpha})} > R$, we have for all $\bm{v} \in \cP_{\chi} \smallsetminus \{\pm \bm{e}_{\chi}\}$,
$$
\| a(y) \bm{v} \| \geq \| a(y) w \bm{e}_{\chi} \| = y^{-\chi(Y_{\alpha})} > R.
$$
Hence for all $y$ large enough, $B_{\widetilde{X}}(R) \cap  a(y) \cP_{\chi} = \{\pm \bm{e}_{\chi}\}$, as required. 

Conversely, assume that $\mathrm{rank}_{\Q} \, \bG \geq 2$. We will show that for every radius $r > 1$, the annulus $A(r) = B_{\widetilde{X}}(r) \smallsetminus B_{\widetilde{X}}(r^{-1})$ satisfies that $S_{\chi} \mathbbm{1}_{A(r)} \notin L^{\infty}(\Omega)$. Let $\ka$ be the Lie algebra of $\bT(\R)^\circ$ and let $\theta \subset \Delta$ be the set of simple roots such that $\bP = \bP_{\theta}$. Since $\dim \ka \geq 2$, there exists $Y \in \ka$ and a simple root $\beta \in \Delta \smallsetminus \theta $ such that $\chi(Y) = 0$ and $\beta(Y) > 0$. Let $U_{-\beta}$ be the root subgroup associated to $-\beta$; it acts non-trivially on $\bm{e}_{\chi}$. In particular, the orbit $(U_{-\beta} \cap \G) \bm{e}_{\chi} \subset \cP_{\chi}$ is infinite. For every $y \in \R_+^{\times}$, let $\widetilde{a}(y) = \exp(\ln(y) Y)$. Then, by the choice of $Y \in \ka$, for every $Z \in \ku_{-\beta} = \Lie U_{-\beta}$, 
\[
\widetilde{a}(y) \exp(Z) \bm{e}_{\chi} = \exp(\underbrace{-\beta(Y)}_{< 0} \ln(y) Z) y^{\overbrace{\chi(Y)}^{=0}} \bm{e}_{\chi} \rightarrow \bm{e}_{\chi}, \quad \text{ as $y \rightarrow +\infty$.}
\]
In particular, for every $r > 1$, we have, as required,
\[
S_{\chi} \mathbbm{1}_{A(r)}(\widetilde{a}(y) \G) \geq \# \{ \widetilde{a}(y) (U_{-\beta} \cap \G) \bm{e}_{\chi}  \cap A(r) \} \rightarrow + \infty, \quad \text{ as $y \rightarrow +\infty$.}
\]

Finally, we show that $(2) \Leftrightarrow (3)$. Assume that $\mathrm{rank}_{\Q} \, \bG = 1$. We need to show that $\G_L$ is a cocompact lattice in $L$. By a theorem of Borel and Harish-Chandra~\cite[Theorem~11.6]{BHC62}, the quotient $L / \G_L$ is compact if and only if $X^*(\bL^\circ)_{\Q} = \{1\}$ and $\bL(\Q)$ consists of semisimple elements. Since $\Q$ is of characteristic zero, this is equivalent to the fact that $\bL^{\circ}$ is $\Q$-anisotropic (see just below \cite[Definition~10.5]{Borel69}). Since $\dim \bT = 1$, the parabolic $\Q$-subgroup $\bP = \bP_0$ is minimal and the decomposition in \eqref{eq:Langlands} takes the form
\[
\bP = \bM \, \bA \, \bU,
\]
where $\bM$ is the largest connected $\Q$-anisotropic $\Q$-subgroup of $\cZ_{\bG}(\bT)$, $\bA = \bT$, and $\bU$ is the unipotent radical of $\bP$. In particular, $\bL^\circ = \bM \, \bU$ is the semi-direct product of two $\Q$-anisotropic $\Q$-subgroups and as such $\Q$-anisotropic. 

Conversely, assume that $\G_L$ is a cocompact lattice in $L$. By Theorem~\ref{thm:L1}, the parabolic $\Q$-subgroup $\bP = \bP_{\theta}$ is maximal (and hence $\# \theta = \# \Delta - 1)$. Hence to show that $\mathrm{rank}_{\Q} \, \bG = \# \Delta = 1$, it suffices to show that $\theta = \emptyset$ is the empty set. By the structure of parabolic $\Q$-subgroups as in Section \ref{sec:Structure}, the parabolic $\bP$ contains a connected semisimple $\Q$-subgroup $\bH$ of $\Q$-rank $\# \theta$. In fact, since $X^*(\bH^\circ)_{\Q} = \{1\}$, we have $\bH \subset \bL$. By \cite[Theorem 8.7]{Borel69}, since $\G_L$ is a cocompact lattice in $L$, the $\Q$-group $\bL$ is $\Q$-anisotropic. This implies that $\bH$ is $\Q$-anisotropic and hence that $\# \theta = \mathrm{rank}_{\Q}\, \bH = 0$. The proof of Theorem $\ref{thm:Linfty}$ is complete. 
\end{proof}

\subsection{$L^2$-integrability}

In this section, we prove Theorem \ref{thm:L2}. The following lemma turns the $L^2$-condition into an $L^1$-condition. 
\begin{lemma} \label{lem:Period}
The Siegel transform $S_\chi$ maps $B_{c}^{\infty}(\widetilde{X})$ into $L^2(\Omega)$ if and only if
\[
\forall \, f \in B_{c}^{\infty}(\widetilde{X}) \, \quad \int_{L/\G_L} |S_{\chi} f | \, \dd \mu_{L/\G_L} < + \infty.
\]
\end{lemma}
For any function $f : \widetilde{X} \rightarrow \R$ and $g \in G$, we let $(g \, f) : \widetilde{X} \rightarrow \R$ denote the function 
\[
\forall \, \bm{v} \in V_{\chi}, \quad (g \, f)(\bm{v}) = f(g^{-1} \bm{v}).
\]

\begin{proof}
Suppose first that the Siegel transform $S_\chi$ maps $B_{c}^{\infty}(\widetilde{X})$ into $L^2(\Omega)$ (thus, in particular, into $L^1(\Omega)$). By Theorem \ref{thm:L1}, the parabolic $\Q$-subgroup $\bP$ must be maximal. Fix $f \in B_{c}^{\infty}(\widetilde{X})$; we need to show that $\int_{L/\G_L} |S_{\chi} f | \, \dd \mu_{L/\G_L}$ converges.

Let $\rho \in B_{c}^{\infty}(\widetilde{X})$ be a non-negative $K$-invariant function such that $\rho(\bm{e}_{\chi}) > 0$ and such that for every $k \in K$ and $y \in [1/2, 3/2]$, $|f| \leq ({(k a(y))^{-1}} \, \rho)$.  
By assumption, 
\[
\int_\Omega \left | S_{\chi} \rho \right |^{2}  \, \dd \mu_{\Omega} < +\infty.
\]
On the other hand, using the identification $\widetilde{X} = G/L = K/K_L \times A$ and the corresponding measure description of $\lambda_{\widetilde{X}}$ in Equation \eqref{eq:Measure-On_Tilde-X}, by applying again a standard folding/unfolding argument, we have
\begin{align*}
\int_\Omega \left | S_{\chi} \rho \right |^{2}  \, \dd \mu_{\Omega} &= \int_{G / \G} \sum_{\gamma \in \G / \G_L} \big (\rho(g \gamma \bm{e}_{\chi}) (S_{\chi} \rho)(g \G) \big ) \, \dd \mu_{\Omega}(g \G) \\
&= \int_{G / \G_L} \rho(g \gamma \bm{e}_{\chi}) (S_{\chi} \rho)(g\G) \, \dd \mu_{G / \G_L}(g \G_L) \\
&= \int_{\widetilde{X}} \rho(k a(y) \bm{e}_{\chi})  \, \Bigg (\int_{L/\G_L} S_{\chi} \rho(k a(y) l \G) \, \dd \mu_{L/\G_L}(l \G_L) \Bigg ) \, d \lambda_{\widetilde{X}}(k a(y) \bm{e}_{\chi}).
\end{align*}
Since $\rho(\bm{e}_{\chi}) > 0$ and $\rho$ is $K$-invariant, by continuity there exists $\varepsilon \in (0, 1/2)$ such that $\rho(k a(y) \bm{e}_{\chi}) > 0$ for every $k \in K$ and $y \in [1-\varepsilon, 1+ \varepsilon]$. Since the above integral converges, for $\lambda_{\widetilde{X}}$-almost every $k a(y) \bm{e}_{\chi}$ with $k \in K$ and $y \in [1-\varepsilon, 1+ \varepsilon]$, we have
\[
\int_{L/\G_L} S_{\chi} \rho(k a(y) l\G) \, \dd \mu_{L/\G_L}(l) < + \infty.
\]
Fix such an element $k a(y) \bm{e}_{\chi}$. By construction of $\rho$, since $|f| \leq ({(k a(y))^{-1}} \, \rho)$,
\begin{align*}
\int_{L/\G_L} |S_{\chi} f | \, \dd \mu_{L/\G_L} &\leq \int_{L/\G_L} S_{\chi} ({(k a(y))^{-1}} \, \rho) \, \dd \mu_{L/\G_L} \\
&= \int_{L/\G_L} (S_{\chi} \rho)(k a(y) l \G) \, \dd \mu_{L/\G_L}(l)  < + \infty.
\end{align*}
Let us now prove the other implication. Fix $f \in B_{c}^{\infty}(\widetilde{X})$. 
Since  $\int_{\Omega} |S_{\chi} f|^2 \, \dd \mu_{\Omega} \leq \int_{\Omega} (S_{\chi} |f|)^2 \, \dd \mu_{\Omega}$, we may without loss assume that $f$ is non-negative. Let $\rho \in B_{c}^{\infty}(\widetilde{X})$ be a non-negative function with $\rho(\bm{e}_{\chi}) = 1$. Then, for every $l \in L$, we have $S_{\chi} \rho(l \G) = \sum_{\gamma \in \G/\G_L} \rho(l\gamma \bm{e}_{\chi}) \geq 1$ and thus
\[
\mu_{L/\G_L}(L/\G_L) \leq \int_{L/ \G_L} S_{\chi} \rho (l \G) \, \mu_{L/\G_L}(l \G_L) \, < \, + \infty,
\]
showing that $\G_L$ is a lattice in $L$ and hence,  by Theorem \ref{thm:L1}, that the parabolic $\Q$-subgroup $\bP$ is maximal. In particular, the measure description of $\lambda_{\widetilde{X}}$ in Equation \eqref{eq:Measure-On_Tilde-X} applies again. Using that $f$ is non-negative, by the same argument as before,
\[
\int_\Omega \left | S_{\chi} f \right |^{2}  \, \dd \mu_{\Omega} = \int_{\widetilde{X}} f(k a(y) \bm{e}_{\chi})  \, \Bigg (\int_{L/\G_L} S_{\chi} f(k a(y) l \G) \, \dd \mu_{L/\G_L}(l \G_L) \Bigg ) \, d \lambda_{\widetilde{X}}(k a(y) \bm{e}_{\chi}).
\]
By assumption, the assignment $(k K_L, a(y)) \mapsto \int_{L/\G_L} S_{\chi} f(k a(y) l \G) \, \dd \mu_{L/\G_L}(l \G_L)$ defines a function $K/K_L \times A \rightarrow \R$. Since the support of $f$ is compact, in order to conclude, it suffices to establish the continuity of this function, which is in fact a consequence of Lebesgue's dominated converges theorem.
\end{proof}

Let $W$ be the Weyl group of $\bG$ relative to $\bT$, which we identify with the abstract Weyl group of the root system $\Phi(\bG,\bT)$. By a slight abuse of notation, we also write $w \in \cN_{\bG}(\bT)(\Q)$ for a representative of a Weyl group element $w \in W$. We fix a $W$-invariant inner product $\langle \cdot, \cdot \rangle$ on $X^*(\bT) \otimes \R$ and recall that two distinct simple roots $\beta_1, \beta_2 \in \Delta$ are said to be \emph{neighbors} if $\langle \beta_1 , \beta_2 \rangle \neq 0$. We identify $X^*(\bT) \otimes \R$ with the dual $\mathfrak{a}^*$ of the Lie algebra $\mathfrak{a}$ of $T$. For each $w \in W$, we set $\bL_w = \bL \cap w \bL w^{-1}$.

\begin{proof} [Proof of Theorem \ref{thm:L2}]
So suppose that the Siegel transform $S_{\chi}$ maps $B_{c}^{\infty}(\widetilde{X})$ into $L^2(\Omega)$. Since $L^2(\Omega)$ is contained in $L^1(\Omega)$, Theorem \ref{thm:L1} implies that $\bP$ is a maximal parabolic $\Q$-subgroup. Hence there exists a unique simple root $\alpha \in \Delta$, such that $\bP = \bP_{\Delta \smallsetminus \{\alpha\}}$.  We need to show that, for every $w \in W$, we have $X^*(\bL_w^\circ)_{\Q} = \{1\}$. Let $w \in W$. We may assume that it is represented by an element $w \in \cN_{\bG}(\bT)(\Q)$. In particular, $[w \bm{e}_{\chi}]$ is a rational point of the flag variety $\bX = \bG / \bP$ and there exists $\lambda > 0$ such that $\lambda w \bm{e}_{\chi} \in \cP_{\chi}$. Its stabilizer in $L$ is given by $L_w$. Using Lemma~\ref{lem:Period} and the fact that $\G_L \, \lambda w \bm{e}_{\chi} \subset \cP_{\chi}$, for every non-negative function $f \in B_{c}^{\infty}(\widetilde{X})$,
\[
\int_{L / \G_L} \sum_{\gamma \in \G_L / (\G \cap L_w)} f(l \gamma \lambda w \bm{e}_{\chi}) \, \dd \mu_{L/\G_L} \leq \int_{L/\G_L} S_{\chi} f \, \dd \mu_{L/\G_L} < + \infty.
\]
Therefore,
$$
\Lambda_{w}(f) = \int_{L / \G_L} \sum_{\gamma \in \G_L / (\G \cap L_w)} f(l \gamma \lambda w \bm{e}_{\chi}) \, \dd \mu_{L/\G_L}
$$ 
defines a positive $L$-invariant linear functional on $B_{c}^{\infty}(\widetilde{X})$, inducing a unique $L$-invariant Radon measure $\lambda_{w}$ on $\widetilde{X}$, whose support is contained in the closure of the orbit $L \, \lambda w \bm{e}_{\chi}$. Therefore $L\, \lambda w \bm{e}_{\chi} = L/L_w$ carries an $L$-invariant Radon measure and since $L$ is unimodular, this implies that $L_w$ is also unimodular. Hence, applying again a standard folding/unfolding argument, there exists an $L$-invariant Radon measure $\mu_{L/L_w}$ on $L/L_w$ that can be normalized such that, for every non-negative $f \in B_{c}^{\infty}(\widetilde{X})$,
\begin{align*}
\Lambda_{w}(f) &= \int_{L/L_w} \int_{L_w/(\G \cap L_w)} f (l l' \lambda w \bm{e}_{\chi}) \, \dd \mu_{L_w/(\G \cap L_w)}(l' (\G \cap L_w)) \dd \mu_{L/L_w}(l L_w) \\
&= \mu_{L_w/(\G \cap L_w)}(L_w/(\G \cap L_w)) \int_{L/L_w} f (l \lambda w \bm{e}_{\chi}) \, \dd \mu_{L/L_w}(l L_w).
\end{align*}
It follows that $\G \cap L_w$ is a lattice in $L_w$. Thus, by a theorem of Borel and Harish-Chandra~\cite[Theorem 9.4]{BHC62}, we have $X^*(\bL_w^\circ)_{\Q} = \{1\}$, as required. 

Let us now show the last claim in Theorem~\ref{thm:L2}.
Denote the set of neighbors of $\alpha$ by $V(\alpha)$ and set $B(\alpha) = V(\alpha) \cup \{\alpha\}$. Let $s_{\alpha} \in W$ be the reflection across the hyperplane defined by $\alpha$. We claim that 
\begin{equation} \label{eq:Inclusion}
\bL_{s_\alpha} = \bL \cap s_{\alpha} \bL s_{\alpha}^{-1} \subset \bP_{\Delta \smallsetminus B(\alpha)}.
\end{equation}
Since $\bL_{s_\alpha}$ is normalized by $\bT$, its Lie algebra decomposes as a direct sum of a maximal toral subalgebra and root spaces. For every $\beta \in \Phi$, we let 
\[
\kg_{\beta} = \{Y \in \kg : \forall \, t \in T, \, \Ad(t) Y = \beta(t) Y \}
\]
and $\mathfrak{z} = \{Y \in \kg : \forall \, t \in T, \, \Ad(t) Y = Y \}$. Let $\langle \Delta \setminus B(\alpha) \rangle^-$ be the set of all roots that can be expressed as a negative integer linear combination of elements in $\Delta \setminus B(\alpha)$. Then the Lie algebra of $P_{\Delta \smallsetminus B(\alpha)}$ is
\[
\kp_{\Delta \smallsetminus B(\alpha)} = \mathfrak{z} \oplus \left ( \bigoplus_{\beta \in \Phi^+ \cup \langle \Delta \smallsetminus B(\alpha) \rangle^-} \kg_{\beta} \right ).
\]
Let us denote by $\Phi(L_{s_{\alpha}}, T)$ the set of roots arising in the adjoint representation of $T$ on the Lie algebra of $L_{s_\alpha}$. The inclusion~\eqref{eq:Inclusion} follows once we prove that every negative root 
\begin{equation} \label{eq:Root_Lambda}
\lambda = \sum_{\beta \in \Delta} n_{\beta} \beta \in \Phi(L_{s_{\alpha}}, T)    
\end{equation}
has $n_{\beta} = 0$ for all $\beta \in B(\alpha)$. Since $s_{\alpha}$ normalizes $L_{s_\alpha}$, we have $s_{\alpha}(\lambda) \in \Phi(L_{s_{\alpha}}, T)$. Moreover, since $L_{s_{\alpha}} \subset P_{\Delta \smallsetminus \{\alpha\}}$, we must have $n_{\alpha} = 0$, and hence the sum in \eqref{eq:Root_Lambda} is in fact a sum over $\Delta \smallsetminus \{\alpha\}$. Therefore, applying $s_{\alpha}$ to \eqref{eq:Root_Lambda} yields 
\[
s_{\alpha}(\lambda) = \sum_{\beta \in \Delta \smallsetminus \{\alpha\}} n_{\beta} s_{\alpha}(\beta) = \left( \sum_{\beta \in \Delta \smallsetminus \{\alpha\}} n_{\beta} \beta \right ) - \left ( \sum_{\beta \in V(\alpha)} \frac{2 \langle \alpha, \beta \rangle}{ \langle \alpha, \alpha \rangle} n_\beta \right) \alpha.
\]
For every $\beta \in V(\alpha)$, one has $\langle \alpha, \beta \rangle < 0$. Therefore, we have $n_{\beta} = 0$ for all $\beta \in B(\alpha)$, proving the inclusion~\eqref{eq:Inclusion}. Moreover, the $\Q$-rank of $\bL_{s_\alpha}$ is at least $\mathrm{rank}_{\Q} \, \bG-2$, since it contains the $\Q$-split torus $\bT \cap \ker(\chi) \cap \ker(s_{\alpha} \chi)$. On the other hand, the $\Z$-rank of $X^*(\bP_{\Delta \setminus B(\alpha)})_{\Q}$ equals $|B(\alpha)|$. If $|V(\alpha)| > 1$ (and hence $|B(\alpha)| > 2$), then $\bL_{s_\alpha}$ would admit a non-trivial $\Q$-character. Hence we must have $|V(\alpha)| \leq 1$, as required. The proof of Theorem~\ref{thm:L2} is complete.
\end{proof}

\section{Equidistribution of maximal compact subgroup orbits} \label{sec:Equi-compact}

In this section, we prove Theorem~\ref{thm:Effective}, which we will use in the next section, together with our integrability criterion (Theorem~\ref{thm:L1}), to derive the application (Theorem~\ref{thm:Critical}) to metric Diophantine approximation on rank-one flag varieties. The proof is inspired by the approach in \cite[Proposition~4.1]{Ouaggag23}, where Ouaggag proves a similar result in the special case where $G = \SO(n,1)$ is the special orthogonal group of signature $(n,1)$. 

Contrary to the convention used in the rest of this paper, this section is kept slightly more general. In particular, as in Theorem~\ref{thm:Effective}, we let $A = \{a(y) : y \in \R_+^\times\}$ denote an arbitrary one-parameter $\Ad$-diagonalizable subgroup of $G$, that is, $\Ad(A)$ is diagonal with respect to some basis of the Lie algebra $\kg$ of $G$. We let $K$ be a maximal compact subgroup and denote by $\mu_K$ its Haar probability measure. The parabolic subgroup $P \subset G$ associated to $A$ is given by
\[
P = \left \{ g \in G : \text{the limit $\lim_{y \rightarrow + \infty} a(y) \, g \, a(y)^{-1} \in G$ exists} \right \}.
\]
Then, $P = \cZ_G(A) \ltimes U$ is the semi-direct product of the centralizer of $A$ in $G$ and the unipotent radical $U$ of $P$, which is also the  stable horospherical subgroup with respect to $A$:
\[
U = \left \{ g \in G : \lim_{y \rightarrow + \infty} a(y) \, g \, a(y)^{-1} = \mathrm{Id} \right \}.
\]
The unstable horospherical subgroup $U^- \subset G$ with respect to $A$ is given by
\[
U^- = \left \{ g \in G : \lim_{y \rightarrow 0} a(y) \, g \, a(y)^{-1} = \mathrm{Id} \right \}.
\]
We remark that $G$ decomposes as $G = KP$. Indeed, by the Iwasawa decomposition of $G$  there exists a minimal parabolic subgroup $P_0'$ of $G$ such that $G = K P_0'$ and, by definition, $P$ contains a conjugate $g P_0' g^{-1}$ of $P_0'$. Express $g = k p_0'$ with $k \in K$ and $p_0' \in P_0'$. Then we have 
\begin{equation} \label{eq:G=KP}
K P \supseteq K (g P_0' g^{-1}) = K (k P_0' k^{-1}) = KP_0'k^{-1} = G.
\end{equation}
Let $\mu_{U^-}$ be a Haar measure on $U^-$ and let $d_{U^-}(\cdot, \cdot)$ be a right-invariant metric on $U^-$, induced from the right-invariant metric $d_G(\cdot, \cdot)$ on $G$. For every $r_0 > 0$, we let $B_{U^-}(r_0) \subset U^-$ denote the open ball around $0$ in $U^-$ with radius $r_0$.

Let us introduce Sobolev norms on the spaces $C_c^\infty(\Omega)$, $C_c^\infty(U^-)$, and $C^\infty(K)$. Each element $Z$ in the Lie algebra $\kg$ of $G$ defines a first order differential operator $\cD_Z$ on $C_c^\infty(\Omega)$ by 
\[
\forall \, \phi \in C_c^{\infty}(\Omega), \, \forall \, x \in \Omega, \quad \cD_Z \phi(x) = \frac{d}{dt}\Big |_{t = 0} \phi(\exp(t Z)x).
\]
Let $D$ be the dimension of $\kg$ and let $\cB = (Z_i)_{1 \leq i \leq D}$ be a basis of the real vector space $\kg$. Then each monomial 
\begin{equation} \label{eq:Differential_Operator}
\cD_Z = \cD_{Z_1}^{j_1} \circ \cdots \circ \cD_{Z_{D}}^{j_{D}}
\end{equation}
with $(j_1, \dots, j_{D}) \in \N^{D}$ defines a differential operator of degree $\deg(\cD_Z)= \sum_{i = 1}^D j_i$. For all $r \geq 1$ and $\phi \in C_c^\infty(\Omega)$, we define the {\it degree $r$ Sobolev norm} of $\phi$ by 
\begin{equation} \label{eq:Norm}
\cS_r(\phi) = \sum_{\deg(\cD) \leq r} \left \|  \cD \phi \right \|_{\infty},
\end{equation}
where $\cD$ ranges over all monomials of elements in $\cB$ of degree $\leq r$. Similarly, one can define degree $r$ Sobolev norms on $C^\infty(K)$ and $C_c^\infty(U^-)$, which, by abuse of notation, we also denote by $\cS_r$.

For the convenience of the reader, we record here the effective single and double equidistribution result for expanding horospherical translates, which is a consequence of an effective multiple equidistribution result of Shi \cite[Theorem~1.5]{Shi21}) and from which we derive Theorem~\ref{thm:Effective}. 

\begin{thm} 
\label{thm:Shi21} 
Suppose that the action of $G$ on $\Omega = G/\G$ has a spectral gap. Let $A$ and $U^-$ be as above and suppose that the projection of $a(e)$ to each simple factor of $G$ is not the identity element. 
Then there exists a constants $c' = c'(G,\Omega,A) > 0$ and an integer $r \geq 1$ such that the following holds. For every $r_0 > 0$ and compact subset $Q \subset \Omega$, there exists a constant $C' = C'(r_0,Q) > 0$ such that for all $f \in C_c^{\infty}(U^-)$ with $\supp(f) \subset B_{U^-}(r_0)$, $\phi \in C_c^\infty(\Omega)$, $x \in Q$, and $y > 1$, we have
\begin{equation} \label{ml:single-Eq}
\left | \, \int_{U^-} f(u^-) \phi \bigl(a(y) u^- x \bigr) \dd \mu_{U^-}(u^-) - \mu_{U^-} (f) \mu_{\Omega}(\phi) \, \right | \, \leq \, C' \, y^{-c'} \cS_r(f) \cS_r(\phi),
\end{equation} 
and for all $f \in C_c^{\infty}(U^-)$ with $\supp(f) \subset B_{U^-}(r_0)$, $\phi_1, \phi_2 \in C_c^\infty(\Omega)$, $x_1, x_2 \in Q$, and $y_2 \geq y_1 > 1$, we have
\begin{multline} \label{ml:Double-Eq}
\left | \, \int_{U^-} f(u^-) \phi_{1}\bigl(a(y_1) u^- x_1 \bigr) \phi_{2} \bigl(a(y_2)  u^- x_2 \bigr) \dd \mu_{U^-}(u^-) - \mu_{U^-} (f) \mu_{\Omega}(\phi_{1}) \mu_{\Omega}(\phi_{2}) \, \right | \\ \leq \, C' \, \min \{y_1, y_2 / y_1 \}^{-c'} \cS_r(f) \cS_r(\phi_1) \cS_r(\phi_2).
\end{multline} 
\end{thm}

\begin{proof} [Proof of Theorem~\ref{thm:Effective}]
Let us prove the effective double equidistribution statement in \eqref{eq:Double-Eq}. The proof of the effective single equidistribution statement in \eqref{eq:Single-Eq} follows along the same line of argument and we omit the details.

Let $\mathfrak{k}$ be the Lie algebra of $K$ and let $\mathfrak{k}_P$ be that of $K_P = K \cap P$. Choose a subspace $\mathfrak{s} \subset \mathfrak{k}$ such that $\mathfrak{k} = \mathfrak{s} \oplus \mathfrak{k}_P$. Let $V$ be a small neighborhood of the origin in $\mathfrak{g}$ such that the exponential map $\exp : \mathfrak{g} \rightarrow G$ restricts to a diffeomorphism $\exp |_V : V \rightarrow \exp(V)$ and consider the embedded submanifold $S = \exp(V \cap \mathfrak{s})$ of $\exp(V)$. In order to relate the double equidistribution of translated $K$-orbits to that of translated horospherical orbits, we first construct a local diffeomorphism from $S$ to $U^-$. 

Let $\mu_{K_P}$ be the Haar probability measure on $K_P$. The product map $S \times K_P  \rightarrow K$ restricts to a local diffeomorphism, giving a decomposition $k= s(k) k_P(k)$ with $s(k) \in S$ and $k_P(k) \in K_P$ for every $k \in K$ close to $\mathrm{Id}$. Moreover, there is a smooth measure $\nu$ on a neighborhood of $\mathrm{Id} \in S$ such that for all bounded integrable functions $f' : K \rightarrow \R$ supported in a sufficiently small neighborhood of $\mathrm{Id} \in K$,  
\[
\int_K f'(k) \, \dd \mu_K(k) = \int_{S} \int_{K_P} f'(s k_P) \, \dd \mu_S(s) d \mu_{K_P}(k_P).
\]
We recall that $\ku^-$ and $\mathfrak{p}$ denote the Lie algebras of $U^-$ and $P$ respectively and that the Lie algebra of $G$ decomposes as the direct sum $\mathfrak{g} = \ku^- \oplus \mathfrak{p}$. Thus every element $g \in G$ sufficiently close to the identity $\mathrm{Id} \in G$ can be uniquely decomposed as $g = u^-(g) p(g)$ with $u^-(g) \in U^-$ and $p(g) \in P$. In particular, this induces a coordinate map $h : g \mapsto u^-(g)$ that is defined on a small neighborhood of $\mathrm{Id} \in G$. Let us denote by $D_{\mathrm{Id}} h$ its derivative at $\mathrm{Id} \in G$ and observe that $\ker D_{\mathrm{Id}} h = \mathfrak{p}$. We claim that there exists a small neighborhood $V_S$ of the identity $\mathrm{Id} \in S$ such that the restriction $h|_{V_S} : V_S \rightarrow U^-$ defines a diffeomorphism onto a neighborhood of the identity in $U^-$. By \eqref{eq:G=KP}, we know that $G = KP$. Thus, the dimensions of $S$ and $U^-$ agree
\[
\dim S = \dim K / K_P = \dim G/P = \dim U^-.
\]
To show the claim, by the inverse function theorem, it suffices to show that the restriction $D_{\mathrm{Id}} h|_{\mathfrak{s}}$ is injective. This however follows from the fact that 
\[
\ker (D_{\mathrm{Id}} h|_{\mathfrak{s}}) = \mathfrak{s} \cap \mathfrak{p} = \{0\}.
\]

Let $r \geq 1$ be as in Theorem \ref{thm:Shi21}. There exist constants $c_1 > 0$ and $c_2 > 0$ such that the following holds. For every small $r_0 > 0$, there exists $N \in \N$ with $N \ll r_0^{-c_1}$, non-negative functions $\kappa_i \in C_c^{\infty}(K)$ with $1 \leq i \leq N$, all supported in $B_K(r_0)$ and satisfying  $\cS_r( \kappa_i) \ll r_0^{-c_2}$, and elements $k_i \in K$ with $1 \leq i \leq N$ such that we have a partition of unity: for every $k \in K$, we have $1= \sum_{i=1}^{N} \kappa_i (k k_i^{-1})$.

For all $f \in C^{\infty}(K)$, $\phi_1, \phi_2 \in C_c^\infty(\Omega)$ $x \in Q$, and $y_2 \geq y_1 \geq 0$, we define
\[
I_{f,\phi_1,\phi_2}(y_1,y_2,x_1,x_2) = \int_K f(k) \prod_{j=1}^2 \phi_{j} \bigl(a(y_j) k x_j \bigr) \dd \mu_K(k),
\]
where, for brevity, we have introduced the product symbol $\prod_{j=1}^2$. Let $r_0 > 0$ be small, to be fixed later. Using the direct sum decompositions
\[
\mathfrak{k} = \mathfrak{k}_P \oplus \mathfrak{s} \quad \text{and} \quad \mathfrak{g} = \mathfrak{p} \oplus \mathfrak{u}^{-},
\]
for every $k$ close the identity in $K$, we have 
\[
k = k_P(k) s(k) = k_P(k) p(s(k)) u^-(s(k)). 
\]
Observe that $K_P$ is the centralizer $\mathcal{Z}_K(A)$ in $K$ of $A$. Putting everything together and letting, for $i =1,2$, 
\[
g(k,y_i) = k_P(k) a(y_i) p(s(k)) a(y_i)^{-1},
\]
we have
\begin{align*}
&I_{f,\phi_1,\phi_2}(y_1,y_2,x_1,x_2) = \sum_{i = 1}^N \int_K \kappa_i(k) f(k k_i) \prod_{j=1}^2 \phi_{j} (a(y_j) k k_i x_j)  
\dd \mu_K(k) \\
&= \sum_{i = 1}^N \int_K \kappa_i(k) f(k k_i) \prod_{j=1}^2 \phi_{j}(g(k,y_j) a(y_j) u^-(s(k)) k_i x_j) \dd \mu_K(k).
\end{align*}
We recall that $P$ is the semi-direct product of the centralizer $\mathcal{Z}_G(A)$ in $G$ of $A$ and the contracting horospherical subgroup with respect to $A$:
\[
U = \left \{ g \in G : \lim_{y \rightarrow + \infty} a(y) \, g \, a(y)^{-1} = I \right \}.
\]
In particular, for every $y \geq 1$, elements $p \in P$ do not get expanded by the conjugation action of $a(y)$. By Lipschitz continuity of the coordinate maps $k \mapsto k_P(k)$, $k \mapsto p(s(k))$ on $B_K(r_0)$ with $r_0 > 0$ small enough, there exists a constant $C_1 > 0$, independent of $y_1$ and $y_2$, such that for every $k \in B_K(r_0)$, we have
\[
k_P(k), \, a(y_i) p(s(k)) a(y_i)^{-1} \in B_G(C_1 r_0).
\]
By the Lipschitz continuity of $\phi_1$ and $\phi_2$, we have
\begin{multline*}
\left| I_{f,\phi_1,\phi_2}(y_1,y_2,x_1,x_2) - \sum_{i = 1}^N \int_K \kappa_i(k) f(k k_i) \prod_{j=1}^2\phi_j(a(y_j) u^-(s(k)) k_i x_j) \dd \mu_K(k) \right | 
\\ \ll_r \, r_0 \, \cS_r(f) \cS_r(\phi_1) \cS_r(\phi_2).
\end{multline*}
Recall that there exists a neighborhood $V_S$ of the identity $1 \in S$ such that the map $h : V_S \rightarrow U^-$ given by sending $s \in S$ to $u^-(s)$ defines a diffeomorphism  onto a neighborhood of the identity $1 \in U^-$. Hence, denoting by $u^- \mapsto s(u^-)$ the local inverse of this diffeomorphism, there exists a smooth density $\rho_0$ defined in a neighborhood of $1 \in U^-$ such that for all sufficiently small $r_0 > 0$ and all $h \in C_c(S)$ supported in $B_S(r_0)$, we have
\begin{equation*}
\int_S h(s) \, \dd \mu_S(s) = \int_{U^-} h(s(u^-)) \rho_0(u^-) \, \dd \mu_{U^-}(u^-).
\end{equation*}
Using the local decomposition of the measure $\mu_K$ as a product of the measures $\mu_{K_P}$ and $\mu_S$, we have
\begin{align}
&\sum_{i = 1}^N \int_K \kappa_i(k) f(k k_i) \prod_{j=1}^2 \phi_{j} (a(y_j) u^-(s(k)) k_i x_j) 
\dd \mu_K(k) \nonumber \\
&=~\sum_{i=1}^N \int_{K_P} \int_S \kappa_i(k_P s) f(k_P s k_i) \prod_{j=1}^2 \phi_{j} (a(y_j) u^-(s) k_i x_j) 
\dd \mu_S(s) d\mu_{K_P}(k_P) \nonumber \\
&= \int_{K_P} \sum_{i=1}^N \left( \int_{U^-} f_{k_P,i}(u^-) \prod_{j=1}^2 \phi_j (a(y_j) u^- k_i x_j) 
\dd \mu_{U^-}(u^-) \right)d\mu_{K_P}(k_P), \label{integral}
\end{align}
where in the last line we abbreviated 
\[
f_{k_P,i}(u^-) := \kappa_i(k_P s(u^-)) f(k_P s(u^-) k_i) \rho_0(u^-).
\]
which is defined in a neighborhood $V$ of the identity in $U^-$ with smooth boundary and compact closure. 
By Theorem \ref{thm:Shi21}, applied to the functions $f_{k_P,i}$, $\phi_1$, $\phi_2$ and to the compact subset $KQ \subset \Omega$,
there exist constants $c' > 0$, $C' > 0$, independent of $f_{k_P,i}$, $\phi_1$, $\phi_2$, such that we have
\begin{multline*}
\left | \int_{U^-} f_{k_P,i}(u^-) \prod_{j=1}^2 \phi_j(a(y_1) u^- k_i x_j) \dd \mu_{U^-}(u^-) - \mu_{U^-} (f_{k_P,i}) \mu_\Omega(\phi_1) \mu_\Omega(\phi_2) \right | \\
\leq \, C' \, \min \{y_1,y_2/y_1 \}^{-c'} \cS_r(f_{k_P,i}) \cS_r(\phi_1) \cS_r(\phi_2). 
\end{multline*}
Observing that $\cS_r(f_{k_P,i}) \ll \cS_r(\kappa_i) \cS_r(f) \cS_r(\rho_0)$, that $N \leq r_0^{-c_1}$, that $\cS_r(\kappa_i) \ll r_0^{-c_2}$, that $\cS_r(\rho_0|_{B_{U^-}(r_0)}) \ll 1 $, and that 
\[
\int_{K_P} \sum_{i=1}^N \mu_{U^-} (f_{k_P,i}) d\mu_{K_P}(k_P) = \int_{K} \sum_{i=1}^N \kappa_i(k) f(k k_i) \, \dd \mu_{K}(k) = \mu_K(f),
\]
the integral \eqref{integral} is equal to
\[
\mu_K(f) \mu_\Omega(\phi_1) \mu_\Omega(\phi_2) + O\left( \min \{y_1,y_2/y_1 \}^{-c'} \, r_0^{-c_1 - c_2} \, \cS_r(f) \cS_r(\phi_1) \cS_r(\phi_2) \right).
\]
Hence, putting everything together, we have
\begin{multline*}
I_{f,\phi_1,\phi_2}(y_1,y_2,x_1,x_2) = \mu_K(f) \mu_\Omega(\phi_1) \mu_\Omega(\phi_2) \\
+ O\left( (\min \{y_1,y_2/y_1 \}^{-c'} \, r_0^{-c_1 - c_2} + r_0) \, \cS_r(f)\cS_r(\phi_1) \cS_r(\phi_2) \right).
\end{multline*}
Setting $c = \frac{c'}{1+c_1+c_2}$ and $r_0 = \min \{y_1,y_2/y_1 \}^{-c}$ completes the proof of Theorem~\ref{thm:Effective}.
\end{proof}

\section{Application to Diophantine approximation on flag varieties}

In this section, we prove Theorems \ref{thm:Critical}. Let us briefly recall the setting and explain, why the effective equidistribution Theorem~\ref{thm:Effective} applies in this case. Let $\bX$ be a generalized flag variety defined over $\Q$, obtained as the quotient $\bX = \bG / \bP$ of a connected simply-connected almost $\Q$-simple $\Q$-group $\bG$ by a proper parabolic $\Q$-subgroup $\bP$ with abelian unipotent radical. In particular, the $\Q$-rank of $\bX$ is $1$ and there exists a unique simply root $\alpha \in \Delta$ such that $\bP = \bP_{\Delta \smallsetminus \{\alpha\}}$. As before, we let $Y$ be the unique element in the Lie algebra of $\bT(\R)$ such that 
\[
\alpha(Y) = -1 \quad \text{and} \quad \beta(Y) = 0 \quad \text{for all } \beta \in \Delta \smallsetminus \{\alpha\}.
\]
We assumed that the element $\exp(Y)$ projects non-trivially to each simple factor of $G$. For all $y \in \R_+^\times$, we let $a(y) := \exp( \ln(y) Y)$. Since $\bG$ is assumed to be almost $\Q$-simple and $\Q$-isotropic (which follows from the assumption that $\bP$ is a proper parabolic $\Q$-subgroup of $\bG$), the group $G$ has no non-trivial compact factors (since any such factor would be contained in a $\Q$-anisotropic $\Q$-simple factor of $\bG$). Moreover, since $\bG$ is almost $\Q$-simple, the arithmetic subgroup $\G \subset \bG(\Q)$ is an irreducible lattice in $G$. Thus, by \cite{KS09}, the action of $G$ on $\Omega = G/\G$ has a spectral gap, and Theorem~\ref{thm:Effective} applies to our setting. 

Let $K$ be a maximal compact subgroup of $G$. We refer the reader to the introductory Section~\ref{sec:prel} for the specification of the height function $H_{\chi}$ on $\bX(\Q)$ associated to a dominant $\Q$-weight $\chi$, the $K$-invariant probability measure $\sigma_{X}$ and the $K$-invariant Riemannian distance $d(\cdot, \cdot)$ on $X$. Moreover, we recall that $\beta_{\chi} > 0$ denotes the Diophantine exponent of $X$ relative to $\chi$ \eqref{eq:DiophantineExpo}.

Let $c > 0$. Our goal is to determine the asymptotic behavior as $T \rightarrow + \infty$, for $\sigma_{X}$-almost every $x \in X$, of the counting function 
\[
\cN_{c,\beta_{\chi}}(x,T) = \# \{ v \in \bX(\Q) : d(x,v) < c \, H_{\chi}(v)^{-\beta_{\chi}}, \, 1 \leq H_{\chi}(v) < T\}.
\]
For simplicity, we will work with $c = 1$ and we write $\cN_{\beta_{\chi}}(x,T)$ for $\cN_{1,\beta_{\chi}}(x,T)$. 

\subsection{Diophantine approximation and counting lattice points} \label{sec:Reduction}
Let us first translate the problem of counting rational approximations of bounded height in $X$ to the problem of counting primitive lattice points in a certain family of growing sets in the cone $\widetilde{X}$ over $X$. We may assume that the lattice $\bV_{\chi}(\Z)$ is spanned over $\Z$ by an orthonormal basis of $V_{\chi}$ consisting of $\Q$-weight vectors, and that $\G$ is its stabilizer in $G$. 

We recall that $x_0 = [\bm{e}_{\chi}]$ stands for the line through the highest weight vector $\bm{e}_{\chi}$. For every $T \geq 1$, we define the set 
\begin{equation} \label{eq:Definition_E}
\cE_{\beta_{\chi}}(T) = \big \{ \bm{v} \in \widetilde{X} : d(x_0, [\bm{v}]) < \|\bm{v}\|^{-\beta_{\chi}}, 1 \leq \|\bm{v}\| < T \big \}.
\end{equation}
Fix a section $\ks : X  \rightarrow K$ of the orbital map $K \to X$ sending $k$ to $kx_0$. Given $x \in X$, we shall write $k_x = \ks(x)$.
As the following lemma shows, estimating the counting function $\cN_{\beta_{\chi}}(x,T)$ amounts to counting lattice points in the increasing family $\{\cE_T\}_{T \geq 1}.$ Let $[K \cap P : K \cap L] \in \{1,2\}$ be the index of $K \cap L$ in $K \cap P$.
\begin{lemma} \label{lem:Reduction}
For every $x \in X$ and $T \geq 1$, we have
\begin{equation} \label{eq:Reduction}
\cN_{\beta_{\chi}}(x,T) = [K \cap P : K \cap L]^{-1} \, \# \left (k_x^{-1} \cP_{\chi}  \cap \cE_{\beta_{\chi}}(T) \right ).
\end{equation}
\end{lemma}

\begin{proof}
It suffices to show that $\cN_{\beta_{\chi}}(x,T) = \# \left (\cP_{\chi}  \cap k_x \cE_{\beta_{\chi}}(T) \right )$. Using the $K$-invariance of the norm $\|\cdot \|$ and the distance $d ( \cdot, \cdot)$ on $X$, we have
$$
k_x \cE_{\beta_{\chi}}(T) = \left \{ \bm{v} \in \widetilde{X} : d(x, [\bm{v}]) < \|\bm{v}\|^{-\beta_{\chi}}, 1 \leq \|\bm{v}\| < T \right \}.
$$
By the definition of the height function $H_{\chi}$, a rational point $v \in \bX(\Q)$ satisfies $d(x, v) < H_\chi(v)^{-\beta_{\chi}}$ and $1 \leq H_\chi(v) < T$ if and only if any of the primitive vectors $\bm{v} \in \cP_{\chi}$ representing $v$ satisfy $d(x, [\bm{v}]) < \|\bm{v}\|^{-\beta_{\chi}}$ and $1 \leq \|\bm{v}\| < T$. This finishes the proof of the lemma.
\end{proof}

\subsection{The geometry of $\cE_{\beta_{\chi}}(T)$ and the diagonal action of $A$} \label{sec:Approx_E_T}

In this section, we observe that the set $\cE_{\beta_{\chi}}(T)$ can be approximated from inside and from outside by sets that admit a convenient decomposition under the action of the one-parameter diagonal subgroup $A = \{a(y) : y \in \R_+^\times \}\subset G$. 

Let us now define these sets that approximate $\cE_{\beta_{\chi}}(T)$. We recall from Section \ref{sec:Distance} that the map $\ku^- \rightarrow X$ sending $u \mapsto \exp(u) x_0$ restricts to a diffeomorphism from a neighborhood of $1 \in \ku^-$ to a neighborhood of $x_0 \in X$. In particular, any $\bm{v} \in \widetilde{X}$, such that $[\bm{v}]$ is close to $x_0$, defines an element $u_{\bm{v}}^-$ in the Lie algebra $\ku^-$ by $[\bm{v}] = \exp(u_{\bm{v}}^-) x_0$. The adjoint action of $a(y) \in A$ on $\ku^- = T_{x_0} X$ acts by scalar multiplication: for all $y \in \R_+^\times$, $\Ad(a(y)) u^- = y \, u^-$. Observe that 
\[
[a(y) \bm{v}] = a(y) [\bm{v}] =  a(y) \exp(u_{\bm{v}}^-) a(y) a(y)^{-1} x_0 = \exp(\Ad(a(y)) u_{\bm{v}}^-) x_0.
\]
But we also have $[a(y) \bm{v}] = \exp(u_{a(y) \bm{v}}^- ) x_0$. By uniqueness, this gives the relation
\begin{equation} \label{eq:LieAlgebraDiagonal}
u_{a(y) \bm{v}}^- = y \, u_{\bm{v}}^-.
\end{equation}
Moreover, by the distance estimate \eqref{eq:Distance_Estimate}, there exists a constant $C_0 > 0$ such that 
\[
\left | d(x_0, [\bm{v}]) - \|u_{\bm{v}}^-\|_{\ku^-} \right | \leq C_0 \, \|u_{\bm{v}}^-\|_{\ku^-}^2.
\]
Let $\pi^+ : V_\chi \rightarrow V_\chi$ be the orthogonal projection onto $\R\bm{e}_{\chi}$ and we abbreviate $\pi^+(\bm{v})$ by $\bm{v}^+$. For every $T \geq 1$ and $c > 0$ close to $1$, we will work with the sets
\[
\cE_{T, c}^+ = \{\bm{v} \in \widetilde{X} : \|u_{\bm{v}}^-\|_{\ku^-} < c \, \|\bm{v}^+\|^{-\beta_{\chi}}, 1 \leq \|\bm{v}^+\| < c \, T \}.
\] 
By enlarging $C_0$ if necessary, we can assume that $\|\bm{v}^+\| \geq C_0^{-1} \|\bm{v}\|$ as soon as $d(x_0, [\bm{v}]) < 1$. For every integer $\ell \geq 1$, we let 
\[
Q_{\ell} = \{\bm{v} \in \widetilde{X} : \|\bm{v}\| \leq C_0 \, \ell \}
\]
and we define
\[
\widehat{c}_{\ell} = \left ( 1 + C_0 \, \ell^{-\beta_{\chi}}  \right )^{-(1+\beta_{\chi})} \in (0,1).
\]
In particular, we have $\widehat{c}_{\ell} \nearrow 1$ as $\ell \rightarrow + \infty$.

The sets $\cE_{T,c}^+$ have the following nice properties. For every $c > 0$, let
\begin{equation} \label{eq:F_c_New}
\cF_c = \{\bm{v} \in \widetilde{X} : \|u_{\bm{v}}^-\|_{\ku^-} < c  \|\bm{v}^+\|^{-\beta_{\chi}}, 1 \leq \|\bm{v}^+\| < e \}.
\end{equation}

\begin{lemma} \label{lem:Sandwich}
\begin{enumerate}
\item (Approximation) For all large enough $\ell \geq 1$ and $T \geq 1$,
\begin{equation} \label{eq:lem:Sandwich}
    \cE_{T, \widehat{c}_{\ell}}^+ \smallsetminus Q_{2 \ell} \, \subseteq \, \cE_{\beta_{\chi}}(T) \smallsetminus Q_{\ell} \, \subseteq \, \cE_{T, \widehat{c}_{\ell}^{-1}}^+.
\end{equation}
\item (Decomposition) For every $c > 0$ and $T \geq 1$ such that $c T = e^N$ for some $N \in \N$, we have
\begin{equation} \label{eq:decomposition}
    \cE_{T,c}^+ = \bigsqcup_{i=0}^{N-1} a(y_i)^{-1} \, \cF_c, \quad \quad \text{with $y_i = e^{\beta_{\chi} i}$ for $j \in \N$. }
\end{equation}
\end{enumerate}
\end{lemma}

\begin{proof}
Let $\bm{v} \in \cE_{ T, \widehat{c}_{\ell}}^+ \smallsetminus Q_{2 \ell}$. Let us first prove that for all sufficiently large $\ell$ and $T$, we have
\[  
d(x_0, [\bm{v}]) < \|\bm{v}\|^{-\beta_{\chi}}, \, C_0 \, \ell < \| \bm{v}\| < T.
\]
Let $\bm{v}^{\perp} = \bm{v} - \bm{v}^+$. Observe that, for large enough $\ell$, we have 
\[
\frac{\| \bm{v}^{\perp} \|}{\|\bm{v}^+ \|} \asymp d(x_0, [\bm{v}]) < \|\bm{v}\|^{-\beta_{\chi}}.
\]
Therefore, we have
$$
\frac{\|\bm{v} \|^2}{\|\bm{v}^+ \|^2} = 1 + \frac{\| \bm{v}^{\perp} \|^2}{\|\bm{v}^+ \|^2} \leq 1 + C_0 \, \|\bm{v}^+ \|^{-2\beta_{\chi}},
$$
and hence
$$
\frac{\|\bm{v} \|^{\beta_{\chi}}}{\|\bm{v}^+ \|^{\beta_{\chi}}} \leq \left ( 1 + C_0 \, \|\bm{v}^+ \|^{-2\beta_{\chi}} \right )^{\beta_{\chi} / 2}.
$$
Since $\bm{v}$ does not lie in $Q_{2\ell}$, we have $\| \bm{v}^+ \| \geq C_0^{-1} \|\bm{v}\| \geq 2 \ell$. Thus, using the definition of $\widehat{c}_{\ell}$, we have (by enlarging $C_0$ where necessary)
\begin{align*}
d(x_0, [\bm{v}]) &\leq \|u_{\bm{v}}^-\|_{\ku^-}  \left ( 1 + C_0 \, \|u_{\bm{v}}^-\|_{\ku^-} \right ) \\
&\leq \widehat{c}_{\ell} \|\bm{v}^+ \|^{-\beta_{\chi}} \left ( 1 + C_0 \, \|\bm{v}^+ \|^{-\beta_{\chi}} \right ) \\
&= \|\bm{v}\|^{-\beta_{\chi}} \left ( \widehat{c}_{\ell} \frac{\|\bm{v} \|^{\beta_{\chi}}}{\|\bm{v}^+ \|^{\beta_{\chi}}}(1 + C_0 \, \|\bm{v}^+ \|^{-\beta_{\chi}}) \right ) \\
&\leq \|\bm{v}\|^{-\beta_{\chi}} \left ( \widehat{c}_{\ell} \left ( 1 + C_0 \, \ell^{-2\beta_{\chi}} \right )^{\beta_{\chi} / 2} (1 + C_0 \, \ell^{-\beta_{\chi}}) \right ) \\
&\leq \|\bm{v}\|^{-\beta_{\chi}}
\end{align*}
Since $\bm{v}$ does not lie in $Q_{2\ell}$, it does, in particular, not lie in $Q_{\ell}$. Moreover, we have $\|\bm{v}\| = \|\bm{v}^+\| \, \frac{\|\bm{v}\|}{\|\bm{v}^+\|} \leq  \widehat{c}_{\ell} \frac{\|\bm{v}\|}{\|\bm{v}^+\|} \, T < T$.  
This shows the left inclusion in Equation \eqref{eq:lem:Sandwich}. The other inclusion is proved similarly.  

As for the last claim, we note that $(a(y) \bm{v})^+ = y^{-\frac{1}{\beta_{\chi}}}\bm{v}^+$. Then, using \eqref{eq:LieAlgebraDiagonal} and observing that
$$
a(y_i)^{-1} \, \cF_c = \{\bm{v} \in \widetilde{X} : \|u_{\bm{v}}^-\|_{\ku^-} < c \, \|\bm{v}^+\|^{-\beta_{\chi}}, e^i \leq \|\bm{v}^+\| < e^{i+1} \}
$$
yields the desired decomposition. 
\end{proof}

Intersecting \eqref{eq:lem:Sandwich} with the discrete set $k_x^{-1} \cP_{\chi}$, we get for all large enough $\ell \geq 1$ and $T \geq 1$,
\[
\# \left ( k_x^{-1} \cP_{\chi}  \cap \cE_{T, \widehat{c}_{\ell}}^+ \smallsetminus Q_{2 \ell} \right )\, \leq \, \# \bigg ( k_x^{-1} \cP_{\chi}  \cap  \cE_{\beta_{\chi}}(T) \smallsetminus Q_{\ell} \bigg ) \, \leq \, \# \left ( k_x^{-1} \cP_{\chi}  \cap  \cE_{T, \widehat{c}_{\ell}^{-1}}^+ \right ).
\]
Moreover, by \cite[Theorem~C]{Pfitscher24}, where we counted rational points on $X$ of bounded height, and since $Q_\ell$ is $K$-invariant, we have, as $\ell \rightarrow + \infty$, 
\[
\# \left ( k_x^{-1} \cP_{\chi}  \cap Q_\ell \right ) = \# \left ( \cP_{\chi}  \cap Q_\ell \right ) = 2 \, \# \{v \in \bX(\Q) : H_{\chi}(v) \leq C_0 \ell \} \, \asymp \, \ell^{\beta_\chi d}.
\]
Therefore, we get, by enlarging $C_0$ if necessary,
\begin{equation} \label{eq:Sandwich_lem:Reduction}
\# \left ( k_x^{-1} \cP_{\chi}  \cap \cE_{T, \widehat{c}_{\ell}}^+ \right ) - C_0 \ell^{\beta_{\chi} d} \leq \# \left ( k_x^{-1} \cP_{\chi}  \cap \cE_{\beta_{\chi}}(T) \right ) \leq \# \left ( k_x^{-1} \cP_{\chi}  \cap \cE_{T, \widehat{c}_{\ell}^{-1}}^+ \right ) + C_0 \ell^{\beta_{\chi} d}.
\end{equation}
Using the decomposition of $\cE_{T,c}^+$ as in \eqref{eq:decomposition} with $T = \frac{1}{c} e^N$ for every integer  $N \geq 1$, we have
\begin{align} \label{eq:Method_Birkhoff}
\# (k_x^{-1} \cP_{\chi} \cap \cE_{T,c}^+) &= \# ( \cP_{\chi} \cap k_x \cE_{T,c}^+) \notag \\
&= \sum_{i=0}^{N-1} S_{\chi} \mathbbm{1}_{k_x a(y_i)^{-1}\cF_c} ( \G) = \sum_{i=0}^{N-1} S_{\chi} \mathbbm{1}_{\cF_c} (a(y_i) k_x^{-1} \G). 
\end{align} 
Plugging this back into \eqref{eq:Sandwich_lem:Reduction}, we get the lower and upper bounds, for every $T' \geq 1$ and large enough $\ell$,
\begin{align} \label{eq:Sandwich_lem:Reduction_Birkhoof}
\sum_{i=0}^{\lfloor T' + \ln \widehat{c}_\ell \rfloor -1} S_{\chi} \mathbbm{1}_{\cF_{c_{\ell}}} (a(y_i) k_x^{-1} \G) - C_0 \, \ell^{\beta_{\chi} d} &\leq \# \left ( k_x^{-1} \cP_{\chi}  \cap \cE_{\beta_\chi}(e^{T'}) \right ) \notag \\ 
&\leq \sum_{i=0}^{\lceil T' - \ln \widehat{c}_{\ell} \rceil - 1} S_{\chi} \mathbbm{1}_{\cF_{\widehat{c}_{\ell}^{-1}}} (a(y_i) k_x^{-1} \G) + C_0 \, \ell^{\beta_{\chi} d}.
\end{align}
The proof of Theorem \ref{thm:Critical} consists of effectively estimating the left- and right-hand sides of \eqref{eq:Sandwich_lem:Reduction_Birkhoof}. To this end, we will now develop the necessary tools and ingredients for these estimates.

\subsection{A uniform upper bound} \label{sec:Approximation}

A crucial input in the proof of Theorem~\ref{thm:Critical} is a uniform upper bound for the $K$-average of $S_{\chi} \mathbbm{1}_{\cF_c}(a(y) k \G)$ for all $y \geq 1$. Using Theorem~\ref{thm:L1}, we fix $\varepsilon > 0$ such that the Siegel transform $S_\chi$ maps $B_{c}^{\infty}(\widetilde{X})$ into $L^{1 + \varepsilon}(\Omega)$. Let $d_P(\cdot, \cdot)$ be the distance on $P$ induced from the right-invariant Riemannian metric $d_G(\cdot, \cdot)$ on $G$.

\begin{lemma} \label{lem:Replacement_for_alpha_s}
Let $s \in [1, 1 + \varepsilon)$. For every $c > 0$ and $g_0 \in G$, we have
\[ 
\sup_{y \geq 1} \int_{K} \big | S_{\chi} \mathbbm{1}_{\cF_c} ( a(y) k g_0 \G) \big|^s \, \dd \mu_K(k) < + \infty. 
\]
The upper bound is uniform as $g_0$ and $c$ vary in compact sets.
\end{lemma}
In the proof we use that the parabolic subgroup is a semi-direct product
\[
P = \mathrm{Levi} \ltimes U,
\]
where $\mathrm{Levi}$ is centralized by $A$ and $\ku = \Lie (U)$ is contracted by $\operatorname{Ad}(A)$. As a result, the $K$-average can be bounded by an integral over a neighborhood in $G$, and ultimately by an integral over $\Omega = G/\Gamma$, with the $G$-invariant measure $\mu_\Omega$, which absorbs the $a(y)$-factor.

\begin{proof}
Fix $s \in [1, 1 + \varepsilon)$ and let $Q \subset G$ be a compact subset. Consider the symmetric open neighborhood $B_{P}(1) = \{p \in P : d_P(p, \mathrm{Id}) < 1\}$ of $I \in P$. For all $y \geq 1$, we have $a(y) B_{P}(1) a(y)^{-1} \subseteq B_{P}(1)$. For each $r > 1$, we let $A(r) = B_{\widetilde{X}}(r) \smallsetminus B_{\widetilde{X}}(r^{-1})$. Thus there exists $r > 1$ such that for all $y\geq 1 $ and $p \in B_{P}(1)$, we have
\begin{equation} \label{eq:P_doesn't_grow_under_A}
a(y)^{-1} \cF_c \subset p^{-1} a(y)^{-1} A(r).
\end{equation}
Let $\mu_P$ be a left Haar measure on $P$ and consider the probability measure 
\[
\mu_P(B_P(1))^{-1} \mu_P|_{B_P(1)}.
\]
By \eqref{eq:P_doesn't_grow_under_A}, we get for all $y \geq 1$ and  $g_0 \in Q$,
\begin{multline*}
\int_{K} \big | S_{\chi} \mathbbm{1}_{\cF_c} ( a(y) k g_0 \G) \big|^s \, \dd \mu_K(k) \\ \leq \, \mu_P(B_P(1))^{-1} \int_{B_P(1)}\int_{K} \big | S_{\chi} \mathbbm{1}_{A(r)} ( a(y) p k g_0 \G) \big|^s \, \dd \mu_K(k) \dd \mu_P(p).
\end{multline*}
Now $B_P(1) \times K$ is a neighborhood of $I \in G$ and the product measure $\mu_P \times \mu_K$ is comparable to the Haar measure $\mu_G$ on $G$. Thus, we have  
\begin{align*}
\int_{B_P(1)}\int_{K} \big | S_{\chi} &\mathbbm{1}_{A(r)} ( a(y) p k g_0 \G) \big|^s \, \dd \mu_K(k) \dd \mu_P(p) \\
&\ll \, \int_{B_P(1) K} \big | S_{\chi} \mathbbm{1}_{A(r)} ( a(y) g g_0 \G) \big|^s \, \dd \mu_G(g) \\
&\leq \, \int_{B_P(1)  K  Q} \big | S_{\chi} \mathbbm{1}_{A(r)} ( a(y) g \G) \big|^s \, \dd \mu_G(g).
\end{align*}
Let $\mathfrak{F}$ be a fundamental domain with non-empty  interior for $\G \subset G$. There exists a finite subset $I \subset G$ such that $B_P(1) K Q \subset \bigcup_{h \in I} h \mathfrak{F}$. Using the $G$-invariance of $\mu_{\Omega}$, 
\begin{align*}
\int_{B_P(1)  K  Q} \big | S_{\chi} &\mathbbm{1}_{A(r)} ( a(y) g \G) \big|^s \, \dd \mu_G(g) \leq \int_{\bigcup_{h \in I} h \mathfrak{F}} \big | S_{\chi} \mathbbm{1}_{A(r)} ( a(y) g \G) \big|^s \, \dd \mu_G(g) \\
&\leq \sum_{h \in I} \int_{\mathfrak{F}} \big | S_{\chi} \mathbbm{1}_{A(r)} ( a(y) h g \G) \big|^s \, \dd \mu_G(g) \\
&\leq |I| \int_{\Omega} \big | S_{\chi} \mathbbm{1}_{A(r)} ( g \G) \big|^s \, \dd \mu_{\Omega}(g\G).
\end{align*}
By the choice of $s \in [1,1+ \varepsilon)$ and Theorem~\ref{thm:L1}, the last integral converges, as required. 
\end{proof}

\subsection{Non-escape of mass} \label{sec:Non-espace}
Next, we establish a non-escape of mass property for the orbit $a(y) K \G / \G \subset \Omega$.  We recall that, for every $g \in G$, we defined 
\[
\lambda_{\chi} (g\G) = \min_{\bm{v} \in \bV_{\chi}(\Z) \smallsetminus \{\bm{0}\}} \| g \bm{v} \|
\]
to be the length of the shortest non-zero vector in $g \bV_{\chi}(\Z)$. Let us also recall that $d$ stands for the dimension of $X$ and $\beta_{\chi}$ is the Diophantine exponent of $X$ with respect to $\chi$.  For each $\delta \in (0,1)$, we define an open cusp neighborhood in $\Omega$ by
\[
\Omega_\delta = \{g\G \in \Omega : \lambda_{\chi} (g\G) < \delta \}.
\]

\begin{lemma} \label{prop:Non-Escape}
There exists $\kappa > 0$ such that for all $\delta \in (0,1)$ and $y \in [\delta^{-\kappa}, + \infty)$,
\[
\mu_K \big (\{ k \in K : \lambda_\chi(a(y) k \G) < \delta \} \big ) \, \ll \, \delta^{\beta_{\chi} d}.
\]
\end{lemma} 

\begin{proof}
Let us denote by $\mathbbm{1}_{\Omega_\delta^c}$ the indicator function of the complementary subset $\Omega_\delta^c = \Omega \smallsetminus \Omega_\delta$. By Mahler's compactness criterion, the support of $\mathbbm{1}_{\Omega_\delta^c}$ is compact. We fix once and for all a non-negative function $\rho_1 \in C_c^{\infty}(G)$ with $\int_G \rho_1 \, \dd \mu_G  = 1$ and define $\chi_\delta = \rho_1 * \mathbbm{1}_{\Omega_\delta^c} : \Omega \rightarrow [0, +\infty)$. Since $\mu_{\Omega}$ is $G$-invariant, we have 
$$
\int_\Omega \chi_\delta \, \dd \mu_{\Omega} = \int_\Omega \mathbbm{1}_{\Omega_\delta^c} \, \dd \mu_{\Omega} = \mu_{\Omega}(\{\lambda_\chi \geq \delta \}).
$$
Moreover, by the $G$-invariance of $\mu_{\Omega}$ again, for any differential operator $\cD$, we have $\cD \chi_\delta = \cD(\rho) * \mathbbm{1}_{\Omega_\delta^c}$. In particular, we have $\chi_\delta \in C_c^{\infty}(\Omega)$ and, letting $r \geq 1$ be the integer from Theorem~\ref{thm:Effective}, also $\cS_r(\chi_\delta) \ll \cS_r(\rho_1) \ll 1$. Moreover, there exists $\xi = \xi(\rho_1) > 1$ such that, for every $g \in \supp(\rho_1)$ and $x \in \Omega$, we have $\lambda_\chi(g x) \leq \xi \lambda_\chi(x)$. Therefore, for every $g \in \supp(\rho_1)$, we have 
\[
\{ x \in \Omega : \lambda_\chi(gx) \geq \delta\} \subseteq \{ x \in \Omega : \lambda_\chi(x) \geq \xi^{-1} \delta \}
\]
and hence $\chi_\delta \leq \mathbbm{1}_{\Omega_{\xi^{-1} \delta}^c}$. Thus, for every $y \geq 1$, we have
\begin{align*}
\mu_K(\{k \in K:\lambda_\chi( a(y)k \G) \geq \xi^{-1} \delta \}) &= \int_{K} \mathbbm{1}_{\Omega_{\xi^{-1}\delta}^c}(a(y) k \G) \,  \dd\mu_K(k) \\
&\geq \int_K \chi_\delta(a(y) k \G) \, \dd\mu_K(k).
\end{align*}
By Theorem~\ref{thm:Effective},  there exists $c >0 $ such that, for every $y \geq 1$, we have
\begin{align*}
\int_K \chi_\delta(a(y) k \G ) \, d\mu_K(k) &= \int_\Omega \chi_\delta \, d\mu_{\Omega} + O\left(y^{-c} \cS_r(\chi_\delta) \right) \\
&=\mu_{\Omega} (\{\lambda_\chi \geq \delta \}) + O ( y^{-c} ).
\end{align*}
By \cite[Proposition 3.1.1]{deSaxce20}, we have the measure estimate 
\begin{equation} \label{eq:Measure-Cusp-Saxce}
\mu_{\Omega}(\Omega_\delta) \asymp \delta^{\beta_{\chi} d}.
\end{equation}
Putting everything together, for every $y \geq 1$, this yields
\[
\mu_K(\{k \in K: \lambda_\chi(a(y) k \G) \geq \xi^{-1} \delta \}) \geq \mu_{\Omega}(\{\lambda_\chi \geq \delta \})+O\left(y^{-c} \right) = 1 + O\left(\delta^{\beta_{\chi} d} + y^{-c} \right).
\]
Therefore, since $\rho_1$ is fixed and $\xi$ only depends on $\rho_1$, for every $y \geq 1$, we have
\[
\mu_K(\{k \in K:\lambda_\chi(a(y) k \G) < \delta \}) \ll (\xi \delta)^{\beta_{\chi} d}+y^{-c} \ll \delta^{\beta_{\chi} d}+y^{-c}.
\]
Letting $\kappa=\frac{\beta_{\chi} d}{c}$, the claim holds for all $\delta \in (0,1)$ and $y \in [ \delta^{-\kappa}, +\infty)$.
\end{proof}

\subsection{Approximation by smooth compactly supported functions} \label{sec:Approx-Smooth}

The Siegel transform of the indicator function of the set $\cF_{c}$ appearing in the decomposition \eqref{eq:decomposition} is neither smooth nor bounded. In order to apply effective equidistribution results, we approximate $S_{\chi} \mathbbm{1}_{\cF_{c}}$ by smooth compactly supported functions. We again fix the integer $r \geq 1$ as in Theorem~\ref{thm:Effective}.

\begin{lemma} \label{lem:Truncated-Siegel}
For every $\xi>1$, there exists a family of functions $( D_\delta )_{\delta \in (0,1)}$ in $C_c^{\infty}(\Omega)$ satisfying
$$ 0 \leq D_\delta \leq 1 , \quad D_\delta = 1 \text{ on } \{ \lambda_\chi \geq \xi \delta \} , \quad D_\delta = 0 \text{ on } \{ \lambda_\chi < \xi^{-1} \delta \} , \quad \cS_r(D_\delta ) \ll 1.\\
$$
\end{lemma}
The proof is essentially analogous to that of \cite[Lemma~4.11]{BG19} and we omit the details. We refer to the family $( D_\delta )_{\delta \in (0,1)}$ as a family of smooth cut-off functions on $\Omega$ and, fixing once and for all a $\xi > 1$ in the above lemma, we will omit $\xi$ from the notation. For every $\delta \in (0,1)$ and $f \in B_c^{\infty}(\widetilde{X})$, we define the \emph{$\delta$-truncated Siegel transform} $S_{\chi}^{(\delta)} f : \Omega \rightarrow \C$ of $f$ by
\begin{equation} \label{eq:Truncated-Siegel}
\forall \, g \in G, \quad S_{\chi}^{(\delta)} f(g) = D_\delta(g\G) \, S_{\chi} f (g\G).
\end{equation}

Next, we approximate $\mathbbm{1}_{\cF_c}$ for $c$ arbitrarily close to $1$ by a family of non-negative smooth compactly supported functions. For every $\varepsilon \in (0,1)$, we recall the definition of the $\varepsilon$-neighborhood $\cF_{c}(\varepsilon)$ of $\cF_c$ given by
$$
\cF_{c}(\varepsilon) = \{\bm{v} \in \widetilde{X} : \|u_{\bm{v}}^-\|_{\ku^-} < (1+\varepsilon)^{1+\beta_{\chi}} \, c \, \|\bm{v}^+\|^{-\beta_{\chi}}, (1+\varepsilon)^{-1} \leq \|\bm{v}^+\| < (1+\varepsilon) e \}.
$$
There exists a family $( f_{\varepsilon,c} )_{\varepsilon \in (0,1), \, c \in [1/2,3/2]} \subset C_c^\infty(\widetilde{X})$ with $\supp(f_{\varepsilon,c}) \subset \cF_{c}(\varepsilon)$ satisfying the following properties:
\begin{equation} \label{eq:EpsilonApproximationChi}
\forall \, \varepsilon \in (0,1), \quad \mathbbm{1}_{\cF_c} \leq f_{\varepsilon,c} \leq 1, \quad \|f_{\varepsilon,c} - \mathbbm{1}_{\cF_c} \|_{L^{1}(\widetilde{X})} \ll \varepsilon, 
\quad \cS_r(f_{\varepsilon,c}) \ll \varepsilon^{-r},
\end{equation}
and the implicit constants are uniform in $c \in [1/2,3/2]$.

\begin{lemma} \label{prop:SmoothApproximationOrbit}
There exists $\varepsilon_1 > 0$ such that for all $\varepsilon \in (0,1)$ and $y \in [ (\beta_{\chi} \varepsilon)^{- \tfrac{\beta_{\chi}}{\varepsilon_1}}, + \infty)$, we have
$$
\int_{K} \left |S_{\chi} f_{\varepsilon,c}( a(y) k \G) - S_{\chi} \mathbbm{1}_{\cF_c}(a(y) k \G) \right | \dd \mu_K(k) \, \ll \, \varepsilon.
$$
The implicit constant is uniform in $c \in [1/2,3/2]$, but depends on $\beta_{\chi}$ and $d$. 
\end{lemma}

\begin{proof} 
Since $\supp(f_{\varepsilon,c}) \subset \cF_c(\varepsilon)$, we have $S_{\chi} \mathbbm{1}_{\cF_c(\varepsilon)} - S_{\chi} \mathbbm{1}_{\cF_c} \geq S_{\chi} f_{\varepsilon,c} - S_{\chi} \mathbbm{1}_{\cF_c}$.  The difference $\mathbbm{1}_{\cF_c(\varepsilon)} - \mathbbm{1}_{\cF_c}$ is bounded by the sum $\mathbbm{1}_{\cR_1(\varepsilon)}+\mathbbm{1}_{\cR_2(\varepsilon)}+\mathbbm{1}_{\cR_3(\varepsilon)}$ of indicator functions of the sets
\begin{align*}
\cR_1(\varepsilon) &= \{\bm{v} \in \widetilde{X} : \|u_{\bm{v}}^-\|_{\ku^-} < (1+\varepsilon)^{1+\beta_{\chi}} \, c \, \|\bm{v}^+\|^{-\beta_{\chi}}, (1+\varepsilon)^{-1} \leq \|\bm{v}^+\| < 1 \}, \\
\cR_2(\varepsilon) &= \{\bm{v} \in \widetilde{X} : \|u_{\bm{v}}^-\|_{\ku^-} < (1+\varepsilon)^{1+\beta_{\chi}} \, c \, \|\bm{v}^+\|^{-\beta_{\chi}}, e \leq \|\bm{v}^+\| < (1+\varepsilon) e \}, \\
\cR_3(\varepsilon) &= \{\bm{v} \in \widetilde{X} :  c \, \|\bm{v}^+\|^{-\beta_{\chi}} \leq \|u_{\bm{v}}^-\|_{\ku^-} < (1+\varepsilon)^{1+\beta_{\chi}} \, c \, \|\bm{v}^+\|^{-\beta_{\chi}}, 1 \leq \|\bm{v}^+\| < e \}. 
\end{align*} 
In particular, we have $ S_{\chi} f_{\varepsilon,c} - S_{\chi} \mathbbm{1}_{\cF_c} \leq S_{\chi} \mathbbm{1}_{\cR_1(\varepsilon)} + S_{\chi} \mathbbm{1}_{\cR_2(\varepsilon)}+S_{\chi} \mathbbm{1}_{\cR_3(\varepsilon)}$, and it is enough to show that for all $y \geq 1$ sufficiently large in terms of $\varepsilon$, we have, for every $i \in \{1,2,3\}$,
$$
J_i(y) = \int_{K} S_{\chi} \mathbbm{1}_{\cR_i(\varepsilon)}(a(y) k \G) \, d\mu_K(k) = \sum_{\bm{v} \in \cP_{\chi}} \int_{K} \mathbbm{1}_{\cR_i(\varepsilon)}(a(y)k \bm{v}) \, \dd \mu_K(k)\ll \varepsilon.
$$
We start with $J_1(y)$. Using polar coordinates on $\widetilde{X}$ (see Section \ref{sec:Measure}), for every $\bm{v} \in \cP_{\chi}$ there exist $k_{\bm{v}} \in K$ and $y(\bm{v}) \in \R_+^\times$ such that $\bm{v} = k_{\bm{v}} a(y(\bm{v})) \bm{e}_{\chi}$. Taking norms of both sides and recalling that $a(y(\bm{v}))$ acts through the character $\chi$ on $\bm{e}_{\chi}$, we have $\bm{v} = \|\bm{v}\| k_{\bm{v}} \bm{e}_{\chi}$. Now, the right $K$-invariance of $\mu_K$ gives, for every $t \in \R$,
\[
J_1(t) = \sum_{\bm{v} \in \cP_{\chi}} \int_{K} \mathbbm{1}_{\cR_1(\varepsilon)}(a(y) \|\bm{v}\| k \bm{e}_{\chi}) \, \dd \mu_K(k).
\]
Let $\bm{v} \in \cP_{\chi}$ and let us define the two intervals $I_1(\varepsilon) = [(1+\varepsilon)^{-1}, 1)$ and $I_2(\varepsilon) = [0, (1+\varepsilon)^{1+\beta_{\chi}} c)$. We observe that, for every $k \in K$, we have $\mathbbm{1}_{\cR_1(\varepsilon)}(a(y) \|\bm{v}\| k \bm{e}_{\chi}) = 1$ if and only if
\begin{equation} \label{eq:Proof-Prop-Smooth-Approx-1}
\mathbbm{1}_{I_{1}(\varepsilon)} \left( \|(a(y) \|\bm{v}\| k \bm{e}_{\chi})^+\| \right) \mathbbm{1}_{I_{2}(\varepsilon)} \left( \| u_{a(y) k \bm{e}_{\chi}}^- \|_{\ku^-} \, \|(a(y) \|\bm{v}\| k \bm{e}_{\chi})^+\|^{\beta_{\chi}} \right) = 1.
\end{equation}
For every $\bm{v} \in \widetilde{X}$ and $k \in K$, we have $(a(y) \|\bm{v}\| k \bm{e}_{\chi})^+ = y^{-\frac{1}{\beta_{\chi}}}(\|\bm{v}\| k \bm{e}_{\chi})^+$ and $\| u_{a(y) k \bm{e}_{\chi}}^- \|_{\ku^-} = y \| u_{k \bm{e}_{\chi}}^- \|_{\ku^-}$ (see Equation \eqref{eq:LieAlgebraDiagonal}). Moreover, we have $\| (\|\bm{v}\| k \bm{e}_{\chi})^+ \| = \|\bm{v} \| \, | \langle k \bm{e}_{\chi}, \bm{e}_{\chi} \rangle |$, since $\bm{e}_{\chi}$ is unitary. Hence, \eqref{eq:Proof-Prop-Smooth-Approx-1} holds if and only if
\[
\mathbbm{1}_{I_{1}(\varepsilon)} \left(y^{-\frac{1}{\beta_{\chi}}}\|\bm{v} \| \, | \langle k \bm{e}_{\chi}, \bm{e}_{\chi} \rangle | \right) \mathbbm{1}_{I_{2}(\varepsilon)} \left( \| u_{ k \bm{e}_{\chi}}^- \|_{\ku^-} \, (\|\bm{v} \| \, | \langle k \bm{e}_{\chi}, \bm{e}_{\chi} \rangle |)^{\beta_{\chi}} \right) = 1.
\]
By the definition of $\cR_1(\varepsilon)$, there exists an absolute constant $\widehat{C} > 1$ such that for every $\bm{v} \in \cR_1(\varepsilon)$, we have $\|u_{\bm{v}}^-\|_{\ku^-} < \widehat{C}$. In particular, if $[\cR_1(\varepsilon)]$ denotes the corresponding set in $X$, then $[\cR_1(\varepsilon)] \subset B_{\ku^-}(\widehat{C}) \, x_0$. This implies that there exists a small constant $\widehat{c} > 0$ such that the region $\cR_1(\varepsilon)$ is contained in the set $\mathcal{C} = \{\bm{v} \in \widetilde{X} : \|\bm{v}^+ \| \geq \widehat{c} \, \|\bm{v}\| \}$. For every $y \geq 1$, the set $\mathcal{C}$ is stable under the action of $a(y)$ and for every $\bm{v} \in \mathcal{C}$, written as $\bm{v} = \|\bm{v}\| k \bm{e}_{\chi}$ as above, we have $|\langle k \bm{e}_{\chi}, \bm{e}_{\chi} \rangle | \geq \widehat{c}$. Therefore, letting $K(\widehat{c}) = \{k \in K : |\langle k \bm{e}_{\chi}, \bm{e}_{\chi} \rangle| \geq \widehat{c}\}$, we have that $J_1(t)$ is given by
\begin{align*}
\sum_{\bm{v} \in \cP_{\chi}}
\int_{K(\widehat{c})}
\mathbbm{1}_{I_{1}(\varepsilon)} \left( y^{-\frac{1}{\beta_{\chi}}} \|\bm{v} \| | \langle k \bm{e}_{\chi}, \bm{e}_{\chi} \rangle | \right) 
\mathbbm{1}_{I_{2}(\varepsilon)} \left( \| u_{ k \bm{e}_{\chi}}^- \|_{\ku^-} (\|\bm{v} \| | \langle k \bm{e}_{\chi}, \bm{e}_{\chi} \rangle | )^{\beta_{\chi}} \right) \, \dd \mu_K(k). 
\end{align*}
Let us write $J_1(y) = B_1(y) + B_2(y)$, where $B_1(y)$ is given by
\begin{align*}
\sum_{\bm{v} \in \cP_{\chi}}
\int_{K((1+\varepsilon)^{-1})}
\mathbbm{1}_{I_{1}(\varepsilon)} \left( y^{-\frac{1}{\beta_{\chi}}} \|\bm{v} \| | \langle k \bm{e}_{\chi}, \bm{e}_{\chi} \rangle | \right) 
\mathbbm{1}_{I_{2}(\varepsilon)} \left( \| u_{ k \bm{e}_{\chi}}^- \|_{\ku^-} (\|\bm{v} \| |\langle k \bm{e}_{\chi}, \bm{e}_{\chi} \rangle|)^{\beta_{\chi}} \right) d \mu_K(k). 
\end{align*}
and $B_2(y) = J_1(y) - B_1(y)$. Consequently, for every $k \in K$ satisfying that $|\langle k \bm{e}_{\chi}, \bm{e}_{\chi} \rangle | \geq (1 + \varepsilon)^{-1}$, if $\mathbbm{1}_{I_{1}(\varepsilon)} \left( y^{-\frac{1}{\beta_{\chi}}} \|\bm{v} \| |\langle k \bm{e}_{\chi}, \bm{e}_{\chi} \rangle | \right) = 1$, then
\[
(1+\varepsilon)^{-1} y^{\frac{1}{\beta_{\chi}}} \leq \|\bm{v} \| \leq (1 + \varepsilon) y^{\frac{1}{\beta_{\chi}}} 
\]
Hence, if also $\mathbbm{1}_{I_{2}(\varepsilon)} \left( \| u_{ k \bm{e}_{\chi}}^{-} \|_{\ku^-} (\| \bm{v} \| | \langle k \bm{e}_{\chi}, \bm{e}_{\chi} \rangle| )^{\beta_{\chi}} \right) = 1$, then
\[
\| u_{ k \bm{e}_{\chi}}^- \|_{\ku^-} \ll y^{-1}. 
\] 
The latter implies that we have $d(x_0, k x_0) \ll y^{-1}$. Together this gives, 
\begin{align*}
B_1(y) &\ll \sum_{\bm{v} \in \cP_{\chi}, \, \|\bm{v} \| = y^{\frac{1}{\beta_{\chi}}} + O(\varepsilon y^{\frac{1}{\beta_{\chi}}})} \int_{X} \mathbbm{1}_{B_X(y^{-1})}(x) \, d\sigma_X(x) \\
&\ll \sum_{\bm{v} \in \cP_{\chi}, \, \|\bm{v} \| = y^{\frac{1}{\beta_{\chi}}} + O(\varepsilon y^{\frac{1}{\beta_{\chi}}})} y^{- d}.
\end{align*}
By \cite[Theorem~C]{Pfitscher24}, there exist constants $\varkappa_1 > 0$ and $\varepsilon_1 > 0$ such that 
\[
\#\{\bm{v} \in \cP_{\chi} : \|\bm{v}\| \leq T\} = \varkappa_1 \, T^{\beta_\chi d}(1 + O(T^{-\varepsilon_1})),
\]
and consequently
\[
\#\{\bm{v} \in \cP_{\chi} : (1+\varepsilon)^{-1} y^{\frac{1}{\beta_{\chi}}} \leq \|\bm{v}\| \leq (1 + \varepsilon) y^{\frac{1}{\beta_{\chi}}}\} = 2 \, \beta_{\chi} \, d \, \varkappa_1 \, \varepsilon \, y^{d} + O(\varepsilon^2 y^{d} + y^{- \frac{\varepsilon_1}{\beta_{\chi}}} y^{ d}).
\]
Hence we get $B_1(y) \ll \beta_{\chi} \, \varepsilon$ for all $y \geq  (\beta_{\chi} \, \varepsilon)^{- \tfrac{\beta_{\chi}}{\varepsilon_1} }$.  

Let us now bound the term $B_2(y)$, where we integrate over all $k \in K$ such that $\widehat{c} \leq |\langle k \bm{e}_{\chi}, \bm{e}_{\chi} \rangle| \leq (1 +  \varepsilon)^{-1}$. The vector $(k\bm{e}_{\chi})^{\perp} = k\bm{e}_{\chi} - \langle k\bm{e}_{\chi}, \bm{e}_{\chi} \rangle$ satisfies 
\[
\|u_{k \bm{e}_{\chi}} \|_{\ku^-} \asymp d(k x_0, x_0) \asymp \|(k\bm{e}_{\chi})^{\perp} \| \geq 1 - |\langle k\bm{e}_{\chi}, \bm{e}_{\chi} \rangle | \gg \varepsilon.
\]
In particular, there exists $\delta_0 > 0$ such that $\|u_{k \bm{e}_{\chi}} \|_{\ku^-} | \langle k \bm{e}_{\chi} , \bm{e}_{\chi} \rangle |^{\beta_{\chi}} \geq \delta_0 \varepsilon$ is bounded away from zero. Thus, if $\mathbbm{1}_{I_{2}(\varepsilon)} \left( \| u_{ k \bm{e}_{\chi}}^{-} \|_{\ku^-} (\| \bm{v} \| | \langle k \bm{e}_{\chi}, \bm{e}_{\chi} \rangle| )^{\beta_{\chi}} \right) = 1$, then $\mathbbm{1}_{I_{2}(\varepsilon)} \left( \delta_0 \varepsilon \|\bm{v} \|^{\beta_{\chi}} \right) = 1$ and hence 
\begin{equation}\label{eq:longcomplicatedterm-1}
\|\bm{v}\|^{\beta_{\chi}} < (1+\varepsilon)^{1+\beta_{\chi}} c (\delta_0 \varepsilon)^{-1}.  
\end{equation}
Therefore, for every $y\geq 1$, the non-negative term $B_2(y)$ is bounded from above by
\begin{equation} \label{eq:longcomplicatedterm}
\sum_{\substack{\bm{v} \in \cP_{\chi} \\ \|\bm{v}\|^{\beta_{\chi}} < (1+\varepsilon)^{1+\beta_{\chi}} c (\delta_0 \varepsilon)^{-1}}} \int_{K((1+\varepsilon)^{-1}) \smallsetminus K(\widehat{c})} \mathbbm{1}_{I_{1}(\varepsilon)} \left( y^{-\frac{1}{\beta_{\chi}}} \|\bm{v} \| | \langle k \bm{e}_{\chi}, \bm{e}_{\chi} \rangle | \right) \, \dd \mu_K(k).
\end{equation}
We claim that, for all $y \geq (1+\varepsilon)^{1+\beta_{\chi}} c (\delta_0 \varepsilon)^{-1}$, the expression \eqref{eq:longcomplicatedterm} is zero. In fact, by \eqref{eq:longcomplicatedterm-1} and since $\widehat{c} \leq |\langle k \bm{e}_{\chi}, \bm{e}_{\chi} \rangle| \leq (1 +  \varepsilon)^{-1}$, we have 
\[
y^{-\frac{1}{\beta_{\chi}}} \|\bm{v} \| | \langle k \bm{e}_{\chi}, \bm{e}_{\chi} \rangle | \leq y^{-\frac{1}{\beta_{\chi}}} \left ( (1+ \varepsilon)^{1 + \beta_{\chi}} c (\delta_0 \varepsilon)^{-1} \right )^{\frac{1}{\beta_{\chi}}} (1+ \varepsilon)^{-1}
\]
and, for all $y \geq (1+\varepsilon)^{1+\beta_{\chi}} c (\delta_0 \varepsilon)^{-1}$, the right-hand side is smaller than $(1+\varepsilon)^{-1}$ (and hence $\mathbbm{1}_{I_{1}(\varepsilon)} \left( y^{-\frac{1}{\beta_{\chi}}} \|\bm{v} \| | \langle k \bm{e}_{\chi}, \bm{e}_{\chi} \rangle | \right) = 0$).

The calculations for $J_2(y)$ and $J_3(y)$ are essentially analogous to that for $J_1(t)$ and we omit the details. 
\end{proof}

\subsection{Proof of Theorem \ref{thm:Critical}} \label{sec:CountingCritical}
We first generalize \cite[Lemma 1.4]{Harman98}, a useful tool in the theory of metric Diophantine approximation for deriving effective counting statements from an $L^2$-bound, to obtain such effective statements from an $L^p$-bound for arbitrary $p \in (1,2]$. The idea of the proof goes back to the work of H. Weyl \cite{Weyl16} on the equidistribution of numbers modulo one. 
\begin{lemma} \label{lem:Weyl-Metric}
Let $(Y, \nu)$ be a probability space and let $(\phi_{i,\ell} : Y \rightarrow \R)_{i,\ell \in \N^*}$ be a family of non-negative random variables. Let $C_1 > 1$ be a constant and $(\overline{\phi}_{i,\ell})_{i,\ell \in \N^*}$ and  $(\overline{\phi}_{i})_{i \in \N^*}$ be families of real numbers satisfying, for all $i,\ell \in \N^*$, $0 \leq \overline{\phi}_{i,\ell} \leq \overline{\phi}_{i} \leq C_1$, and put $Z_{i,\ell} = \phi_{i,\ell} - \overline{\phi}_{i,\ell}$. Assume that $\sum_{i = 1}^{\infty} \overline{\phi}_i = +\infty $ and that for some $p \in (1,2]$ and $C_2 > 0$, we have 
\begin{equation} \label{eq:1:lem:From_L_p}
\forall \, N \in \N^*, \forall \, \ell \in \N^*, \quad \int_{Y} \Big | \sum_{i = 1}^N Z_{i,\ell}(y) \Big |^p \, d \nu(y) \leq C_2 \sum_{i = 1}^N \overline{\phi}_i. 
\end{equation}
Let $\varepsilon > 0$ and let $(\ell_N)_{N \geq 1}$ be a sequence of positive integers. Then there is a constant $C_3 > 0$ such that almost surely
\begin{equation} \label{eq:2:lem:From_L_p}
\forall \, N \in \N^*, \quad \left | \sum_{i = 1}^{N} Z_{i,\ell_N} \right| \leq C_3 \left ( \sum_{i = 1}^N \overline{\phi}_i \right )^{\frac{2}{p+1} + \varepsilon}
\end{equation}
\end{lemma}
\begin{proof}
Let $\varepsilon > 0$ and let $(\ell_N)_{N \geq 1}$ be a sequence of positive integers. For every $N \in \N^*$ and $y \in Y$, we define 
\begin{align*}
\Psi(N,y) &= \sum_{i = 1}^N \phi_{i,\ell_N}(y), \quad \Psi(N) = \sum_{i = 1}^N \overline{\phi}_{i,\ell_N}, \quad \Phi(N) = \sum_{i = 1}^N \overline{\phi}_i, \\ E(N,y) &= \Psi(N,y) - \Psi(N) = \sum_{i=1}^N Z_{i,\ell_N}(y).
\end{align*}
Let $x > 1$. For every $k \in \N^*$, let $N_k \in \N^*$ be the smallest integer with  
\begin{equation} \label{eq:Phi(N_k)}
\Phi(N_k) > k^{px - 1}.
\end{equation}
We remark that the sequence $(N_k)_{k \in \N^*}$ is increasing. Next, for every $k \in \N^*$, we define the subset $A_k$ of $Y$ by
\[
A_k = \{y \in Y : | E(N_k , y) | >  k^{x +\varepsilon}\}.
\]
By Chebyshev's inequality and our assumption \eqref{eq:1:lem:From_L_p}, we have
\begin{align*}
\nu(A_k) &\leq \frac{1}{k^{px + p \varepsilon}} \int_{Y} | E(N_k , y) |^p \, d \nu(y) \leq \frac{C_2 k^{px - 1} + C_1}{k^{px + p \varepsilon}} \ll k^{-1-p\varepsilon}.
\end{align*}
Hence $\sum_{k = 1}^{\infty} \nu(A_k)$ converges. By the Borel-Cantelli lemma, we have, for almost every $y \in Y$, that there exists an integer $k(y) \in \N^*$ such that for all $k \geq k(y)$, we have
\[
| E(N_k , y) | \leq  k^{x +\varepsilon} 
\]
and also, by \eqref{eq:Phi(N_k)}, 
\[
k^{x +\varepsilon} \ll_x (k-1)^{x +\varepsilon} \ll \Phi(N_{k-1})^{\frac{x+\varepsilon}{px-1}}.
\]
Now, for an arbitrary $N \in \N^*$ there exists $k \in \N^*$ such that $N_{k-1} \leq N < N_k$ and 
\[
\Psi(N_{k-1}, y) \leq \Psi(N, y) \leq \Psi(N_{k}, y). 
\]
Thus, for almost every $y \in Y$, there exists an integer $k(y) \in \N^*$ such that for all $k \geq k(y)$, we have
\begin{equation} \label{eq:Exponent_A}
\Psi(N, y) = \Psi(N) + O\left (\Psi(N_{k}) - \Psi(N_{k-1}) + \Phi(N_{k-1})^{\frac{x+\varepsilon}{px-1}} \right ),
\end{equation}
Next, we note that, for all $k \in \N^*$, we have
\begin{equation} \label{eq:Exponent_B}
\Psi(N_{k}) - \Psi(N_{k-1}) \leq \Phi(N_{k}) - \Phi(N_{k-1}) \ll k^{px-2} \ll \Phi(N_{k-1})^{\frac{px-2}{px-1}}.
\end{equation}
We put $x = \frac{2+\varepsilon}{p-1}$, so that the exponents in \eqref{eq:Exponent_A} and \eqref{eq:Exponent_B} match.
Plugging this back into $\frac{x+\varepsilon}{px-1}$, we find that
\[
\frac{\frac{2+\varepsilon}{p-1}+\varepsilon}{p\frac{2+\varepsilon}{p-1}-1} = \frac{2 + \varepsilon(2-p)}{p + 1 + p \varepsilon} \leq \frac{2 + \varepsilon(2-p)}{p + 1} \leq \frac{2}{p + 1} + \varepsilon.
\]
The proof of Lemma \ref{lem:Weyl-Metric} is complete. 
\end{proof}

\subsection{Upper bound estimate}
We recall that, for all $\ell \in \N^*$, we defined 
\[
\widehat{c}_{\ell} = \left ( 1 + C_0 \, \ell^{-\beta_{\chi}}  \right )^{-(1+\beta_{\chi})} \in (0,1)
\]
and we let, for every $c \in [1/2, 3/2]$, 
\[
\cF_c = \{\bm{v} \in \widetilde{X} : \|u_{\bm{v}}^-\|_{\ku^-} < c  \|\bm{v}^+\|^{-\beta_{\chi}}, 1 \leq \|\bm{v}^+\| < e \}.
\]
We also defined, for every $i \in \N$, the times $y_i = e^{\beta_{\chi} i}$ and recall from Section \ref{sec:Approx_E_T} that there exists a constant $C_0 > 0$ such that for every $T' \geq 1$ and $x \in X$, we have the following lower and upper bounds on the lattice point counting function $\# \left ( k_x^{-1} \cP_{\chi}  \cap \cE_{\beta_\chi}(e^{T'}) \right )$: for all large enough $\ell \in \N^*$, we have
\begin{align} \label{eq:Sandwich_lem:Reduction_Birkhoof_2}
\sum_{i=0}^{\lfloor T' + \ln \widehat{c}_\ell \rfloor -1} S_{\chi} \mathbbm{1}_{\cF_{c_{\ell}}} (a(y_i) k_x^{-1} \G) - &C_0 \, \ell^{\beta_{\chi} d} \leq \# \left ( k_x^{-1} \cP_{\chi}  \cap \cE_{\beta_\chi}(e^{T'}) \right ) \notag \\ 
&\leq \sum_{i=0}^{\lceil T' - \ln \widehat{c}_{\ell} \rceil - 1} S_{\chi} \mathbbm{1}_{\cF_{\widehat{c}_{\ell}^{-1}}} (a(y_i) k_x^{-1} \G) + C_0 \, \ell^{\beta_{\chi} d}.
\end{align}
The proof of Theorem \ref{thm:Critical} consists of effectively estimating the left- and right-hand sides of \eqref{eq:Sandwich_lem:Reduction_Birkhoof_2}. We start with the latter and, by symmetry, omit the lower bound estimates. 
\begin{proposition} \label{prop:Upper_Bound_on_Lattice_Count}
There exist $\delta' \in (0,1)$ and $\varepsilon \in (0,1)$ such that the following holds. Let $\varkappa_1 = \lambda_{\widetilde{X}}(\cF)$ be the volume of $\cF = \cF_1$ and, for every $N \in \N^*$, let $\ell_N = \lfloor N^{\delta'}\rfloor$. Then for almost every $k \in K$ and for all large enough $N \in \N^*$, we have
\begin{equation}
\sum_{i=0}^{N-1} S_{\chi} \mathbbm{1}_{\cF_{\widehat{c}_{\ell_N}^{-1}}} (a(y_i) k \G) + C_0 \, \ell_N^{\beta_{\chi} d}= \varkappa_1 \, N + O_x(N^{1-\varepsilon}). 
\end{equation}
\end{proposition}
Using the inequality \eqref{eq:Sandwich_lem:Reduction_Birkhoof_2}, by Lemma \ref{lem:Reduction} and Proposition \ref{prop:Upper_Bound_on_Lattice_Count}, this proves the desired upper bound in \eqref{eq:Thm_Critical}. he lower bound in \eqref{eq:Thm_Critical} is shown analogously and we omit the details.


\begin{proof} [Proof of Proposition \ref{prop:Upper_Bound_on_Lattice_Count}]
We will apply Lemma \ref{lem:Weyl-Metric} with $(Y, \nu) = (K, \mu_K)$. Let us define, for all $i \in \N^*$ and $\ell \in \N$, the function $\phi_{i, \ell} : K \rightarrow \R$  and the number $\overline{\phi}_{i,\ell}$ by
\[
\forall \, k \in K, \quad \phi_{i,\ell}(k) = S_{\chi} \mathbbm{1}_{\cF_{c_{\ell}^{-1}}}(a(y_i) k \G), \quad \overline{\phi}_{i,\ell} = \int_{\Omega} S_{\chi} \mathbbm{1}_{\cF_{c_{\ell}^{-1}}} \, \dd \mu_{\Omega}.
\]
Moreover, we define the function $Z_{i,\ell} : K \rightarrow \R$ by $Z_{i,\ell} = \phi_{i,\ell} - \overline{\phi}_{i,\ell}$. Let us also put $\overline{\phi}_{i} = \overline{\phi}_{i,1}$ (so that, $i \in \N^*$ and $\ell \in \N$, $\overline{\phi}_{i} \geq \overline{\phi}_{i,\ell}$). Fix $\ell \in \N^*$ and $\varepsilon_0 \in (0,1]$ such that the Siegel transform $S_{\chi}$ maps $B_{c}^{\infty}(\widetilde{X})$ into $L^{1+\varepsilon_0}(\Omega)$ (see Theorem \ref{thm:L1}). Let us simply write $\mathbbm{1}_{\ell}$ for $\mathbbm{1}_{\cF_{c_{\ell}^{-1}}}$. By Lemma \ref{lem:Weyl-Metric}, we need to show that there exists $p \in (1,2]$, such that 
\begin{multline} \label{eq:Prop:Upper_Bound_on_Lattice_Count}
\forall \, N \in \N^*, \, \forall \, \ell \in \N^*, \\ \int_K \left | \sum_{i = 1}^N Z_{i,\ell}(k ) \right |^p \, \dd \mu_K(k) = \int_K \left | \sum_{i = 1}^N \bigg (S_{\chi} \mathbbm{1}_{\ell}(a(y_i)k \G) - \int_\Omega S_{\chi} \mathbbm{1}_{\ell} \dd \mu_{\Omega} \bigg ) \right |^p \dd \mu_K(k) \ll N.
\end{multline}

First, we approximate the Siegel transform $S_{\chi} \mathbbm{1}_\ell$ by the $\delta$-truncated counterpart $S_{\chi}^{(\delta)} \mathbbm{1}_{\ell} = D_\delta\, S_{\chi} \mathbbm{1}_\ell$ (with $(D_{\delta})_{\delta \in (0,1)}$ as in Lemma \ref{lem:Truncated-Siegel}) with $\delta \in (0,1)$ to be determined later. That is, we would like to give an upper bound in terms of the truncation parameter $\delta$ for
\[
\Big \| \Big (S_{\chi} \mathbbm{1}_{\ell} \circ a(y) - \int_\Omega S_{\chi} \mathbbm{1}_{\ell} \, \dd \mu_{\Omega} \Big) - \Big ( S_{\chi}^{(\delta)} \mathbbm{1}_{\ell} \circ a(y) - \int_\Omega S_{\chi}^{(\delta)} \mathbbm{1}_{\ell} \, \dd \mu_{\Omega} \Big ) \Big \|_{L^p(K)}.
\]
Note that, using Theorem \ref{thm:L1} and H\"older's inequality with $p' = 1 + \varepsilon_0 > 1$ and $q' = (1 - \tfrac{1}{p'})^{-1}$, we have
\begin{align*}
\int_{\Omega}\left|S_{\chi} \mathbbm{1}_{\ell} - S_{\chi}^{(\delta)} \mathbbm{1}_{\ell} \right|\, \dd \mu_{\Omega} &= \int_{\Omega}\left| (S_{\chi} \mathbbm{1}_{\ell}) \, (1 - D_{\delta} ) \right|\, \dd \mu_{\Omega} \\
&\leq \left ( \int_\Omega (S_{\chi} \mathbbm{1}_{\ell})^{p'} \, \dd \mu_{\Omega} \right )^{\frac{1}{p'}} \mu_{\Omega} \left ( \{\lambda_\chi \leq \xi^{-1} \, \delta \}\right )^{\frac{1}{q'}} \\ &\ll_{\supp(\mathbbm{1}_{\ell}), \, p'}  \, \delta^{\beta_\chi d (1 - \frac{1}{1 + \varepsilon_0})}. 
\end{align*}
Using H\"older's inequality again with $p \in (1, 1 + \varepsilon_0/2)$, $s  = 1 + \varepsilon_0/2$,
and $q = (\tfrac{1}{p}- \tfrac{1}{s})^{-1}$, and Lemma \ref{prop:Non-Escape} with $y \geq \delta^{-\kappa}$, we thus have
\begin{align} \label{eq:EstimateTruncation}
\Big \| \Big ( S_{\chi} \mathbbm{1}_{\ell} \circ &a(y) - \int_\Omega S_{\chi} \mathbbm{1}_{\ell} \, \dd \mu_{\Omega} \Big ) - \Big ( S_{\chi}^{(\delta)} \mathbbm{1}_{\ell} \circ a(y) - \int_\Omega S_{\chi}^{(\delta)} \mathbbm{1}_{\ell} \, \dd \mu_{\Omega} \Big ) \Big \|_{L^p(K)} \nonumber \\
&\leq \left \|   S_{\chi} \mathbbm{1}_{\ell} \circ a(y)- S_{\chi}^{(\delta)} \mathbbm{1}_{\ell}\circ a(y)\right \|_{L^p(K)}+ \int_{\Omega}\left|S_{\chi} \mathbbm{1}_{\ell}-S_{\chi}^{(\delta)} \mathbbm{1}_{\ell} \right| \, \dd \mu_{\Omega}\nonumber \\
&\ll \left \| S_{\chi} \mathbbm{1}_{\ell} \circ a(y) \right \|_{L^s(K)} \, \mu_K \big (\{ k \in K  : \lambda_\chi(a(y) k \G) < \delta \} \big )^{1/q} + \delta^{\beta_\chi d (1 - \frac{1}{1+\varepsilon_0})}\nonumber \\
&\ll \delta^{\beta_\chi d ( \frac{1}{p} - \frac{1}{1 + \varepsilon_0/2})} .
\end{align}
Next, we approximate the $\delta$-truncated Siegel transform $S_{\chi}^{(\delta)} \mathbbm{1}_{\ell}$ by the $\delta$-truncated Siegel transform of the smooth compactly supported approximation function $f_{\varepsilon, \ell} = f_{\varepsilon, c_\ell}$ with $\varepsilon \in (0,1)$, as constructed in \eqref{eq:EpsilonApproximationChi}. That is, we would like to give an upper bound in terms of the truncation parameter $\delta$ and in terms of the approximation parameter $\varepsilon$ for
\[
\Big \|  \Big ( S_{\chi}^{(\delta)} \mathbbm{1}_{\ell}\circ a(y)-\int_{\Omega} S_{\chi}^{(\delta)} \mathbbm{1}_{\ell} \, \dd \mu_{\Omega} \Big ) - \Big ( S_{\chi}^{(\delta)} f_{\varepsilon, \ell} \circ a(y) - \int_{\Omega} S_{\chi}^{(\delta)} f_{\varepsilon, \ell} \, \dd \mu_{\Omega} \Big ) \Big \|_{L^p(K)}.
\]
By the mean value formula in Theorem \ref{thm:L1} and by the approximation properties of $f_{\varepsilon, \ell}$ (see Equation \eqref{eq:EpsilonApproximationChi}), we have
\begin{align*}
\int_{\Omega} \left| S_{\chi}^{(\delta)} f_{\varepsilon, \ell} - S_{\chi}^{(\delta)} \mathbbm{1}_{\ell} \right| \, \dd \mu_{\Omega} &\leq \int_{\Omega} \left| S_{\chi} f_{\varepsilon, \ell} - S_{\chi} \mathbbm{1}_{\ell} \right| \, \dd \mu_{\Omega} = \int_{\widetilde{X}} (f_{\varepsilon, \ell} - \mathbbm{1}_{\ell}) \, d \lambda_{\widetilde{X}} \ll \varepsilon.
\end{align*}
By the estimate in \eqref{eq:Replace_Schmidt} applied with $\mathbbm{1}_{\ell} -f_{\varepsilon, \ell}$, we have
\[
\forall \, g \in G, \quad | S_{\chi} (\mathbbm{1}_{\ell} -f_{\varepsilon, \ell})(g)| \ll_{\supp(\mathbbm{1}_{\ell} -f_{\varepsilon, \ell})} \, \lambda_\chi(g\G)^{-\beta_{\chi} d }.
\]
This together with the fact that $\supp(D_\delta) \subseteq \{x \in \Omega : \lambda_\chi(x) \geq \xi^{-1} \delta \}$, gives
\begin{align*}
\left \|S_{\chi}^{(\delta)} (\mathbbm{1}_{\ell} -f_{\varepsilon, \ell})\circ a(y) \right \|_{L^\infty(Y)}^{\frac{p-1}{p}} \ll \delta^{- \beta_\chi d  \frac{p-1}{p}}.
\end{align*}
Putting everything together, by Proposition \ref{prop:SmoothApproximationOrbit}, there exists $\varepsilon_1 > 0$ such that for all $y \in [(\beta_{\chi} \varepsilon)^{- \tfrac{\beta_{\chi}}{\varepsilon_1} }, + \infty)$, we have
\begin{align} \label{eq:Smooth}
\Big \|  &\Big ( S_{\chi}^{(\delta)} \mathbbm{1}_{\ell}\circ a(y)-\int_{\Omega} S_{\chi}^{(\delta)} \mathbbm{1}_{\ell} \, \dd \mu_{\Omega} \Big ) - \Big( S_{\chi}^{(\delta)} f_{\varepsilon, \ell} \circ a(y) - \int_{\Omega} S_{\chi}^{(\delta)} f_{\varepsilon, \ell} \, \dd \mu_{\Omega} \Big ) \Big \|_{L^p(K)} \nonumber \\
&\leq \left\| S_{\chi}^{(\delta)} \mathbbm{1}_{\ell}\circ a(y)- S_{\chi}^{(\delta)} f_{\varepsilon, \ell} \circ a(y) \right\|_{L^p(K)}+ \int_{\Omega}\left|S_{\chi}^{(\delta)} \mathbbm{1}_{\ell}- S_{\chi}^{(\delta)} f_{\varepsilon, \ell} \right| \, \dd \mu_{\Omega} \nonumber \\
&\leq \left \|(S_{\chi}^{(\delta)}(\mathbbm{1}_{\ell} -f_{\varepsilon, \ell}))\circ a(y) \right \|_{L^\infty(Y)}^{\frac{p-1}{p}}  \left \|(S_{\chi}^{(\delta)}(\mathbbm{1}_{\ell} -f_{\varepsilon, \ell})\circ a(y) \right \|_{L^1(Y)}^{\frac{1}{p}}+ \varepsilon \nonumber \\ 
&\ll \delta^{- \beta_\chi d \frac{p-1}{p}} \varepsilon^{\frac{1}{p}} + \varepsilon \, \ll \, \delta^{- \beta_\chi d  (1 -\frac{1}{p})} \varepsilon^{\frac{1}{p}}.
\end{align}
Finally, using Theorem~\ref{thm:Effective}, we would like to give an upper bound in terms of the rate of equidistribution on
\[
\int_K \left | \sum_{i = 1}^N \Big ( S_{\chi}^{(\delta)} f_{\varepsilon, \ell} (a(y_i) k \G)-\int_{\Omega}S_{\chi}^{(\delta)} f_{\varepsilon, \ell} \,  \dd \mu_{\Omega} \Big ) \right |^p \, \dd \mu_K(k).
\]
Note that, for every $f \in C_c^{\infty}(\widetilde{X})$, we have $S_{\chi}^{(\delta)} f \in C_c^{\infty}(\Omega)$. Let us now show that, for every $f \in C_c^{\infty}(\widetilde{X})$, we have
\begin{align} \label{eq:C_rNorm}
\forall \, r \in \N^*, \quad \cS_{r}(S_{\chi}^{(\delta)} f) \ll_{\mathrm{supp}(f)} \delta^{- \beta_\chi d } \cS_{r}(f).
\end{align}
First, for every $f \in C_c^{\infty}(\widetilde{X})$ and every differential operator $\cD$ as in \eqref{eq:Differential_Operator}, we note that $\cD ( S_{\chi} f) = S_{\chi}(\cD f)$. Then, by the point-wise upper bound estimate \eqref{eq:Replace_Schmidt} for the Siegel transform applied with $\cD f$, we have
\[
\forall \, g \in G, \quad | S_{\chi} (\cD f)(g \G)| \ll_{\supp(f)} \, \cS_{r}(f) \, \lambda_{\chi} (g\G)^{-\beta_{\chi} d }.
\]
Since $\supp(D_\delta) \subseteq \{x \in \Omega : \lambda_\chi(x) \geq \xi^{-1} \delta \}$ and $\cS_{r}(D_\delta) \ll 1$, we deduce \eqref{eq:C_rNorm}, as desired.

First, by the monotonicity of $L^p$-norms on probability spaces, 
\[
\Bigg \|  \sum_{i=1}^{N} \Big ( S_{\chi}^{(\delta)} f_{\varepsilon, \ell} \circ a(y_i) -\int_{\Omega}S_{\chi}^{(\delta)} f_{\varepsilon, \ell}\Big ) \Bigg \|_{L^p(K)} \leq \left \|  \sum_{i=1}^{N} \Big ( S_{\chi}^{(\delta)} f_{\varepsilon, \ell} \circ a(y_i) -\int_{\Omega}S_{\chi}^{(\delta)} f_{\varepsilon, \ell} \Big )\right \|_{L^2(K)}
\]
Expanding the square of the right-hand side gives
\begin{align*}
\sum_{i, j =1}^{N} &\int_{K} \Big(S_{\chi}^{(\delta)} f_{\varepsilon, \ell}(a(y_i)k \G) -\int_{\Omega}S_{\chi}^{(\delta)} f_{\varepsilon, \ell} \Big )  \Big (S_{\chi}^{(\delta)} f_{\varepsilon, \ell}(a(y_j)k \G) -\int_{\Omega} S_{\chi}^{(\delta)} f_{\varepsilon, \ell} \Big ) \dd\mu_K 
\\
&\leq \sum_{i, j =1}^{N} \Bigg | \int_{K} S_{\chi}^{(\delta)} f_{\varepsilon, \ell}(a(y_i)k \G) \, S_{\chi}^{(\delta)} f_{\varepsilon, \ell}(a(y_j)k \G) \dd\mu_K - \Big ( \int_{\Omega}S_{\chi}^{(\delta)} f_{\varepsilon, \ell} \Big )^2 \Bigg | \\
& \quad\quad\quad + \, 2 N \sum_{i =1}^{N} \Big ( \int_{\Omega} S_{\chi}^{(\delta)} f_{\varepsilon, \ell} \Big ) \Bigg | \int_{K} S_{\chi}^{(\delta)} f_{\varepsilon, \ell}(a(y_i)k \G) - \int_{\Omega}S_{\chi}^{(\delta)} f_{\varepsilon, \ell} \Bigg |.
\end{align*}
Note that the sum $\sum_{i,j = 1}^N$ is symmetric in $i$ and $j$, and we split it into the diagonal and the off-diagonal part. Recall that $y_i = e^{\beta_{\chi} i}$. By Theorem~\ref{thm:Effective} on the effective single and double equidistribution of translated $K$-orbits, there exists a constant $c>0$ such that this is bounded by an absolute constant times
\[
2 \sum_{1 \leq i < j \leq N} \, \min \{y_i, y_j / y_i \}^{-c} \cS_r \big ( S_{\chi}^{(\delta)} f_{\varepsilon, \ell} \big )^2 \, + \, \sum_{i =1}^{N} \cS_r \big ( S_{\chi}^{(\delta)} f_{\varepsilon, \ell} \big )^2 \, + \, 2 N \sum_{i = 1}^N y_i^{-c} \cS_r \big ( S_{\chi}^{(\delta)} f_{\varepsilon, \ell} \big ),
\]
where we used also that $\int_{\Omega} S_{\chi}^{(\delta)} f_{\varepsilon, \ell} \ll 1$. Noting that $\sum_{i = 1}^{\infty} y_i^{-c} < + \infty$ and using the estimate \eqref{eq:C_rNorm} for the Sobolev norm, the inequality $\cS_r(f_{\varepsilon, \ell}) \ll \varepsilon^{-r}$ as in \eqref{eq:EpsilonApproximationChi}, this is essentially bounded by
\begin{equation*} 
2 \sum_{1 \leq i < j \leq N} \, \min \{y_i, y_j / y_i \}^{-c} \delta^{- 2 \beta_\chi d } \varepsilon^{- 2 r} \, + \, N \delta^{- 2 \beta_\chi d } \varepsilon^{- 2 r} \, + \, 2 N \delta^{- \beta_\chi d } \varepsilon^{-r}.
\end{equation*}
Since $\delta, \varepsilon \in (0,1)$, we have $\delta^{- 2\beta_\chi d } \varepsilon^{-2r} \geq \delta^{- \beta_\chi d } \varepsilon^{-r}$. Let us show that 
\begin{equation} \label{eq:}
\sum_{1 \leq i < j \leq N} \, \min \{y_i, y_j / y_i \}^{-c} \, \ll \, N.
\end{equation}
We split the sum into the part where $y_i \leq y_j / y_i$ and the remainder. We have $y_i \leq y_j / y_i$ if and only if $2i \leq j$ and hence
\[
\sum_{\substack{1 \leq i < j \leq N \\ 2i \leq j}} y_i^{-c} \, \leq \sum_{i=1}^{\lceil N/2 \rceil} \sum_{j=2i}^{N} y_i^{-c} =  \sum_{i=1}^{\lceil N/2 \rceil} (N- 2i +1) y_i^{-c} \, \ll \, N.
\]
For the remaining part, using a change of variable and the fact that $y_j/y_i = y_{j-i}$,
\[
\sum_{\substack{1 \leq i < j \leq N \\ 2i > j}} (y_j / y_i)^{-c} \, \leq \, \sum_{1 \leq i < j \leq N} (y_j / y_i)^{-c} \leq \sum_{k = 1}^{N-1} \sum_{\ell = 1}^{N-k} y_k^{-c} \, \ll \, N.
\]
Hence, putting all the estimates together, we have shown that 
\begin{equation} \label{eq:EstimateEquidistribution}
\Bigg \|  \sum_{i=1}^{N} \Big ( S_{\chi}^{(\delta)} f_{\varepsilon, \ell}\circ a(y_i) -\int_{\Omega}S_{\chi}^{(\delta)} f_{\varepsilon, \ell}\Big ) \Bigg \|_{L^p(K)} \ll \delta^{- \beta_{\chi} d} \, \varepsilon^{-r} \, N^{1/2}.
\end{equation}
Let $\kappa$ be as in Proposition \ref{prop:Non-Escape} and $\varepsilon_1$ as in Proposition \ref{prop:SmoothApproximationOrbit}. We wish to choose suitable $\lambda_1, \lambda_2 > 0$ and put $\delta = N^{-\lambda_1}$ and $\varepsilon = N^{-\lambda_2}$. Let us determine them by a heuristic argument first. So assume for now that Propositions \ref{prop:Non-Escape} and \ref{prop:SmoothApproximationOrbit} hold for all $y \geq 1$. Then, using Minkowski's inequality and combining the estimates \eqref{eq:EstimateTruncation}, \eqref{eq:Smooth}, and \eqref{eq:EstimateEquidistribution}, we have
\begin{multline} \label{eq:AltogetherBound} 
\left \| \sum_{i=1}^{N} \left( S_{\chi} \mathbbm{1}_{\ell} \circ a(y_i) -\int_{\Omega}S_{\chi} \mathbbm{1}_{\ell} \, \dd \mu_{\Omega} \right) \right \|_{L^p(K)} \\ \ll  N(\delta^{\beta_\chi d ( \frac{1}{p} - \frac{1}{1 + \varepsilon_0/2})} + \delta^{- \beta_\chi d (1 - \frac{1}{p})} \varepsilon^{\frac{1}{p}}) + N^{1/2}\delta^{- \beta_\chi d }\varepsilon^{-r}.
\end{multline} 
By setting 
\begin{equation} \label{eq:Choice}
\delta = N^{-\frac{1}{2 \beta_\chi d (\frac{1}{p} - \frac{1}{1+ \varepsilon_0/2} + 1 + rp(1-\frac{1}{1+ \varepsilon_0/2}))}} \quad \text{ and } \quad \varepsilon = \delta^{\beta_\chi d p (1 - \frac{1}{1+\varepsilon_0/2})}
\end{equation}
the exponents of the three terms on the right-hand side in \eqref{eq:AltogetherBound} match.
Hence we let 
\[
\lambda_1 = \frac{1}{2 \beta_\chi d (\frac{1}{p} - \frac{1}{1+ \varepsilon_0/2} + 1 + rp(1-\frac{1}{1+ \varepsilon_0/2}))}
\]
and
\[
\lambda_2 = \frac{p (1 - \frac{1}{1+\varepsilon_0/2})}{2 (\frac{1}{p} - \frac{1}{1+ \varepsilon_0/2} + 1 + rp(1-\frac{1}{1+ \varepsilon_0/2}))}.
\]
Now having fixed $\lambda_1$ and $\lambda_2$, let us show that (up to a multiplicative constant) \eqref{eq:AltogetherBound} holds true.
We recall that, by Lemma \ref{lem:Replacement_for_alpha_s}, there exists an absolute constant $C_0 \geq 1$ such that, for every $c \in [1/2,2/3]$, we have a uniform upper bound
\[ 
\sup_{y \geq 1} \int_{K} \big | S_{\chi} \mathbbm{1}_{\cF_c} ( a(y) k \G) \big |^p \, \dd \mu_K(k) \, \leq \, C_0. 
\]
We apply this estimate whenever $y_i = e^{\beta_{\chi} i}$ is too small to satisfy the assumptions of Propositions \ref{prop:Non-Escape} or \ref{prop:SmoothApproximationOrbit}, that is, by taking the natural logarithm, if $i$ is smaller than a constant $C_1 > 0$ (that depends on $\lambda_1$, $\lambda_2$, $\kappa$, $\varepsilon_1$, $\beta_{\chi}$) times $\ln(N)$. Using Minkowski's inequality and combining \eqref{eq:EstimateTruncation}, \eqref{eq:Smooth}, and \eqref{eq:EstimateEquidistribution}, we have
\begin{align*}
&\left \| \sum_{i=1}^{N} \left( S_{\chi} \mathbbm{1}_{\ell} \circ a(y_i) - \int_{\Omega}S_{\chi} \mathbbm{1}_{\ell}\right)\right \|_{L^p(K)} \ll \left ( \sum_{\substack{i= 1\\ i \geq C_1 \ln(N)}}^{N} \delta^{\beta_\chi d ( \frac{1}{p} - \frac{1}{1 + \varepsilon_0/2})} \right ) + \left ( \sum_{\substack{i= 1\\ i \leq C_1 \ln(N)}}^{N} C_0 \right ) \\
&\quad \quad + \left ( \sum_{\substack{i= 1\\ i \geq C_1 \ln(N)}}^{N} \delta^{- \beta_\chi d  (1 -\frac{1}{p})} \varepsilon^{\frac{1}{p}}\right ) + \left ( \sum_{\substack{i= 1\\ i < C_1 \ln(N)}}^{N} C_0 \right ) + \delta^{- \beta_{\chi} d} \, \varepsilon^{-r} \, N^{1/2} \\
&\quad \ll  N(\delta^{\beta_\chi d ( \frac{1}{p} - \frac{1}{1 + \varepsilon_0/2})} + \delta^{- \beta_\chi d (1 - \frac{1}{p})} \varepsilon^{\frac{1}{p}}) + N^{1/2}\delta^{- \beta_\chi d }\varepsilon^{-r}.
\end{align*}
Plugging in the formulas for $\delta$ and $\varepsilon$, we have
\begin{equation} \label{eq:AltogetherBound-2}
\left \| \sum_{i=1}^{N} \left( S_{\chi} \mathbbm{1}_{\ell} \circ a(y_i) -\int_{\Omega}S_{\chi} \mathbbm{1}_{\ell}\right) \right \|_{L^p(K)} \ll  N^{1 -\frac{ (\frac{1}{p} - \frac{1}{1+\varepsilon_0/2})}{2 (\frac{1}{p} - \frac{1}{1+ \varepsilon_0/2} + 1 + rp(1-\frac{1}{1+ \varepsilon_0/2}))}}.
\end{equation}
Consider the function $h : (1,1+\varepsilon_0/2) \rightarrow \R$ given by
\[
h(p) = 1 -\frac{ (\frac{1}{p} - \frac{1}{1+\varepsilon_0/2})}{2 (\frac{1}{p} - \frac{1}{1+ \varepsilon_0/2} + 1 + rp(1-\frac{1}{1+ \varepsilon_0/2}))}. 
\]
This function satisfies $\lim_{p \rightarrow 1} h(p) < 1$. 
Let $p \in (1,1+\varepsilon_0/3]$ be maximal with the property that $h(p) \leq \frac{1}{p}$. Then we have
\[
\left \| \sum_{i=1}^{N} \left( S_{\chi} \mathbbm{1}_{\ell}-\int_{\Omega}S_{\chi} \mathbbm{1}_{\ell}\right)\circ a(y_i) \right \|_{L^p(K)}^p \ll N.
\]
Let $\delta' > 0$ and, for every $N \in \N^*$, let $\ell_N = \lfloor N^{\delta'} \rfloor$. By Lemma \ref{lem:Weyl-Metric}, for every $\varepsilon'>0$ and for almost every $k \in K$, we have
\[
\forall \, N \in \N^*, \quad \sum_{i=0}^{N-1} S_{\chi} \mathbbm{1}_{\cF_{\widehat{c}_{\ell}^{-1}}} (a(y_i) k \G) = \int_{\Omega} S_{\chi} \mathbbm{1}_{\cF_{c_\ell^{-1}}} \, \dd \mu_{\Omega} \, N + O\left( N^{\frac{2}{p+1} + \varepsilon'}\right ).
\]
By the mean value formula in Theorem \ref{thm:L1}, we have
\[
\int_{\Omega} S_{\chi} \mathbbm{1}_{\cF_{c_\ell^{-1}}} \, \dd \mu_{\Omega} = \lambda_{\widetilde{X}}\left (\cF_{c_\ell^{-1}} \right ).
\]
There exists a constant $\kappa' > 0$ such that $\lambda_{\widetilde{X}}\left (\cF_{c_\ell^{-1}} \right ) = \lambda_{\widetilde{X}}\left (\cF_{1} \right ) + O(\ell^{-\kappa'})$. Plugging this into the inequality \eqref{eq:Sandwich_lem:Reduction_Birkhoof_2}, we have, for all $N \in \N^*$,
\begin{multline*}
\sum_{i=0}^{N-1} S_{\chi} \mathbbm{1}_{\cF_{\widehat{c}_{\ell}^{-1}}} (a(y_i) k \G) + C_0 \, \ell_N^{\beta_{\chi} d} \\
= \lambda_{\widetilde{X}}\left (\cF_{1} \right ) \, N + O(\lfloor N^{\delta'} \rfloor^{-\kappa'}) + O\left( N^{\frac{2}{p+1} + \varepsilon'}\right ) + O \left ( \lfloor N^{\delta'} \rfloor^{\beta_{\chi} d} \right ).
\end{multline*}
This completes the proof of Proposition \ref{prop:Upper_Bound_on_Lattice_Count} with 
\[
\delta' = \frac{1}{\beta_{\chi} d} \frac{2}{p+1}, \quad  \varepsilon = 1 - \frac{2}{p+1}, \quad \text{and} \quad \varkappa_1 = \lambda_{\widetilde{X}}(\cF).
\]
\end{proof}

\bibliographystyle{plain}

\end{document}